\documentclass[11pt,oneside]{amsart}

\usepackage{amsmath,ifthen, amsfonts, amssymb,
srcltx, amsopn, color, enumerate} 
\usepackage[linktocpage=true]{hyperref}
\usepackage[cmtip,arrow]{xy}
\usepackage{pb-diagram, pb-xy}
\usepackage{overpic,mathabx}

\usepackage[T1]{fontenc}

\newcommand{\showcomments}{yes}

\newsavebox{\commentbox}

\newcounter{ax}
\newtheorem{thm}{Theorem}[section]

\newtheorem{lem}[thm]{Lemma}

\newtheorem{cor}[thm]{Corollary}

\newtheorem{prop}[thm]{Proposition}

\newtheorem{thmi}{Theorem}
\newtheorem{cori}[thmi]{Corollary}

\newtheorem{propi}[thmi]{Proposition}

\theoremstyle{definition}
\newtheorem{defn}[thm]{Definition}
\newtheorem{rem}[thm]{Remark}
\newtheorem*{remi}{Remark}
\newtheorem{exmp}[thm]{Example}

\newtheorem{notation}[thm]{Notation}

\newtheorem{claim}{Claim}

\newtheorem{claim*}{Claim}

\newtheorem*{remnon}{Remark}

\newtheorem{question}[thm]{Question}

\DeclareMathOperator{\image}{im}

\DeclareMathOperator{\Aut}{Aut}

\DeclareMathOperator{\link}{Lk}

\DeclareMathOperator{\stabilizer}{Stab}
\DeclareMathOperator{\diam}{diam}

\newcommand{\neb}{\mathcal N}
\def\MCG{\mathcal{MCG}}

\newcommand{\homology}{\ensuremath{{\sf{H}}}}

\newcommand{\coll}{\;\;\makebox[0pt]{$\bot$}\makebox[0pt]{$\smile$}\;\;}

\newcommand{\field}[1]{\mathbb{#1}}
\newcommand{\integers}{\ensuremath{\field{Z}}}

\newcommand{\naturals}{\ensuremath{\field{N}}}
\newcommand{\reals}{\ensuremath{\field{R}}}

\newcommand{\closure}[1]{Cl\left({#1}\right)}
\newcommand{\propnest}{\sqsubsetneq}

\makeatletter

\newcommand{\Rmnum}[1]{\mathbf{{\expandafter\@slowromancap\romannumeral #1@}}}

\newcommand{\contact}[1]{\ensuremath{\mathcal C#1}}
\newcommand{\crossing}[1]{\ensuremath{\mathcal C_{\sharp}#1}}
\newcommand{\concom}[1]{\ensuremath{\mathcal C_\bullet#1}}

\makeatother

\newcommand{\tup}[1]{\vec{#1}}

\let\oldmarginpar\marginpar
\renewcommand\marginpar[1]{\-\oldmarginpar[\raggedleft\footnotesize #1]{\raggedright\footnotesize #1}}

\newcommand{\factorsup}{{\mathfrak F}}
\newcommand{\factorseq}{\overline{{\mathfrak F}}}

\newcommand{\tsh}[1]{\left\{\kern-.9ex\left\{#1\right\}\kern-.9ex\right\}}
\newcommand{\Tsh}[2]{\tsh{#2}_{#1}}
\newcommand{\ignore}[2]{\Tsh{#2}{#1}}

\newcommand{\co}{\colon}

\newcounter{enumitemp}

\newcommand{\dist}{\textup{\textsf{d}}}

\newcommand{\OL}{\overleftarrow}

\newcommand{\cuco}[1]{{\mathcal #1}}
\newcommand{\I}{\mathbf I}
\newcommand{\II}{\mathbf{II}}

\newcommand{\fontact}{\widehat{\mathcal C}}

\newcommand{\gate}{\mathfrak g}

\newcommand{\bbf}{\mathfrak C}
\newcommand{\subp}[2]{\rho_{#1}^{#2}}
\newcommand{\seq}[1]{\mbox{\boldmath$#1$}}
\newcommand{\seqcuco}{\seq{\cuco{X}}}
\newcommand{\subseq}[1]{\mbox{\boldmath$\scriptstyle#1$}}
\def\ulim{\lim_\omega}

\newcommand{\Edges}{\mathbf{Edges}}
\newcommand{\Vertices}{\mathbf{Vertices}}

\newcommand{\nest}{\sqsubseteq}
\newcommand{\orth}{\bot}
\newcommand{\transverse}{\pitchfork}

\begin{document}
\title[Hierarchically hyperbolic spaces I: curve complexes for 
cubical groups]{Hierarchically hyperbolic spaces I:\\ curve complexes for cubical groups}

\author[J. Behrstock]{Jason Behrstock}
\address{Lehman College and The Graduate Center, CUNY, New York, New York, USA}
\curraddr{Barnard College, Columbia University, New York, New York, USA}
\email{jason@math.columbia.edu}
\thanks{\flushleft {Behrstock was supported as an Alfred P. Sloan
Fellow, a Simons Fellowship, 
and by the National Science Foundation under Grant Number NSF
1006219.}}

\author[M.F. Hagen]{Mark F. Hagen}
\address{U. Michigan, Ann Arbor, Michigan, USA}
\email{markfhagen@gmail.com}
\thanks{\flushleft {Hagen was supported by the National Science Foundation under Grant Number NSF 1045119.}}

\author[A. Sisto]{Alessandro Sisto}
\address{ETH, Z\"{u}rich, Switzerland}
\email{sisto@math.ethz.ch}
\thanks{\flushleft{Sisto was supported by the Swiss National Science Foundation project 144373.}}

\maketitle

\begin{abstract}
In the context of CAT(0) cubical groups, 
we develop an analogue of the theory of curve complexes and subsurface
projections.  The role of the subsurfaces is
played by a collection of convex subcomplexes called a \emph{factor
system}, and the role of the curve graph is played by the
\emph{contact graph}.  There are a number of close parallels between the contact graph and the curve graph, including
hyperbolicity, acylindricity of the action, the existence of hierarchy
paths, and a Masur--Minsky-style distance formula.  

We then define a
\emph{hierarchically hyperbolic space}; the class of such spaces
includes a wide class of cubical groups (including all virtually
compact special groups) as well as mapping class groups and
Teichm\"{u}ller space with any of the standard metrics.  
We deduce a number of results about these spaces, all of which are 
new for cubical or mapping class groups, and most of which are new 
for both.  We show that the quasi-Lipschitz image from a ball in a 
nilpotent Lie group into a hierarchically hyperbolic space lies close to a product of hierarchy geodesics.  We also prove a rank theorem for
hierarchically hyperbolic spaces; this generalizes results of Behrstock--Minsky, Eskin--Masur--Rafi, Hamenst\"{a}dt, and Kleiner.  We finally prove
that each hierarchically hyperbolic group admits an acylindrical
action on a hyperbolic space.  This acylindricity result is new for cubical
groups, in which case the hyperbolic space admitting the action 
is the contact graph; in the case of the mapping class group,
this provides a new proof of a theorem of Bowditch.
\end{abstract}

\tableofcontents

\section{Introduction}\label{sec:introduction}
Cube complexes and groups that act on them are fundamental objects in 
geometric group theory.  Examples of cubical groups --- groups acting
geometrically on CAT(0) cube complexes --- are right-angled (and many
other) Coxeter groups~\cite{Davis:book,NibloReeves:coxeter},
right-angled Artin groups~\cite{CharneyDavis:salvetti_cubes}, and,
more generally, graph products of abelian
groups~\cite{RuaneWitzel:graph_prod}.  Other examples of cubical
groups include: groups satisfying sufficiently strong
small-cancellation conditions~\cite{Wise:small_cancel_cube}; many
$3$--manifold groups, including all hyperbolic
ones~\cite{KM09,BergeronWise,Wise:quasiconvex_hierarchy} and some
graph manifold groups~\cite{HagenPrzytycki:graph}; hyperbolic
free-by-$\integers$ groups~\cite{HagenWise:freebyz}; etc.  Despite the 
attention cubical groups have attracted in recent years,
their large-scale geometry has remained rather opaque, with a few
salient exceptions, notably the resolution of the Rank Rigidity
Conjecture \cite{CapraceSageev:rank_rigidity}, characterizations of linear
divergence, relative hyperbolicity, and thickness in terms of
combinatorial
data~\cite{BehrstockCharney,HagenBoundary,BehrstockHagen:cubulated1}, 
analysis of quasiflats 
in the $2$--dimensional case~\cite{BKS:quasiflatsCAT0} and 
top-dimensional quasiflats in general \cite{Huang:quasiflats}. 

Recently, there has been enormous progress in understanding 
the mapping class group and related spaces. Highlights have 
included resolutions of the Ending Lamination Conjecture
\cite{BrockCanaryMinsky:ELC2}, 
the Rank Conjecture \cite{BehrstockMinsky:dimension_rank}, 
quasi-isometric rigidity \cite{BKMM:consistency}, finite 
asymptotic dimension \cite{BBF:quasi_tree}, and a number of others. 
Proofs of each of these results have featured the curve complex in a 
central position.

Motivated by the vital role the curve complex has played in 
unveiling the geometry of the mapping class group, 
in this work we develop analogues of those tools in the context of 
cubical groups. In particular, for cubical groups we develop versions of the machinery of curve complex 
projections and hierarchies initiated by Masur--Minsky in 
\cite{MasurMinsky:I,MasurMinsky:II} as well as subsequent tools 
including the consistency and realization theorems 
\cite{Behrstock:asymptotic, BKMM:consistency}. 
We note that right-angled Artin groups are a particularly interesting class of 
groups to which the tools we develop can be applied. 
Finally, we define  
\emph{hierarchically hyperbolic spaces}, which provide a 
framework that encompasses mapping class groups, Teichm\"{u}ller
space, and a large class of cubical groups including the
class of \emph{compact special} groups of
Haglund--Wise~\cite{HaglundWiseSpecial}.  
This allows us to prove new results in both the mapping class group and
cubical contexts.

\subsection{Geometry of contact graphs}

In Part~\ref{part:general}, we develop a number of basic aspects 
of the geometry of \emph{contact graphs}, extending a study which was 
initiated in \cite{Hagen:quasi_arboreal}.  The contact graph
$\contact\cuco X$ of the CAT(0) cube complex $\cuco X$ is the
intersection graph of the \emph{hyperplane carriers}; in other words,
there is a vertex for each hyperplane of $\cuco X$, and two vertices
are adjacent if the corresponding hyperplanes are not separated by a
third.  Since the contact graph is always hyperbolic (in fact, a
quasi-tree)~\cite{Hagen:quasi_arboreal}, it is a natural candidate
for a ``curve complex'' in the context of cubical groups.  The main results
of Part~\ref{part:general} are summarized below.  Recall that the WPD
property, as defined in \cite{BestvinaFujiwara:boundedcohom}, is a
form of properness of an action ``in the direction'' of particular 
elements; see Section \ref{sec:WPD} for the precise definition.  It has
important applications to bounded cohomology (see e.g., \cite{BBF2}), 
is closely related to the Bestvina-Bromberg-Fujiwara construction
\cite{BBF:quasi_tree}, and provides an equivalent
characterization of acylindrical hyperbolicity \cite{Osin:acyl} which 
is used in a number of applications.  

\begin{thmi}\label{thmi:geom_cg}
 Let $\cuco X$ be a CAT(0) cube complex and $\contact \cuco X$ its contact graph. Then:
\begin{enumerate}
 \item\label{itemi:wpd} {\rm (WPD property.)} 
 If $\cuco X$ is uniformly locally finite and  $g\in
 \Aut(\cuco X)$ is rank-one with the property that no positive power of $g$ stabilizes a 
 hyperplane, then $g$ 
 acts as a loxodromic WPD element on $\contact \cuco X$.
 \item\label{itemi:hierarchy} {\rm(Hierarchy paths.)} Let $x,y\in\cuco{X}$ be 0-cubes.  Then there exist hyperplanes $H_0,\ldots,H_k$ with $x\in \neb(H_0),y\in \neb(H_k)$ and combinatorial geodesics $\gamma_i\rightarrow \neb(H_i)$ such that $H_0,H_1,\ldots,H_k$ is a geodesic of $\contact{\cuco{X}}$ and $\gamma_0\gamma_1\cdots\gamma_k$ is a geodesic joining $x,y$.
 \item\label{itemi:contr} {\rm(Contractibility.)} If the set of $1$--cubes of $\cuco X$ is countable and non-empty, then the flag complex spanned by $\contact \cuco X$ is contractible.
\end{enumerate}
\end{thmi}

In the case of mapping class groups, analogues of 
Theorem~\ref{thmi:geom_cg}.\eqref{itemi:wpd} and
Theorem~\ref{thmi:geom_cg}.\eqref{itemi:hierarchy} were proved in  
\cite{BestvinaFujiwara:boundedcohom} and \cite{MasurMinsky:II}. The curve complex is not 
contractible, so Theorem~\ref{thmi:geom_cg}.\eqref{itemi:contr} 
provides a way in which the contact graph is simpler than the 
curve complex.

Theorem~\ref{thmi:geom_cg} has applications to random walks. 
In particular, from 
Theorem \ref{thmi:geom_cg}.\eqref{itemi:wpd},\eqref{itemi:hierarchy} 
and the main result of \cite{Sisto:random_walks_wpd}, when the
non-elementary group $G<\Aut(\cuco X)$ contains a rank-one
element, random
paths in $G$ stay close to geodesics with high probability. 
Further, this property has applications to various parameters associated with
the random walk, including rate of escape and entropy.

\subsection{Factor systems}
The mapping class group, $\MCG(S)$, of a surface $S$ is associated 
 with the curve complex of $S$, together  
 with the collection of curve complexes of subsurfaces of $S$; this association underlies the hierarchy machinery of~\cite{MasurMinsky:II}. 
Analogously, a CAT(0) cube complex $\cuco X$ contains a profusion of
convex subcomplexes, each of which is itself a CAT(0) cube complex,
and the judicious choice of a family of convex subcomplexes enables
the creation of hierarchy machinery.  The role of the collection of
subsurfaces is played by a \emph{factor system} $\factorsup$ in $\cuco
X$, which is a particular uniformly locally finite collection of convex
subcomplexes. 
Table~\ref{table:comparison} summarizes the analogy between the
mapping class group and a CAT(0) cube complex
with a geometric group action and a factor system. We emphasize that 
although the tools and the results we obtain have parallels for 
mapping class groups the techniques that we employ are very 
different.

\begin{center}
\scriptsize
\begin{table}[h]
\begin{tabular}{|p{6.5cm}| p{7.5cm}|}
 \hline \textbf{Mapping class group $\MCG(S)$} & \textbf{Cube complex 
 $\cuco X$ with factor system $\factorsup$ and $G$-action} 
  \\\hline \hline
 Curve complex $\contact S$ is hyperbolic~\cite{MasurMinsky:I} & Contact 
 graph $\contact\cuco X$ and factored contact graph $\fontact \cuco 
 X$ are hyperbolic, indeed, are quasi-trees (Thm.~\ref{thm:hyperbolicity}, Prop.~\ref{prop:fontact_quasi_tree}) \\\hline $\MCG(S)$ acts on $\contact S$ acylindrically~\cite{Bowditch:tight} & $G$ acts on $\contact\cuco X,\fontact \cuco X$ acylindrically (Cor.~\ref{cor:cubical_acyl}) \\\hline Nielsen-Thurston classification~\cite{Thurston:diffeos} & Loxodromic/reducible/generalized reducible (Thm. \ref{thm:NTC}) \\\hline $\exists$ quasi-geodesics in $\MCG(S)$ shadowing geodesics in 
 $\contact S$~\cite{MasurMinsky:II} & $\exists$ geodesics in $\cuco 
 X$ shadowing geodesics in $\contact\cuco X,\fontact \cuco X$ 
 (Prop.~\ref{prop:hierpath}, Prop.~\ref{prop:hier_revisited})\\\hline Subsurfaces & Subcomplexes in $\factorsup$ (Defn.~\ref{defn:factor_system}) \\\hline Projections to subsurfaces ~\cite{MasurMinsky:II} & Projection to $\fontact F$ for $F\in\factorsup$ (Sec.~\ref{subsec:factored_contact_graphs}) \\\hline Formula computing distance in $\MCG(S)$ in terms of curve complex distances~\cite{MasurMinsky:II} & 
 Formula computing distance in  $\cuco X$ in terms of factored 
 contact graph distances (Thm.~\ref{thm:distance_formula}) \\\hline Bounded Geodesic Image~\cite{MasurMinsky:II} & Bounded Geodesic Image, Prop. \ref{prop:BGIIITF} \\\hline Nested, disjoint, overlapping subsurfaces & Parallel into, orthogonal, transverse elements of $\factorsup$ \\\hline  Large Link Lemma~\cite{MasurMinsky:II} & Large Link Lemma, Prop. \ref{prop:bgiII} \\\hline Consistency and realization~\cite{Behrstock:asymptotic},\cite{BKMM:consistency} & Consistency and realization (Thm.~\ref{thm:consistency_realization}) \\\hline\end{tabular}
\smallskip
\caption{}\label{table:comparison}
\end{table}
\normalsize
\end{center}

The collection of subcomplexes which constitute a 
factor system $\factorsup$ in $\cuco
X$ includes $\cuco X$, as well as all combinatorial
hyperplanes of $\cuco X$, further, this collection 
is closed under the following operation: if
$F,F'\in\factorsup$ and $F$ has (combinatorial) projection onto $F'$ of
diameter more than some specified threshold, then the projection of
$F$ onto $F'$ lies in $\factorsup$.  This implies that sufficiently
large hyperplanes of any codimension belong to any factor system, and
indeed each factor system $\factorsup$ contains a minimal factor
system consisting of $\cuco X$, all combinatorial hyperplanes, and
the closure of this family under the above projection.  This
minimal factor system is $\Aut(\cuco
X)$--invariant since each automorphism preserves the set of
hyperplanes.

The reader should have in mind the following example, which already 
shows that the class of groups $G$ acting geometrically on cube 
complexes with $G$--invariant factor systems is very large.  Let
$\Gamma$ be a finite simplicial graph and let $\widetilde S_\Gamma$ be
the universal cover of the Salvetti complex of the corresponding
right-angled Artin group $A_\Gamma$, so that $\widetilde S_\Gamma$ is
a CAT(0) cube complex on which $A_\Gamma$ acts properly with a single
orbit of $0$--cubes~\cite{CharneyDavis:salvetti_cubes}.  Each induced
subgraph $\Lambda$ of $\Gamma$ yields a monomorphism
$A_\Lambda\rightarrow A_\Gamma$ and an $A_\Lambda$--equivariant
embedding $\widetilde S_\Lambda\hookrightarrow\widetilde S_\Gamma$.
The set of all such subcomplexes of $\widetilde S_\Gamma$, and all of
their $A_\Gamma$--translates, forms a factor system for $\widetilde
S_\Gamma$, described in detail in Section~\ref{subsec:examples_of_fs}.
This, and the fact that the existence of a factor system is inherited
by convex subcomplexes
(Lemma~\ref{lem:convex_subcomplex_factor_system}), enables the study
of groups that are \emph{virtually special} in the sense of
Haglund--Wise~\cite{HaglundWiseSpecial} using factor systems:

\begin{propi}\label{propi:special_factor}
Let $\overline{\cuco X}$ be a special cube complex with finitely many hyperplanes.  Then the universal cover $\cuco X$ of $\overline{\cuco X}$ contains a factor system, and hence contains a factor system that is invariant under the action of $\pi_1\overline{\cuco X}$.
\end{propi}

In Corollary \ref{cor:compat_special_factor_system}, for special cube 
complexes, we 
describe the factor system explicitly in terms of the 
hyperplanes of $\overline{\cuco X}$. 
Proposition~\ref{propi:special_factor} also enables one
to study many cubical groups which are far from being special: in
Section~\ref{subsec:coloring}, using 
Proposition~\ref{prop:raag_factor_system} together with 
Burger--Mozes \cite{BurgerMozes} and Wise \cite{Wise:CSC} we show 
there exists 
many non-virtually special groups $G$ which act geometrically on a CAT(0) 
cube complex $\cuco X$ with a factor system. Moreover, we produce 
many examples which, unlike those of  Burger--Mozes and Wise, do not 
admit equivariant embeddings into products of trees; these will be 
used in Section~\ref{subsec:coloring}.

Each $F\in\factorsup$ is a convex subcomplex, and is thus a CAT(0)
cube complex whose hyperplanes have the form $H\cap F$, where $H$ is a
hyperplane of $\cuco X$.  This gives a natural injective graph
homomorphism $\contact F\rightarrow\contact\cuco X$, whose image is an
induced subgraph~\cite{Hagen:quasi_arboreal}.  Just as the elements of
the factor system stand in analogy to the subsurfaces, the 
graphs $\contact F$, where $F\in\factorsup$, essentially play the role of the curve
complexes of the subsurfaces.  In order to obtain 
Theorem~\ref{thm:distance_formula} --- our analogue of the Masur--Minsky
distance formula --- we must modify each $\contact F$ slightly, by
coning off each subgraph which is the contact graph of some smaller
element of $\factorsup$.  It is the resulting \emph{factored contact graphs}
$\fontact F$ that actually play the role of curve complexes.
In Section~\ref{subsec:factored_contact_graphs_are_quasi_trees}, 
we show that factored
contact graphs are all quasi-trees. Moreover, when $\factorsup$ is
the minimal factor system described above, then $\fontact\cuco X$ and
$\contact\cuco X$ are quasi-isometric.

In Section~\ref{sec:distance_formula}, we prove the following 
analogue for cubical groups 
of the celebrated Masur--Minsky distance formula 
\cite[Theorem~6.12]{MasurMinsky:II}. Their  
formula has become an essential tool in studying the geometry of the 
mapping class group. Later, we will take the existence of 
such a 
formula as one of the characteristic features of a hierarchically hyperbolic
space. 

\begin{thmi}[Distance formula]\label{thmi:distance_formula_intro}
Let $\cuco X$ be a CAT(0) cube complex with a
factor system~$\factorsup$.  Let $\factorseq$ contain exactly one
representative of each parallelism class in $\factorsup$.  Then there
exists $s_0\geq 0$ such that for all $s\geq s_0$, there are constants
$K\geq 1,C\geq 0$ such that for all $x,y\in\cuco X^{(0)}$,
$$\dist_{\cuco
X}(x,y)\asymp_{_{K,C}}\sum_{F\in\factorseq}\ignore{\dist_{\fontact
F}(\pi_F(x),\pi_F(y))}{s}.$$
\end{thmi}

\noindent(Here, $\ignore{A}{s}=A$ if $A\geq s$ and $0$ otherwise. The 
notation $\asymp_{_{K,C}}$ means ``up to bounded multiplicative and 
additive error''.)

In Theorem~\ref{thmi:distance_formula_intro}, we use the notion of
\emph{parallelism}: two convex subcomplexes $F,F'$ of $\cuco X$ are
\emph{parallel} if for all hyperplanes $H$, we have $H\cap
F\neq\emptyset$ if and only if $H\cap F'\neq\emptyset$.  Equivalently,
$F,F'$ are parallel if and only if $\contact F,\contact F'$ are the
same subgraph of $\contact\cuco X$; parallel subcomplexes are
isomorphic.  Just as the Masur--Minsky distance formula involves
summing over all curve complexes of subsurfaces, by identifying
parallel elements of the factor system, our sum is over all
factored contact graphs without repetition.

Another important property of the curve complex $\contact S$ is that
the action of $\MCG(S)$ on $\contact S$ is \emph{acylindrical}, by a
result of Bowditch~\cite{Bowditch:tight}.  We obtain an analogous
result for actions on (factored) contact graphs arising from actions
on cube complexes, and, in Section~\ref{sec:acyl} we will show this holds in considerably greater
generality . The statement in the cubical case is:

\begin{thmi}[Acylindrical hyperbolicity from factor systems]\label{thmi:acyl}
Let the group $G$ act properly and cocompactly on the CAT(0) cube complex $\cuco X$ and suppose that $\cuco X$ contains a factor system.  Then the induced action of $G$ on the contact graph $\contact\cuco X$ of $\cuco X$ is acylindrical.
\end{thmi}

Theorem~\ref{thmi:acyl} and the results of~\cite{maher2014random} combine to yield the following, which is related to work of Nevo-Sageev on Poisson boundaries of cube complexes~\cite{NevoSageev:Poisson}:

\begin{cori}[Poisson boundary]\label{cori:poisson_boundary}
Let the group $G$ act properly and cocompactly on the CAT(0) cube
complex $\cuco X$ and suppose that $\cuco X$ contains a factor system.  Let $\mu$ be a probability distribution on $G$ with finite
entropy whose support generates a non-elementary group 
acting on
$\contact\cuco X$ and let $\nu$ be the hitting measure on
$\partial\contact\cuco X$.  Then $(\partial \contact\cuco X,\nu)$ is 
isomorphic to  
the Poisson boundary of $(G,\mu)$.
\end{cori}

Theorem~\ref{thmi:acyl} also allows one to produce free subgroups of $G$ freely generated by finite collections of high powers of elements, each of which acts loxodromically on $\contact \cuco X$ (Corollary~\ref{cor:free_subgroups}).

Using Theorem~\ref{thmi:distance_formula_intro}, together with 
the tools in~\cite{BBF:quasi_tree}, the fact that factored
contact graphs are quasi-trees, and the machinery we develop 
in Section~\ref{sec:par_proj}, we prove:

\begin{thmi}\label{thmi:quasi_trees}
Let $G$ act properly and cocompactly on the CAT(0) cube complex $\cuco X$ and suppose that $\cuco X$ contains a $G$--invariant factor system.  If the action of $G$ is \emph{hereditarily flip-free}, then $G$ quasi-isometrically embeds in the product of finitely many quasi-trees.

Moreover, there exist such $G,\cuco X$ such that, for all finite-index subgroups $G'\leq G$, there is no $G'$--equivariant isometric embedding of $\cuco X$ in the product of finitely many simplicial trees.
\end{thmi}

Hereditary flip-freeness is a mild technical condition on the 
factor system $\factorsup$ which holds for most of the examples we 
have discussed.  For the second assertion of the theorem, we exploit 
the existence of
cocompactly cubulated groups with no finite
quotients~\cite{BurgerMozes,Wise:CSC}. These 
groups are lattices in products of trees stabilizing
the factors; this property gives rise to a factor system.  The space 
for this example is assembled from these pieces in such a way that
the existence of a factor system persists, but there is no
longer a finite equivariant coloring of the hyperplanes with
intersecting hyperplanes colored differently.  This lack of a 
coloring precludes the existence of an isometric embedding in a 
product of finitely many trees.

\subsection{Comparison to the theory of the extension graph of a right-angled Artin group}\label{subsubsec:contact_extension}
In the special case where $\cuco X=\widetilde S_\Gamma$ is the 
universal cover of the Salvetti complex $S_\Gamma$ of a  
right-angled Artin group $A_\Gamma$, the machinery 
of factor systems and contact graphs is not the first attempt to 
define an analogue of the curve complex and the attendant 
techniques.  In~\cite{KimKoberda:embed}, Kim--Koberda introduced the 
\emph{extension graph} $\Gamma^e$ associated to the finite simplicial 
graph $\Gamma$ (and thus to $A_\Gamma$).  This graph has a vertex for 
each conjugate of each standard generator of $A_\Gamma$ (i.e., vertex 
of $\Gamma$), with adjacency recording commutation.  In the same 
paper it is shown that, like $\contact\widetilde 
S_\Gamma$, the extension graph is always quasi-isometric to a tree, 
and in~\cite{KimKoberda:curve_graph}, the analogy between $\MCG(S)$, 
with its action on $\contact S$, and $A_\Gamma$, with its action on 
$\Gamma^e$, is extensively developed: it is shown, for instance, that 
this action is acylindrical and obeys a loxodromic-elliptic dichotomy.   It is observed in~\cite{KimKoberda:curve_graph} that, except 
in exceptional cases, there is a surjective graph homomorphism $\contact\widetilde 
S_\Gamma\rightarrow\Gamma^e$, where $\widetilde S_\Gamma$ is the 
universal cover of the Salvetti complex of $A_\Gamma$, which is also 
a quasi-isometry, so many such \emph{geometric} results about the 
action of $A_\Gamma$ on $\Gamma^e$ can be deduced from the results of 
the present paper about the action of $A_\Gamma$ on 
$\contact\widetilde S_\Gamma$.  It should be strongly emphasized that the papers~\cite{KimKoberda:curve_graph,KimKoberda:embed} also explore interesting and less purely geometric issues, particular to right-angled Artin groups, that cannot be treated with factor system tools.

The authors of~\cite{KimKoberda:curve_graph} also set up some
version of hierarchy machinery, with the role of subsurfaces being played by
subgroups of $A_\Gamma$ of the form $A_{\link(v)}$, where $v$ is a
vertex of $\Gamma$, and their conjugates.  
When $\Gamma$ has girth at least $5$, they obtain a distance formula \cite[Proposition~65]{KimKoberda:curve_graph}, 
but, as they note, the formula they give has 
significant differences with the Masur--Minsky distance formula for
the mapping class group.  For example, the sum is taken over specified
projections, which depend on the points whose distance is being
estimated, rather than over all projections. 
Another significant distinction is that their distance 
formula does not measure distance in the right-angled Artin group  
$A_{\Gamma}$, but rather it measures the \emph{syllable length} in that 
space (although not a perfect analogy: their metric is more similar  
to the Weil--Petersson metric 
on Teichm\"{u}ller space than to the word metric on the mapping class group).
The extension graph seems unable to capture distance in the 
right-angled Artin group via a hierarchical construction, since the extension graph 
is bounded when $A_\Gamma=\integers$.
In the present paper, the geometric
viewpoint afforded by factor systems, and in particular the existence
of hierarchy paths (Proposition~\ref{prop:hier_revisited}) and a Large
Link Lemma (Proposition~\ref{prop:BGIIITF}) 
allows us to overcome these issues.

\subsection{Hierarchically hyperbolic spaces}

Our aim in the last part of the paper is to develop a unified
framework to study mapping class group and CAT(0) cube complexes from
a common perspective.  To this end, we axiomatize the machinery
of factored contact graphs/curve complexes, distance formula, etc.,
to obtain the definition of a \emph{hierarchically
hyperbolic space}, which is formally stated in 
Definition~\ref{defn:space_with_distance_formula}. 
This notion includes the two classes of 
groups just mentioned and allows one to prove new results for both of 
these classes simultaneously.  Hierarchically
hyperbolic spaces come with a notion of \emph{complexity}: complexity
0 corresponds to bounded spaces, infinite diameter $\delta$--hyperbolic spaces have complexity 1, and higher complexity hierarchically hyperbolic spaces coarsely contain 
direct products. 

Roughly, a space $\cuco X$ is hierarchically hyperbolic if $\cuco X$ can be
equipped with a set $\mathfrak S$ of uniformly Gromov-hyperbolic
spaces, and projections $\cuco X\rightarrow W$, with $W\in\mathfrak
S$. These projections are required to satisfy various properties reminiscent of those satisfied by
subsurface projections in the mapping class group case and projections
to factored contact graphs in the case of CAT(0) cube complexes with
factor systems.  Hence a space $\cuco X$ may be hierarchically
hyperbolic in multiple ways, i.e., with respect to projections to
distinct families of hyperbolic spaces.

\begin{remi}[HHS is a QI-invariant property]
It is easily seen from Definition~\ref{defn:space_with_distance_formula} that, if $\cuco X$ is hierarchically hyperbolic by virtue of its projections to a set $\mathfrak S$ of hyperbolic spaces, and $\cuco Y\rightarrow\cuco X$ is a quasi-isometry, then we can compose each projection with the quasi-isometry and conclude that $\cuco Y$ is hierarchically hyperbolic with respect to the same set $\mathfrak S$.
\end{remi}

The motivating examples of hierarchically hyperbolic spaces are as follows:

\begin{thmi}[Hierarchically hyperbolic spaces]\label{thmi:hhs}

\begin{enumerate}
 \item [  ]
\item A CAT(0) cube complex with a factor system $\factorsup$ is hierarchically hyperbolic with respect to the set of factored contact graphs $\fontact W$, with $W\in\factorsup$.  (This is summarized in Remark~\ref{rem:cube_complex_case}.)
 \item Let $S$ be a connected, oriented hyperbolic surface of finite type.  Then $\MCG(S)$ is hierarchically hyperbolic with respect to the collection of curve complexes of subsurfaces of $S$~\cite{MasurMinsky:I,MasurMinsky:II,Behrstock:asymptotic,BKMM:consistency}.
 \item Teichm\"{u}ller space $\mathcal T(S)$ with the Weil-Petersson metric is hierarchically hyperbolic with respect to curve complexes of non-annular subsurfaces of $S$~\cite{MasurMinsky:I,MasurMinsky:II,Brock:pants,Behrstock:asymptotic,BKMM:consistency}.
\item $\mathcal T(S)$ with the Teichm\"{u}ller metric is hierarchically hyperbolic with respect to curve complexes of non-annular subsurfaces and combinatorial horoballs associated to annuli~\cite{MasurMinsky:I,Rafi:combinatorial_teichmuller,Durham:augmented,EskinMasurRafi:large_scale_rank}.
\end{enumerate}

\end{thmi}

In a forthcoming paper we will show that
fundamental groups of non-geometric 3-manifolds are also
hierarchically hyperbolic, and that a metric space that is hyperbolic 
relative to a collection of hierarchically hyperbolic subspaces is 
hierarchically hyperbolic
\cite{BehrstockHagenSisto:MoreHHSexamples}. 
We note that it is already known that relatively hyperbolic groups 
admit a distance formula  \cite{Sisto-distformrelhyp}.
Another interesting
question is whether a right-angled Artin group endowed with the
syllable length metric is a hierarchically hyperbolic space.

\medskip

In Section~\ref{sec:quasi_box}, after defining hierarchically hyperbolic spaces, we study quasi-Lipschitz maps from balls in $\mathbb R^n$, and 
more general nilpotent Lie groups, into these spaces. The next three results will all follow directly from our Theorem 
\ref{thm:density_point_ascone} which provides a single unifying statement in terms of asymptotic cones.  Our first
result is a generalization 
of a result from \cite[Theorem~A]{EskinMasurRafi:large_scale_rank} 
which is about mapping class groups and Teichm\"uller spaces.

\begin{thmi}[Quasi-boxes in hierarchically hyperbolic spaces]\label{thmi:qb_in_hs}
 Let $\cuco X$ be a hierarchically hyperbolic space. Then for every $n\in\mathbb N$ and every $K,C,R_0,\epsilon_0$ the following holds.
There exists $R_1$ so that for any ball $B\subseteq \mathbb R^n$ of radius at least $R_1$ and
$f\colon B\to \cuco X$ a $(K,C)$--quasi-Lipschitz map, there is a ball
$B'\subseteq B$ of radius $R'\geq R_0$ such that $f(B')$ lies inside the
$\epsilon_0 R'$--neighborhood of a standard box.
\end{thmi}

In Theorem~\ref{thmi:qb_in_hs}, we do not require $B,B'$ to be centered at the same point in $\reals^n$.

Our proof uses methods different from those used in 
\cite{EskinMasurRafi:large_scale_rank}. 
Our approach is much shorter 
and does not rely on 
partitions of the set of subsurfaces (or an analogue thereof), which 
plays an important role in their proof. In particular, we do not 
rely on the results from
\cite{BBF:quasi_tree}; this is one reason why our results can be applied in the 
case of CAT(0)
cube complexes where the techniques of
\cite{EskinMasurRafi:large_scale_rank} would fail.  
However, our approach and theirs share some commonalities, for 
instance, we use Rademacher's Theorem (applied to maps that arise at the level of 
asymptotic cones), while in
\cite{EskinMasurRafi:large_scale_rank}, the authors use a coarse
differentiation result.

Using a generalization
of Rademacher's theorem due to Pansu, we consider the case of 
quasi-Lipschitz maps from more
general nilpotent Lie groups.

\begin{thmi}[Restriction on nilpotent groups in hierarchically hyperbolic spaces]\label{thmi:nilpemb}
 Let $\cuco X$ be a hierarchically hyperbolic space. Then for every
simply connected nilpotent Lie group $\mathcal N$, with a 
left-invariant Riemannian metric, and every $K,C$ there exists $R$ 
with the following property. For every $(K,C)$--quasi-Lipschitz map 
$f\colon B\to\cuco X$ from a ball in $\mathcal N$ into $\cuco X$ and for every $n\in\mathcal N$ we have $\diam(f(B\cap n[\mathcal N,\mathcal N]))\leq R$. In particular, if a finitely generated nilpotent group admits a quasi-isometric embedding into $\cuco X$ then it is virtually abelian.
\end{thmi}

The final conclusion of Theorem \ref{thmi:nilpemb} is known in the
case where $\cuco X$ is a CAT(0) space~\cite{Pauls:no_nilp_in_cat0}.
Although it does not appear to be in the literature, the conclusion of
Theorem~\ref{thmi:nilpemb} for $\MCG$ can be alternatively proved
using~\cite{Hume:mcg_product_of_trees} and the results of \cite{Pauls:no_nilp_in_cat0}.

The following theorem generalizes the Rank Theorems from
\cite{BehrstockMinsky:dimension_rank,EskinMasurRafi:large_scale_rank, 
Hamenstadt:TT3arxiv, Kleiner:CBA}:

\begin{thmi}[Rank]\label{thmi:rank}
 Let $\cuco X$ be a hierarchically hyperbolic space with respect to a
 set $\mathfrak S$.  If there exists a quasi-isometric embedding
 $\mathbb R^n\to\cuco X$ then $n$ is at most the maximal cardinality
 of a set of pairwise-orthogonal elements of $\mathfrak S$ and, in
 particular, at most the complexity of~$\cuco X$.
\end{thmi}

When $\cuco X$ is a CAT(0) cube complex with a factor system and 
$\Aut(\cuco X)$ acts cocompactly, then such a space  
naturally has two 
hierarchically hyperbolic structures. One of these structures has 
hierarchy paths that are 
combinatorial geodesics and one with CAT(0) geodesics; the first is 
obtained explicitly in Section~\ref{subsec:factored_contact_graphs} and 
the existence of the latter follows from the first via a simple argument about 
projections of CAT(0) geodesics and convex hulls of $\ell_{1}$ 
geodesics to factored contact graphs.  
By Theorem~\ref{thmi:qb_in_hs} and cocompactness, the existence of a quasi-isometric embedding $\mathbb R^n\to
\cuco X$ then implies that $\cuco X$ contains both an $\ell_1$--isometrically
embedded copy of $\mathbb R^n$ with the standard tiling and, in the 
CAT(0) metric, an isometrically embedded flat of dimension $n$. We 
thus recover the cubical version of a theorem of Kleiner, 
see \cite[Theorem~C]{Kleiner:CBA}.  We also note that, in the special case of
top-dimensional quasiflats in CAT(0) cube complexes, Huang has
very recently proved a stronger statement~\cite[Theorem~1.1]{Huang:quasiflats}.

\medskip

Finally, we relate hierarchically hyperbolic groups to acylindrically
hyperbolic groups, as studied in~\cite{Osin:acyl}.  Although
natural, the definition of an \emph{automorphism} of a hierarchically
hyperbolic space is technical, but includes, in the relevant cases,  all elements of $\MCG$ and
all isometries of a cube complex with a factor system.  The ``maximal
element'' $S$ and $\fontact S$ referred to below are $\cuco X$ and its
factored contact graph in case $\cuco X$ is a CAT(0) cube complex,
while they are the surface $S$ and its curve complex when $G$ is the
mapping class group of $S$.

\begin{thmi}\label{thmi:acyl_hhs_intro}
Let $\cuco X$ be hierarchically hyperbolic with respect to the set $\mathfrak S$ of hyperbolic spaces and let $G\leq\Aut(\mathfrak S)$ act properly and cocompactly on $\cuco X$. Let $S$ be the maximal element of $\mathfrak S$ and denote by $\fontact S$ the corresponding hyperbolic space. Then $G$ acts acylindrically on $\fontact S$.
\end{thmi}

In the case of the acylindricity
of the action of $\MCG$ on the curve complex, 
our argument provides a new proof of a result of Bowditch \cite{Bowditch:tight}. 
We note that our proof is substantially different and 
is also rather short (although it does relies on some amount of
machinery).

Throughout the section on hierarchically hyperbolic spaces, we make
use of a notion of ``gate,'' which, in addition to the uses in this 
paper, we believe will be very
useful for other applications as well.  This notion simultaneously generalizes
gates/projections on (certain) convex subspaces in the cubical context
and coarsely Lipschitz retractions of the marking complex on its
natural subspaces associated with subsurfaces.  Nonetheless, the
definition we give exploits a different point of view, which turns out
to be very convenient and allows us to give different (and 
concise) proofs than previously existing ones for a number of results 
about mapping class groups and CAT(0) cube complexes.

\subsection*{Acknowledgements} The authors thank Matthew Durham,
Jingyin Huang, Sang-hyun Kim, Bruce Kleiner, Thomas Koberda, Urs Lang,
and Kasra Rafi for stimulating discussions and Chris O'Donnell for a
helpful correction. We also thank those that have given us helpful 
feedback on this paper, especially Jacob Russell and the anonymous 
referees for their numerous useful comments.

\section{Background}\label{sec:background} 
We assume that the reader is familiar with basic properties of CAT(0) 
cube complexes, hyperplanes, median graphs, etc., and refer the 
reader to e.g., \cite{BandeltChepoi_survey, ChatterjiNiblo,
Chepoi:cube_median, Hagen:quasi_arboreal, HagenBoundary,
Haglund:graph_product, Wise:quasiconvex_hierarchy, WiseNotes} for
background on these concepts as they are used in this paper.  
Nonetheless,
in Section~\ref{subsec:convex_and_parallelism} and
Section~\ref{subsec:background_contact}, we will review cubical
convexity and make explicit the notions of \emph{gate},
\emph{projection}, and \emph{parallelism}, since they play a
fundamental role, and recall several useful facts about the contact
graph.  In Section~\ref{subsec:raag_extension_graph}, we briefly
discuss background on right-angled Artin groups.  Since, in
Section~\ref{sec:quasi_box}, we will use asymptotic cones, we refer the reader to~\cite{Drutu:cones} for the relevant
background.

\subsection{Convex subcomplexes, combinatorial hyperplanes, gates, and
parallelism}\label{subsec:convex_and_parallelism}Throughout, $\cuco X$
is a CAT(0) cube complex and $\mathcal H$ is the set of hyperplanes.
Unless stated otherwise, we work in the $1$--skeleton of $\cuco X$,
and we denote by $\dist_{\cuco X}$ the graph metric on $\cuco X^{(1)}$.  The
\emph{contact graph} $\contact\cuco X$, defined
in~\cite{Hagen:quasi_arboreal}, is the graph whose vertex set is
$\mathcal H$, with two vertices adjacent if the corresponding
hyperplanes are not separated by a third.

The subcomplex $K\subset\cuco X$ is \emph{full} if $K$ contains each 
cube of $\cuco X$ whose $1$--skeleton appears in $K$.  A subcomplex
$K\subset\cuco X$ is \emph{isometrically embedded} if $K$ is full and
$K\cap \bigcap_iH_i$ is connected for all $\{H_i\}\subset\mathcal H$.  In this case, we say
that $H\in\mathcal H$ \emph{crosses} $K$ when $K\cap H\neq\emptyset$.
The term ``isometrically embedded'' is justified by the well-known
fact (see e.g., \cite{Hagen:quasi_arboreal}) that, if $K\subset\cuco
X$ is isometrically embedded in this sense, then
$K^{(1)}\hookrightarrow\cuco X^{(1)}$ is an isometric embedding with
respect to graph metrics.  The isometrically embedded subcomplex
$K\subset\cuco X$ is \emph{convex} if any of the following equivalent
conditions is met:
\begin{enumerate}
 \item $K$ coincides with the intersection of all combinatorial halfspaces (see below) containing $K$.
 \item  Let $x,y,z\in \cuco X$ be 0-cubes with $x,y\in K$.  Then the median of $x,y,z$ lies in $K$.
 \item  $K^{(1)}$ contains every geodesic of $\cuco X^{(1)}$ whose endpoints lie in $K^{(1)}$.
 \item  Let $c$ be an $n$--cube of $\cuco X$, with $n\geq 2$.  Suppose that $c$ has $n$ codimension-1 faces that lie in $K$.  Then $c\subset K$.
 \item The inclusion $K\rightarrow\cuco X$ is a local isometry.
\end{enumerate}

(Recall that a combinatorial map $\phi:K\to\cuco X$ of cube complexes, with $\cuco X$ nonpositively-curved, is a \emph{local isometry} if $\phi$ is locally injective and, for each $x\in K^{(0)}$, the map induced by $\phi$ on the link of $x$ is injective and has image a full subcomplex of the link of $\phi(x)$; see~\cite{HaglundWiseSpecial,Wise:quasiconvex_hierarchy,WiseNotes} for more on local isometries and local convexity.)

A convex subcomplex $K\subseteq\cuco X$ is itself a CAT(0) cube complex, whose hyperplanes are the subspaces of the form $H\cap K$, where $H\in\mathcal H$.  A useful mantra, following from the definition of convexity, is: ``if $K$ is convex, then any two hyperplanes that cross $K$ and cross each other must cross each other inside of $K$''.  The \emph{convex hull} of $Y\subseteq\cuco X$ is the intersection of all convex subcomplexes containing $Y$.

It follows immediately from the definition that if $K\subseteq\cuco X$ is convex, then for all $x\in\cuco X^{(0)}$, there exists a unique closest 0-cube $\gate_K(x)$, called the \emph{gate} of $x$ in $K$.  The gate is characterized by the property that $H\in\mathcal H$ separates $\gate_K(x)$ from $x$ if and only if $H$ separates $x$ from $K$.  

As discussed in~\cite{Hagen:quasi_arboreal}, it follows from the definition of convexity that the inclusion $K\hookrightarrow\cuco X$ induces an injective graph homomorphism $\contact K\rightarrow\contact\cuco X$ whose image is a full subgraph: just send each $H\cap K$ to $H$.  This allows us to define a projection $\gate_K\colon\cuco X\rightarrow K$.  The map $x\mapsto\gate_K(x)$ extends to a cubical map $\gate_K\colon\cuco X\rightarrow K$ as follows.  Let $c$ be a cube of $\cuco X$ and let $H_1,\ldots,H_d$ be the collection of pairwise-crossing hyperplanes crossing $c$.  Suppose that these are labeled so that $H_1,\ldots,H_s$ cross $K$, for some $0\leq s\leq d$, and suppose that $H_{s+1},\ldots,H_d$ do not cross $K$.  Then the 0-cubes of $c$ map by $g_K$ to the 0-cubes of a uniquely determined cube $\gate_K(c)$ of $K$ in which the hyperplanes $H_1,\ldots,H_s$ intersect, and there is a cubical collapsing map $c\cong[-1,1]^d\rightarrow[-1,1]^s\cong\gate_K(c)$ extending the gate map on the 0-skeleton.  This map is easily seen to be compatible with the face relation for cubes, yielding the desired cubical map $\gate_K\colon\cuco X\rightarrow K$, whose salient property is that for all $x\in\cuco X$, a hyperplane $H$ separates $x$ from $K$ if and only if $H$ separates $x$ from $\gate_K(x)$.  The next lemma follows easily from the definitions and is used freely throughout this paper:

\begin{lem}\label{lem:simple_gate}
Let $\cuco X$ be a CAT(0) cube complex, let $A\subseteq B\subseteq\cuco X$ be convex subcomplexes, and let $x,y\in\cuco X$ be 0-cubes.  Then any hyperplane separating $\gate_A(x),\gate_A(y)$ separates $\gate_B(x),\gate_B(y)$ (and hence separates $x,y$).
\end{lem}

\begin{defn}[Parallel]\label{defn:parallel}
Convex subcomplexes $F,F'\subseteq\cuco X$ are \emph{parallel} if for
each hyperplane $H$, we have $H\cap F\neq\emptyset$ if and only if
$H\cap F'\neq\emptyset$.  Parallelism is clearly an equivalence relation.
\end{defn}

Hyperplanes lead to important examples of parallel subcomplexes.  For
each hyperplane $H\in\mathcal H$, let $\neb(H)$ denote its
\emph{carrier}, i.e., the union of all closed cubes intersecting $H$.
Then there is a cubical isometric embedding $H\times[-1,1]\cong
\neb(H)\hookrightarrow\cuco X$, and we denote by $H^\pm$ be the images
of $H\times\{\pm1\}$.  These convex subcomplexes are
\emph{combinatorial hyperplanes}.  Observe that $H^+$ and $H^-$ are
parallel: a hyperplane crosses $H^{\pm}$ if and only if it crosses
$H$.  (Very occasionally, it will be convenient to refer to
\emph{combinatorial halfspaces} --- a combinatorial halfspace
associated to $H\in\mathcal H$ is a component of $\cuco
X-H\times(-1,1)$.  A component of the boundary of such a
combinatorial halfspace is one of $H^+$ or $H^-$.)

\begin{rem}[Parallelism and dual cube complexes]The reader accustomed to thinking of cube complexes using Sageev's construction of the cube complex dual to a space with walls~\cite{Sageev:cubes_95,ChatterjiNiblo} might appreciate the following characterization of parallelism: the convex subcomplex $F$ can be viewed as the dual to a wallspace whose underlying set is $\cuco X^{(0)}$ and whose walls are the hyperplanes in the vertex set of $\contact F$ -- one must check that in this case, the restriction quotient $\cuco X\rightarrow F$ is split by a cubical isometric embedding $F\rightarrow\cuco X$.  In general, the splitting need not be unique, depending on a choice of basepoint, and the images of the various embeddings are exactly the representatives of the parallelism class of $F$.
\end{rem}

Observe that $F,F'$ are parallel if and only if their contact graphs are the same subgraph of $\contact\cuco X$.   Parallel subcomplexes are
isomorphic and in fact, the following stronger statement holds, and we shall use it throughout the paper.

\begin{lem}\label{lem:parallel_product}
Let $F,F'\subseteq\cuco X$ be convex subcomplexes.  The following are equivalent:
\begin{enumerate}
 \item $F$ and $F'$ are parallel;
 \item there is a
cubical isometric embedding $F\times[0,a]\rightarrow\cuco X$ whose
restrictions to $F\times\{0\},F\times\{a\}$ factor as
$F\times\{0\}\cong F\hookrightarrow\cuco X$ and $F\times\{a\}\cong
F'\hookrightarrow\cuco X$, and $[0,a]$ is a combinatorial geodesic
segment crossing exactly those hyperplanes that separate $F$ from
$F'$. 
\end{enumerate}
Hence there exists a convex subcomplex $E_F$ such that there is a cubical embedding $F\times E_F$ with convex image such that for each $F'$ in the parallelism class of $F$, there exists a 0-cube $e\in E_F$ such that $F\times\{e\}\rightarrow\cuco X$ factors as $F\times\{e\}\stackrel{id}{\rightarrow} F'\hookrightarrow F$. 
\end{lem}

\begin{proof}
Let $F,F'$ be parallel, let $\mathcal H(F)$ be the set of hyperplanes crossing $F$ (hence $\mathcal H(F')=\mathcal H(F)$) and let $\mathcal S$ be the finite set of hyperplanes separating $F$ from $F'$.  Let $E$ be the cubical convex hull of $F\cup F'$.  Then every element of $\mathcal H(F)$ crosses $E$, and the same is true of $\mathcal S$, since any geodesic starting on $F$ and ending on $F'$ must contain a $1$--cube dual to each element of $\mathcal S$.  Conversely, let $H\in\mathcal H$ cross the hull of $F\cup F'$ and suppose that $H\not\in\mathcal H(F)\cup\mathcal S$.  Then either $H$ crosses $F$ but not $F'$, which is impossible, or $F$ and $F'$ lie in the same halfspace $\OL H$ associated to $H$.  But then $\OL H$ contains a combinatorial halfspace $\OL H^*$ that contains $F,F'$, so $E\cap\OL H^*$ is a convex subcomplex containing $F\cup F'$ and properly contained in $E$ (since $H$ does not cross $E\cap\OL H^*$).  This contradicts that $E$ is the convex hull.  Hence the set of hyperplanes crossing $E$ is precisely $\mathcal H(F)\cup\mathcal S$.  Each hyperplane in $\mathcal H(F)$ crosses each hyperplane in $\mathcal S$, and it follows from~\cite[Proposition~2.5]{CapraceSageev:rank_rigidity} that $E\cong F\times I$ for some convex subcomplex $I$ such that the set of hyperplanes crossing $I$ is precisely $\mathcal S$.  It is easily verified that $I$ is the convex hull of a geodesic segment.  This proves $(1)\Rightarrow(2)$; the other direction is obvious.

The ``hence'' assertion follows from an identical argument, once $\mathcal S$ is taken to be the set of hyperplanes $H$ with the property that for some $F',F''$ in the parallelism class of $F$, we have that $H$ separates $F$ from $F'$.
\end{proof}

\begin{lem}\label{lem:fundamental_fs_property}
For each convex subcomplex $F\subseteq\cuco X$, either $F$ is unique in its parallelism class, or $F$ is contained in a combinatorial hyperplane.
\end{lem}

\begin{proof}
If $F$ is parallel to some $F'\neq F$, then as discussed above, there is a cubical isometric embedding $F\times[0,a]\rightarrow\cuco X$, with $a\geq 1$, whose restriction to $F\times\{0\}$ is the inclusion $F\hookrightarrow\cuco X$.  Cubical isometric embeddings take combinatorial hyperplanes to subcomplexes of combinatorial hyperplanes, and $F\times\{0\}$ is a combinatorial hyperplane of $F\times[0,a]$ since $a\geq 1$, so the claim follows.
\end{proof}

The following will be used often starting in Section~\ref{subsec:hypercarrier}.

\begin{lem}\label{lem:2.6}
If $F,F'$ are convex subcomplexes, then $\gate_F(F')$ and $\gate_{F'}(F)$ are parallel subcomplexes. Moreover, if $F\cap F'\neq\emptyset$, then $\gate_F(F')=\gate_{F'}(F)=F\cap F'$; 
\end{lem}

\begin{proof}
Suppose that $H$ is a hyperplane crossing $\gate_F(F')$.  Then $H$ separates $\gate_F(x),\gate_F(y)$ for some $x,y\in F'$.  By Lemma~\ref{lem:simple_gate}, $H$ separates $x,y$ and hence crosses $F'$.  Thus $H$ crosses both $F$ and $F'$.  Conversely, suppose that a hyperplane $H$ crosses $F$ and $F'$, separating $x,y\in F'$.  Then $H$ cannot separate $x$ from $\gate_F(x)$ or $y$ from $\gate_F(y)$, so $H$ separates $\gate_F(x)$ and $\gate_F(y)$, and in particular crosses $\gate_F(F')$.  Thus the set of hyperplanes crossing $\gate_F(F')$ is precisely the set of hyperplanes $H$ that cross both $F$ and $F'$.  Similarly, the set of hyperplanes crossing $\gate_{F'}(F)$ is the set of hyperplanes $H$ that cross both $F$ and $F'$.  Hence $\gate_F(F'),\gate_{F'}(F)$ are parallel.

Suppose that $F\cap F'\neq\emptyset$ and let $x\in F\cap F'$.  Then $\gate_F(x)=x$, by definition, so $\gate_F(F')\supseteq F\cap F'$.  On the other hand, let $y\in\gate_F(F')$, so $y=\gate_F(y')$ for some $y'\in F'$.  Let $m$ be the median of $y,y',x$ for some $x\in F\cap F'$.  By convexity of $F$ and of $F'$, we have $m\in F\cap F'$ and $\dist_{\cuco X}(y,y')\geq\dist_{\cuco X}(y',m)$.  Hence $y=m$, so $y\in F\cap F'$.  Thus $\gate_F(F')\subseteq F\cap F'$, as required.
\end{proof}

The following lemmas concern the projection of geodesics onto hyperplanes:

\begin{lem}\label{lem:background_hyperplane_segment}
Let $\alpha\subseteq\cuco X$ be a combinatorial geodesic and let $K\subseteq\cuco X$ be a convex subcomplex.  Then $\gate_K(\alpha)$ is a geodesic in $K$.  Moreover, suppose that there exists $R$ such that $\dist_{\cuco X}(a,K)\leq R,\dist_{\cuco X}(a',K)\leq R$ where $a,a'$ are the initial and terminal 0-cubes of $\alpha$.  Then $\dist_{\cuco X}(t,K)\leq 2R$ for all 0-cubes $t$ of $\alpha$.
\end{lem}

\begin{proof}
Let $\alpha'=\gate_K(\alpha)$.  Any hyperplane separating $t,t'\in\alpha'$ separates $s,s'\in\alpha$ with
$\gate_K(s)=t,\gate_K(s')=t'$, by the definition of gates.  Hence the
set of hyperplanes crossing $\alpha'$ has the following properties: if
$H,H'$ cross $\alpha'$ and are separated by $H''$, then $H''$ crosses
$\alpha'$; if $H,H',H''$ are pairwise-disjoint hyperplanes crossing
$\alpha'$, then one of them separates the other two.  Hence, if
$a_0,\ldots,a_m$ is an ordered sequence of 0-cubes of $\alpha$, then
$\gate_K(a_0),\ldots,\gate_K(a_m)$ has the property that
$\gate_K(a_i),\gate_K(a_{i+1})$ are either equal or adjacent for all
$i$.  If a hyperplane $H$ separates $\gate_K(a_i)$ from
$\gate_K(a_{i+1})$ and also separates $\gate_K(a_j)$ from
$\gate_K(a_{j+1})$, then $H$ must separate $a_i$ from $a_{i+1}$ and
$a_j$ from $a_{j+1}$, a contradiction.  This proves the first assertion.

To prove the second assertion, let $\alpha'=\gate_K(\alpha)$.  Let $\mathcal H$ be the set of hyperplanes separating $\alpha$ from $\gate_K(\alpha)$.  For all $t\in\alpha$, any hyperplane $V$ separating $t$ from $K$ must not cross $K$.  Hence $V$ separates $a$ or $a'$ from $K$.  The total number of such $V$ is at most $2R-|\mathcal H|$.
\end{proof}

\subsection{Contact and crossing graphs and complexes}\label{subsec:background_contact}
As before, $\cuco{X}$ is an arbitrary CAT(0) cube complex and
$\contact\cuco X$ its contact graph.  Recall that hyperplanes $H,H'$
represent adjacent vertices of $\contact\cuco X$ if
$\neb(H)\cap\neb(H')\neq\emptyset$, which occurs if either $H\cap
H'\neq\emptyset$ (in which case $H,H'$ \emph{cross}, denoted $H\bot
H'$), or if there are 1-cubes, dual to $H,H'$ respectively, with a
common 0-cube that do not form the corner of a 2-cube, in which case
$H,H'$ \emph{osculate}.  If $H,H'$ are adjacent in $\contact\cuco X$,
we say that they \emph{contact}, denoted $H\coll H'$. As we just did, 
we often abuse notation by saying a pair of hyperplanes are adjacent, 
when we really mean that the vertices represented by these 
hyperplanes are adjacent; similarly, we will talk about a sequence of 
hyperplanes forming a geodesic, etc.  We sometimes refer to the \emph{crossing graph} $\crossing\cuco X$ of
$\cuco X$, the spanning subgraph of $\contact\cuco X$ formed
by removing open edges that record osculations.

The following statements are the beginning of the analogy between, on 
one hand, $\cuco X$ and its contact graph and, on the other hand, 
$\MCG(S)$ and the curve complex of the surface $S$.  

\begin{thm}[Hyperbolicity of contact graphs; \cite{Hagen:quasi_arboreal}]\label{thm:hyperbolicity}
For any CAT(0) cube complex $\cuco X$, there is a simplicial tree $T$ and a $(10,10)$--quasi-isometry $\contact\cuco X\rightarrow T$.  
\end{thm}

Under fairly mild geometric hypotheses, whether or not the quasi-tree $\contact\cuco X$ is actually a quasi-point can be detected by examining the \emph{simplicial boundary}; moreover, if $\cuco X$ admits an essential, proper, cocompact group action, then $\contact\cuco X$ is either a join of infinite subgraphs, or is unbounded~\cite{HagenBoundary}.  This is closely related to the action of rank-one elements on $\contact\cuco X$ (recall that $g\in\Aut(\cuco X)$ is \emph{rank-one} if it is hyperbolic and no axis bounds a half-flat in $\cuco X$).  As for the extension graph of a right-angled Artin group, there is a ``Nielsen-Thurston classification'' describing how elements of $\Aut(\cuco X)$ act on the contact graph. 

\begin{thm}[``Nielsen-Thurston classification''; \cite{HagenBoundary}]\label{thm:NTC}
Let $\cuco X$ be a CAT(0) cube complex such that every clique in $\contact\cuco X$ is of uniformly bounded cardinality.  Let $g\in\Aut(\cuco X)$.  Then one of the following holds:
\begin{enumerate}
 \item (``Reducible'')  There exists $n>0$ and a hyperplane $H$ such that $g^nH=H$.
 \item (``Reducible'')  There exists a subgraph $\Lambda$ of $\contact\cuco X$ such that $g\Lambda\subset\Lambda$ and the following holds: $\Lambda=\Lambda_0\sqcup\Lambda_1$, and each vertex of $\Lambda_0$ is adjacent to all but finitely many vertices of $\Lambda_1$.  In this case, $g$ is not a rank-one element.
 \item (``Loxodromic'')  $g$ has a quasigeodesic axis in $\contact\cuco X$.  
\end{enumerate}
In the last case, $g$ is contracting and no positive power of $g$ stabilizes a hyperplane.
\end{thm}

We will not use the subgraph $\Lambda$ in the rest of this paper; that description follows from Theorem~D of~\cite{HagenBoundary}.  We will use the following: if $g\in\Aut(\cuco X)$ is not rank-one, then it either stabilizes a cube in $\cuco X$ and thus has a bounded orbit in $\contact\cuco X$, or its combinatorial axis lies uniformly close to a half-flat $F$ in $\cuco X$.  As explained in~\cite[Proposition~5.1]{HagenBoundary}, the set of hyperplanes crossing $F$ has uniformly bounded diameter in $\contact\cuco X$, and hence the same is true of the $\langle g\rangle$ orbit in $\contact\cuco X$.

\subsection{Special cube complexes, right-angled Artin groups, and the extension graph}\label{subsec:raag_extension_graph}
The \emph{right-angled Artin group} $A_\Gamma$ presented by the graph 
$\Gamma$ is the group with presentation: $$\left\langle \Vertices(\Gamma)\mid\{[v,w]:v,w\in\Vertices(\Gamma),\{v,w\}\in\Edges(\Gamma\}\right\rangle,$$
i.e., the generators are the vertices of $\Gamma$ and two generators commute if and only if the corresponding vertices span an edge of $\Gamma$.  The \emph{Salvetti complex} $S_\Gamma$ is the nonpositively-curved cube complex with a single 0-cube and an $n$--cube for each set of $n$ pairwise-commuting generators, for $n\geq 1$.  Note that $S_\Gamma$ is compact, and $A_\Gamma$ finitely generated, if and only if $\Gamma$ is finite.  See the survey~\cite{CharneySurvey} for more details.

The universal cover $\widetilde S_\Gamma$ is a CAT(0) cube complex with special features.  For each induced subgraph $\Lambda$ of $\Gamma$, the inclusion $\Lambda\hookrightarrow\Gamma$ induces an injective local isometry $S_\Lambda\rightarrow S_\Gamma$, lifting to a $A_\Lambda$--equivariant convex embedding $\widetilde S_\Lambda\hookrightarrow\widetilde S_\Gamma$.  In particular, when $\Lambda$ is the link of a vertex $v$, then $\widetilde S_\Lambda$ is a combinatorial hyperplane in $\widetilde S_\Gamma$ which is a copy of a hyperplane whose dual $1$--cubes are labeled by the generator $v$.  Moreover, when $\Lambda$ is the star of $v$, the subcomplex $\widetilde S_\Lambda\cong\mathbf R\times\widetilde S_{\link(v)}$, where $\mathbf R$ is a convex subcomplex isometric to $\reals$ and stabilized by $\langle v\rangle$.

The \emph{extension graph} $\Gamma^e$ associated to $\Gamma$, 
introduced by Kim--Koberda in~\cite{KimKoberda:embed} is the graph 
with a vertex for each $A_\Gamma$ conjugate of each generator $v$ 
with $v\in\Gamma^{(0)}$ and an edge joining the distinct vertices 
corresponding to $v^g,w^h$ if and only if $[v^g,w^h]=1$.  Like the 
contact graph of a cube complex, $\Gamma^e$ is always a 
quasi-tree~\cite{KimKoberda:embed}, and in fact in many situations, 
$\Gamma^e$ is quasi-isometric to $\contact\widetilde S_\Gamma$, as 
explained in~\cite{KimKoberda:curve_graph}.  

\part{Geometry of the contact graph}\label{part:general} 
We now discuss projection to the contact graph and identify ``hierarchy paths'' in a cube complex.  These results will be reworked in Part~\ref{part:distance_formula} once we have introduced factor systems.  

Disc diagram
techniques, originating in unpublished notes of Casson and developed
in~\cite{Chepoi:cube_median, Sageev:cubes_95,
Wise:quasiconvex_hierarchy}, have proven to be a useful tool in
studying the geometry of $\contact\cuco X$
(see~\cite{ChepoiHagen:embedding,Hagen:quasi_arboreal,HagenBoundary}),
and we will continue to make use of these here.

\section{Hierarchy paths}\label{sec:hierpath}
The next proposition establishes Theorem~\ref{thmi:geom_cg}.\eqref{itemi:hierarchy}.

\begin{prop}[``Hierarchy paths'']\label{prop:hierpath}
Let $\cuco{X}$ be a CAT(0) cube complex with contact graph $\contact{\cuco{X}}$ and let $x,y\in\cuco{X}$ be 0-cubes.  Then there exist hyperplanes $H_0,\ldots,H_k$ with $x\in \neb(H_0),y\in \neb(H_k)$, and combinatorial geodesics $\gamma_i\rightarrow \neb(H_i)$, such that $H_0\coll H_1\coll \ldots\coll H_k$ is a geodesic of $\contact{\cuco{X}}$ and $\gamma_0\gamma_1\cdots\gamma_k$ is a geodesic joining $x,y$.
\end{prop}

\begin{proof}
Since the set of hyperplane carriers covers $\cuco{X}$, there exist
hyperplanes $H,H'$ with $x\in \neb(H),y\in \neb(H')$.  Let
$k=d_{\contact{\cuco{X}}}(H,H')$, and let $H=H_0,H_1,\ldots,H_k=H'$ be
a geodesic sequence of hyperplanes in $\contact{\cuco{X}}$, so that
$\neb(H_i)\cap \neb(H_{i+1})\neq\emptyset$ for $0\leq i\leq k-1$ and
$\neb(H_i)\cap \neb(H_j)=\emptyset$ for $|i-j|>1$.  Let $\gamma_0$ be
a combinatorial geodesic of $\neb(H_0)$ joining $x$ to a 0-cube in
$\neb(H_0)\cap \neb(H_1)$.  For $1\leq i\leq k-1$, let $\gamma_i$ be a
geodesic of $\neb(H_i)$ joining the terminal 0-cube of $\gamma_{i-1}$
to a 0-cube of $\neb(H_i)\cap \neb(H_{i+1})$.  Finally, let $\gamma_k$
be a geodesic of $\neb(H_k)$ joining the terminal 0-cube of
$\gamma_{k-1}$ to $y$.  Since hyperplane carriers are convex
subcomplexes of $\cuco{X}$, each $\gamma_i$ is a geodesic of
$\cuco{X}$ and in particular contains at most one $1$-cube dual to each
hyperplane.  Let $\gamma=\gamma_0\cdots\gamma_k$.

We now show that the above choices can be made in such a way that the path $\gamma$ is
immersed in $\cuco{X}$, i.e., $\gamma$ has no self-intersections.
Since $\neb(H_i)\cap \neb(H_j)=\emptyset$ for $|i-j|>1$ we can
restrict our attention to intersections between $\gamma_{i}$ and
$\gamma_{i+1}$. Any such point of intersection must lie in
$\neb(H_i)\cap \neb(H_{i+1})$ and we can then replace
$\gamma_i$ and $\gamma_{i+1}$ by $\gamma'_i$ and $\gamma'_{i+1}$,
where $\gamma'_i\subset \gamma_i$ and $\gamma'_{i+1}\subset
\gamma'_{i+1}$ are the geodesic subpaths obtained by restricting to the
subpath before,
respectively, after, the intersection point. Applying this
procedure for each $i$ to the first intersection point
ensures that $\gamma$ is immersed.

Let $\tau$ be a
combinatorial geodesic of $\cuco{X}$ joining $y$ to $x$, so that
$\gamma\tau$ is a closed path in $\cuco{X}$ and
let $D\rightarrow\cuco{X}$ be a disc diagram with boundary path $\gamma\tau$.  Suppose that $D$ has minimal area among all disc diagrams with this boundary path, and that the choice of geodesic of $\contact{\cuco{X}}$, and the subsequent choice of paths $\gamma_i$, were made in such a way as to minimize the area of $D$.  See Figure~\ref{fig:RAAGhierarchypaths}.

\begin{figure}[h]
\includegraphics[width=0.75\textwidth]{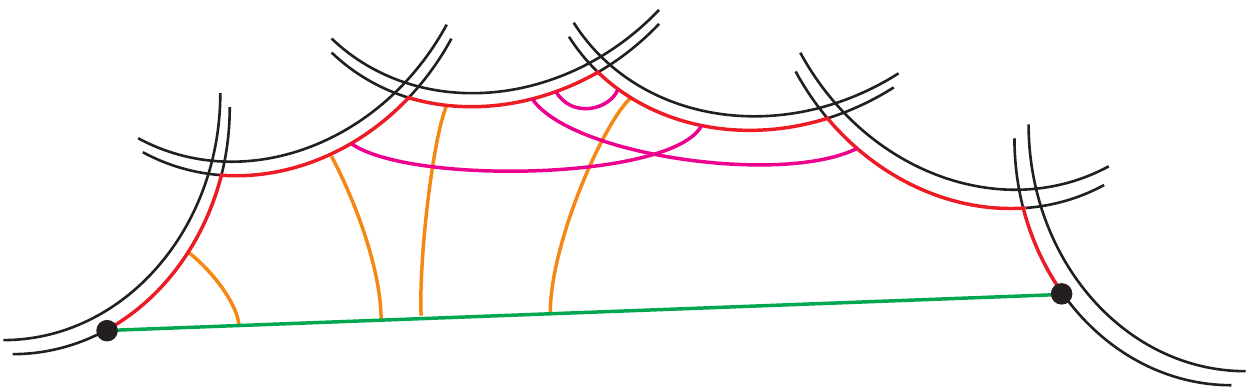}
\caption{The diagram $D$, showing some possible (in orange) and
impossible (in pink) dual curves.}\label{fig:RAAGhierarchypaths}
\end{figure}

Note that for all $i$, if $K,L$ are dual curves emanating from
$\gamma_i$, then $K,L$ do not intersect.  Otherwise, convexity of
$\neb(H_i)$ would enable a new choice of $\gamma_i$, lowering
the area of~$D$.

Let $K$ be a dual curve in $D$ emanating from $\gamma$.  It suffices
to show that $K$ must end on $\tau$.  Indeed, it will then follow
that $|\gamma|\leq|\tau|$, whence $\gamma$ is a geodesic of the
desired type.  To this end, suppose that $K$ emanates from a 1-cube of
$\gamma_i$.  Then $K$ cannot end on $\gamma_i$, since $\gamma_i$
contains exactly one 1-cube dual to the hyperplane to which $K$ maps.
$K$ cannot end on $\gamma_j$ with $|i-j|>2$, for then
$\contact{\cuco{X}}$ would contain a path $(H_i,V,H_j)$, where $V$ is
the hyperplane to which $K$ maps, contradicting the fact that
$d_{\contact{\cuco{X}}}(H_i,H_j)=|i-j|$.  If $K$ ends on
$\gamma_{i\pm2}$, then $\gamma_{i\pm1}$ could be replaced by a path in
$\neb(V)$ that is the image of a path in the interior of $D$,
contradicting our minimal-area choices.  Indeed, recall that we chose the $\contact\cuco X$--geodesic, and the associated geodesics $\gamma_i$, so that the resulting diagram $D$ (which can be constructed given any such choices) is of minimal area among all diagrams constructible by any choice of such $\contact\cuco X$ and $\cuco X$ geodesics.

Finally, suppose $K$ ends on
$\gamma_{i+1}$.  Then no dual curve emanating from the part of
$\gamma_i\gamma_{i+1}$ subtended by the 1-cubes dual to $K$ can cross
$K$.  Hence the 1-cube of $\gamma_i$ dual to $K$ is equal to the
1-cube of $\gamma_{i+1}$ dual to $K$, contradicting the fact that
$\gamma$ is immersed.
\end{proof}

\begin{defn}
    A geodesic $\gamma=\gamma_1\ldots\gamma_n$ such that
    $\gamma_i\rightarrow \neb(H_i)$ for $1\leq i\leq n$ and
    $H_1\coll\ldots\coll H_n$ is a geodesic of $\contact\cuco X$ is a
    \emph{hierarchy path}.
    The geodesic $H_1\coll\ldots\coll H_n$ \emph{carries}
    $\gamma$, and each $\gamma_i$ is a \emph{syllable} of
    $\gamma$.
\end{defn}

The following is immediate from the product structure of hyperplane-carriers:

\begin{prop}\label{prop:canonical_hierarchy_path}
Let $\gamma=\gamma_1\cdots\gamma_n$ be a hierarchy path carried by $H_1\coll\ldots\coll H_n$.  Then there exists a hierarchy path $\gamma'=\gamma_1'\cdots\gamma'_n$, joining the endpoints of $\gamma$ and carried by $H_1\coll\ldots\coll H_n$, such that for each $i$, we have $\gamma'_i=A_ie$, where the combinatorial geodesic $A_i$ lies in $H_i\times\{\pm1\}\subset H_i\times[-1,1]\cong \neb(H_i)$ and $e$ is either a 0-cube or a 1-cube dual to $H_i$.
\end{prop}

A hierarchy path satisfying the conclusion of Proposition~\ref{prop:canonical_hierarchy_path} is \emph{reduced}.  A reduced hierarchy path $\gamma_1\cdots\gamma_n$ is a \emph{hereditary hierarchy path} if it is trivial or if for each $i$, the subpath path $A_i\rightarrow H_i\times\{\pm1\}$ of $\gamma_i$ is a hereditary reduced hierarchy path in $H_i\times\{\pm1\}$.  From Proposition~\ref{prop:canonical_hierarchy_path} and the definition of a hereditary reduced hierarchy path, we obtain:

\begin{cor}\label{cor:fd_hereditary_reduced}
Let $\cuco X$ be a finite-dimensional CAT(0) cube complex and let $x,y\in\cuco X^{(0)}$.  Then $x$ and $y$ are joined by a hereditary reduced hierarchy path.
\end{cor}

\section{Projection to the contact graph}\label{sec:projection}
The notion of projecting geodesics in $\cuco X$ to $\contact\cuco X$ was discussed in~\cite[Section~2]{HagenBoundary}, motivating the following definition:

\begin{defn}[Projection to the contact graph]\label{defn:proj_to_contact}
Let $K$ be a convex subcomplex of $\cuco X$. For each $k\in K$,
let $\{H_i\cap K\}_{i\in\mathcal I}$ be the
maximal collection of hyperplanes of $K$ with
$k\in\cap_{i\in\mathcal I}\neb(H_i)$. Define
$\rho_K\colon K\rightarrow 2^{\contact K}$ by setting $\rho_K(k)=\{H_i\}_{i\in\mathcal I}$.  The
\emph{projection of $\cuco X$ on $\contact K$} is the map
$\pi_K=\rho_K\circ\gate_K\colon \cuco X\rightarrow 2^{\contact K}$.  Note
that $\pi_K$ assigns to each point of $\cuco X$ a clique in the
contact graph of $K$ and hence a clique in $\contact\cuco X$.
\end{defn}

If $H$ is a hyperplane, let $H^\pm$ be $H\times\{\pm1\}\subset
\neb(H)\cong H\times[-1,1]$. There is a bijection between hyperplanes
which intersect $H^{+}$ and ones which intersect $H^{-}$; moreover,
associated hyperplanes contact [respectively, cross] in $H^{+}$ if and only if they
contact [respectively, cross] in $H^{-}$, whence $\contact H^{+}$ and $\contact H^-$ are
both the same subset of $\contact\cuco X$.  Abusing notation slightly, we let $\pi_H$
denote the projection from $\cuco X$ to $\contact H^+=\contact
H^-\subset\contact\cuco X$.  This map is defined as in Definition~\ref{defn:proj_to_contact} since $H^\pm$ is a convex subcomplex, and it is not hard to see that it is
independent of whether we took gates in $H^+$ or $H^-$ (another option
is to pass to the first cubical subdivision and just project to $H$ and then to its contact graph, since subdividing makes $H$ into a
subcomplex).

Let $x,y\in\cuco X$ be 0-cubes, and let $H$ be a hyperplane that does
not separate $x$ from $y$.  Let $H^+$ be the copy of $H$ bounding
$\neb(H)$ that lies in the component of $\cuco X-H$ containing $\{x,y\}$.
It is easily checked that any hyperplane that separates
$\gate_{H^+}(x)$ from $\gate_{H^+}(y)$ must separate $x$ from $y$.  In
particular, if $H$ does not separate $x$ from $y$ and it does not
cross any hyperplane separating $x$ from 
$y$, then $\pi_H(x)=\pi_H(y)$.  In other words, if $\gamma\subset\cuco
X$ is a geodesic with $\pi_{\cuco X}(\gamma)\cap B_1^{\crossing\cuco
X}(H)=\emptyset$, then $\pi_H(x)=\pi_H(y)$, where $x,y$ are the endpoints of
$\gamma$.  Let $x'\in\gamma$ be a 0-cube and suppose that the
hyperplane $U$ crosses $H$ and separates the pair $x,x'$, hence separating
$\gate_{H^+}(x)$ from $\gate_{H^+}(x')$.  Then $U$ separates $x,y$,
and hence belongs to $\pi_{\cuco X}(\gamma)\cap B_1^{\crossing\cuco
X}(H)$, contradicting our assumption.  Hence $\pi_H(\gamma)=\pi_H(x)=\pi_H(y)$.  Thus:

\begin{prop}[Bounded geodesic image]\label{prop:bgiI}
For any $H\in\contact{\cuco{X}}$ and any geodesic $\gamma\subseteq\cuco{X}$ whose image in $\contact{\cuco{X}}$ is disjoint from $B_{1}^{\crossing\cuco X}(H)$, the projection $\pi_{H}(\gamma)$ is a clique.
\end{prop}

\section{Weak proper discontinuity of the action on the contact graph}\label{sec:WPD}

We now consider the elements which act \emph{weakly properly 
discontinuously (WPD)}, in the sense of 
\cite{BestvinaFujiwara:boundedcohom}.
In particular, we study the 
\emph{WPD elements} of a group $G$ acting on $\cuco X$, and prove
Theorem~\ref{thmi:geom_cg}.\eqref{itemi:wpd}.  By definition, as an
isometry of $\contact\cuco X$, some $h\in G$ is WPD if for all
$\epsilon\geq 1$ and hyperplanes $H$, there exists $N>0$ such that
$$\big|\{g\in G:\dist_{\contact\cuco
X}(H,gH)<\epsilon,\dist_{\contact\cuco
X}(h^NH,gh^NH)<\epsilon\}\big|<\infty.$$

In the presence of a factor system and a cocompact group action, we achieve a stronger conclusion in Section~\ref{sec:acyl}, namely that the action on the contact graph is acylindrical.  

\begin{prop}\label{prop:WPD}
Let $\cuco X$ be a CAT(0) cube complex on which the group $G$ acts metrically properly by isometries.  Suppose that $h\in G$ is loxodromic on $\contact\cuco X$.  Then $h$ is WPD.
\end{prop}

\begin{proof}
Fix $\epsilon\geq 1$ and let $H$ be a hyperplane.  By hypothesis,
$\langle h\rangle H$ is a quasigeodesic in $\contact\cuco X$.  It
follows, c.f., \cite[Theorem~2.3]{HagenBoundary}, that there exists
$M=M(h,H,\cuco X)$ such that for all $N\geq 0$, and any hierarchy path
$\gamma=\gamma_1\cdots\gamma_p$ joining some $x\in H$ to $h^Nx\in
h^NH$ and carried on a geodesic $H=H_0\coll H_1\coll\ldots\coll
H_p=h^NH$ of $\contact\cuco X$, we have $|\gamma_i|\leq M$ for $0\leq
i\leq p$.  (Recall that $\gamma$ is a geodesic, being a hierarchy path; hence $\gamma$ contains at most one $1$--cube dual to each hyperplane.)

Suppose that $g\in G$ has the property that $\dist_{\contact\cuco
X}(H,gH)<\epsilon$ and $\dist_{\contact\cuco X}(h^NH,gh^NH)<\epsilon$.
Suppose, moreover, that $N$ has been chosen so that $p>2\epsilon+12$ and choose $i$ so that
$|\frac{p}{2}-i|\leq 2$.  Choose a point $y\in\gamma_i$ and let
$\mathcal H(y)$ be the set of hyperplanes separating $y$ from $gy\in
g\gamma_i$.  
%

We claim that each $W\in\mathcal H(y)$ intersects $\gamma$ and $g\gamma$.  Indeed, suppose that $W$ intersects $\gamma$ but not $g\gamma$.  Suppose that $W$ intersects $\gamma_j$ with $j\geq i$.  Then $W$ separates $x,h^Nx$ (since $\gamma$ is a geodesic) but not $gx,gh^Nx$, so $W$ separates $x,gx$ or $h^Nx,gh^Nx$.  Since $W$ also separates $y,gy$, the former must hold.  Hence $p=\dist_{\contact\cuco X}(H,h^NH)\leq \epsilon+|p-i|\leq p/2+\epsilon$, contradicting our choice of $N$.  The case where $W$ intersects $\gamma_j,j<i$ is similar.
%

%


Hence, if $W$ crosses $\gamma_s$ with $s\geq i$,
then $W$ crosses $g\gamma_t$, with $t\leq i$.  (We have the reverse conclusion if $s\leq i$.)  The fact that
$H_0\coll H_1\coll\ldots\coll H_p$ and its $g$--translate are
$\contact\cuco X$--geodesics implies that $W$
intersects $\gamma_s$ with $|i-s|\leq \epsilon+6$.  Hence $\dist_{\cuco
X}(y,gy)\leq M(\epsilon+6)$, whence the number of such $g$ is finite
since the action of $G$ on $\cuco X$ is proper.
\end{proof}

We say that a group is nonelementary if it does not contain a cyclic subgroup of finite index, and that an action of the group $G$ is WPD if it admits a WPD element and $G$ is nonelementary.

\begin{cor}[Characterization of WPD elements]\label{cor:wPD}
Let $\cuco X$ be a uniformly locally finite CAT(0) cube complex.  Then for any nonelementary group $G$ acting properly on $\cuco X$, one of the following holds: 
\begin{enumerate}
 \item The induced action of $G$ on $\contact\cuco X$ is WPD, and the WPD elements are precisely those $h\in G$ that are rank-one and do not have positive powers that stabilize hyperplanes.
 \item Every $h\in G$ is either not rank-one, or has a positive power 
 stabilizing a hyperplane. 
\end{enumerate}
\end{cor}

\begin{proof}
Apply Proposition~\ref{prop:WPD} and Theorem~\ref{thm:NTC}.
\end{proof}

The following is an application to determining acylindrical 
hyperbolicity. Following \cite{CapraceSageev:rank_rigidity}, we say that the action of the group $G$ on the CAT(0) cube complex $\cuco X$ is essential if every halfspace contains points in a fixed $G$--orbit arbitrarily far away from the associated hyperplane.

\begin{cor}
    Let $G$ be a nonelementary group acting properly and essentially on a 
    uniformly locally finite CAT(0) cube complex, $\cuco X$. 
    Suppose that $\cuco X$ is not a product of unbounded subcomplexes
    and at least one of the following two holds:
    \begin{enumerate}
        \item $G$ acts cocompactly on $\cuco X$.
    
        \item $G$ acts with no fixed point in the simplicial boundary $\partial_\triangle \cuco X$ in the sense of \cite{HagenBoundary}.
    \end{enumerate}
    Then $G$ is acylindrically hyperbolic.
\end{cor}

\begin{proof} By \cite[Theorem~1.2]{Osin:acyl} it suffices to find a 
    hyperbolic space on which $G$ acts with a loxodromic WPD element. 
    By \cite[Theorem~5.4]{HagenBoundary} there exists a rank-one element $g\in G$ no positive power of which stabilizes a hyperplane.
    The conclusion follows from Corollary \ref{cor:wPD}.
\end{proof}

A major motivation for studying WPD actions arises from a result of Bestvina-Fujiwara relating WPD actions to bounded cohomology.  Recall that the space $\widetilde{QH}(G)$, which is the quotient of the space of quasimorphisms of $G$ by the subspace generated by bounded functions and homomorphisms $G\rightarrow\reals$, coincides with the kernel of the map $\homology^2_b(G,\reals)\rightarrow\homology^2(G,\reals)$.  Theorem~7 of~\cite{BestvinaFujiwara:boundedcohom} asserts that if $G$ admits a WPD action, then $\widetilde{QH}(G)$ is infinite-dimensional.  This yields an alternative proof of the dichotomy obtained by Caprace-Sageev in~\cite[Theorem~H]{CapraceSageev:rank_rigidity}, as a consequence of rank-rigidity: a group $G$ admitting a sufficiently nice action on a CAT(0) cube complex $\cuco X$ that is not a product has infinite-dimensional $\widetilde{QH}(G)$; instead of using rank-rigidity and results of~\cite{BurgerMonod:boundedcohom,CapraceMonod:structure,Monod:continuous}, one can deduce their result from rank-rigidity, Corollary~\ref{cor:wPD}, and~\cite{BestvinaFujiwara:boundedcohom}.

\section{Contractibility}\label{sec:contractibility}
We now prove Theorem~\ref{thmi:geom_cg}.\eqref{itemi:contr}.  The \emph{contact complex} $\concom{\cuco{X}}$ of the CAT(0) cube complex $\cuco{X}$ is the flag complex with 1-skeleton $\contact{\cuco{X}}$.

\begin{thm}\label{thm:main}
Let $\cuco{X}$ be a CAT(0) cube complex with countable 0-skeleton and at least one 1-cube.  Then $\concom\cuco X$ is contractible.
\end{thm}

\begin{proof}
For each $n\geq 0$, choose a convex, compact subcomplex $\mathcal B_n$ in such a way that $\mathcal B_m\subseteq\mathcal B_n$ for $m\leq n$ and $\cup_{n\geq0}\mathcal B_n=\cuco X$.  This choice is possible since $\cuco{X}^{(0)}$ may be written as an increasing union of finite sets; $\mathcal B_n$ can then be taken to be the cubical convex hull of the $n^{th}$ one.

Since each $\mathcal B_n$ is convex in $\cuco X$, it is itself a CAT(0) cube complex whose hyperplanes have the form $H\cap\mathcal B_n$, where $H$ is a hyperplane of $\cuco X$.  Moreover, the map $H\cap\mathcal B_n\rightarrow H$ induces an embedding $\concom\mathcal B_n\rightarrow\concom\cuco X$ whose image is a full subcomplex (i.e., any $k+1$ 0-simplices of $\concom\mathcal B_n$ span a $k$--simplex of $\concom\mathcal B_n$ if and only if their images in $\concom\cuco X$ span a $k$--simplex of $\concom\cuco X$).  Thus the set $\{\concom\mathcal B_n\}_{n\geq 0}$ provides a filtration of $\concom X$ by full subcomplexes, each of which is the contact complex of a compact CAT(0) cube complex.  Indeed, every hyperplane intersects all but finitely many of the $\mathcal B_n$, and hence appears as a 0-simplex in all but finitely many of the subcomplexes $\concom\mathcal B_n$.

For any $m\geq 0$ and any continuous map $f\colon \mathbb S^m\rightarrow\concom\cuco X$, compactness of $\image f$ implies that there exists $n\geq 0$ such that $\image f\subseteq\concom\mathcal B_n$.  By Lemma~\ref{lem:finite_contact_complex_contractible}, $\concom\mathcal B_{n+1}$ is contractible, since it is the contact complex of a compact CAT(0) cube complex.  Hence the inclusion $\concom\mathcal B_n\hookrightarrow\concom\mathcal B_{n+1}$ is null-homotopic, whence $\mathbb S^m\stackrel{f}{\longrightarrow}\concom\mathcal B_n\hookrightarrow\concom\mathcal B_{n+1}\hookrightarrow\concom\cuco X$ is null-homotopic.  It then follows from Whitehead's theorem~\cite{whiteheadI,whiteheadII} that $\cuco X$ is contractible.
\end{proof}

\begin{lem}\label{lem:finite_contact_complex_contractible}
Let $\cuco X$ be a compact CAT(0) cube complex with at least one 1-cube.  Then $\concom \cuco X$ is contractible.
\end{lem}

\begin{proof}
We will argue by induction on $n=|\concom\cuco X^{(0)}|$, i.e., the number of hyperplanes in $\cuco X$.  When $n=1$, the cube complex $\cuco X$ is necessarily a single 1-cube, so $\concom\cuco X$ is a 0-simplex.  Let $n\geq 1$.  Since $\cuco X$ has finitely many hyperplanes, there exists a hyperplane $H\subset\cuco X$ such that one of the components of $\cuco X-H$ has closure $H\times[0,1]$, where $\neb(H)$ is identified with $H\times[-1,1]$ and $H$ with $H\times\{0\}$.  This generalizes the case in which $\cuco X$ is a tree and $H$ is the midpoint of an edge containing a degree-1 vertex; accordingly, such a hyperplane $H$ will be called a \emph{leaf hyperplane}.  When $H$ is a leaf hyperplane, we always denote by $H^+$ the halfspace $H\times[0,1]$ and by $H^-$ the other closed halfspace.

By Lemma~\ref{lem:convex_subcomplex}, $\cuco{A}=\closure{\cuco X-\neb(H)}$ is a convex proper subcomplex of $\cuco X$, so that $\concom\cuco A$ embeds in $\concom\cuco X$ as a full subcomplex.  Moreover, by Lemma~\ref{lem:decomposition}, there is a convex proper subcomplex $H'\subset\cuco X$ such that $\concom\cuco X\cong\concom\cuco A\cup_L\left(L\star\{H\}\right),$ where $L$ is a subcomplex of $\concom\closure{\cuco X-\neb(H)}$ isomorphic to $\concom H'$.  Each of $\cuco A$ and $H'$ is a CAT(0) cube complex whose set of hyperplanes corresponds bijectively to a subset of $\cuco X^{(0)}-\{H\}$, so by induction, $\concom\cuco A$ and $L$ are contractible.  Since $L$ is contractible, there is a homotopy equivalence $f\colon L\star\{H\}\rightarrow L$ that is the identity on $L$, whence the pasting lemma yields a homotopy equivalence $\concom\cuco X\rightarrow\concom\cuco A$.  Since $\cuco A$ is contractible, the same is therefore true of $\concom\cuco X$.
\end{proof}

\begin{lem}\label{lem:convex_subcomplex}
Let $H$ be a leaf hyperplane of the CAT(0) cube complex $\cuco X$.  Then $\closure{\cuco X-\neb(H)}$ is a convex subcomplex of $\cuco X$.
\end{lem}

\begin{proof}
Since $H^+\subset \neb(H)$, the subcomplex $\closure{\cuco X-\neb(H)}$ is exactly the convex hull of the halfspace $H^-\cap\cuco X^{(0)}$ of the 0-skeleton induced by $H$.
\end{proof}

\begin{lem}\label{lem:decomposition}
Let $H$ be a leaf hyperplane of the compact CAT(0) cube complex $\cuco X$.  Then there is convex subcomplex $H'\subsetneq\cuco X$ such that there is an isomorphism
$$\concom\cuco X\cong\concom\closure{\cuco X-\neb(H)}\cup_L\left(L\star\{H\}\right),$$ where $L$ is a subcomplex of $\concom\closure{\cuco X-\neb(H)}$ isomorphic to $\concom H'$.
\end{lem}

\begin{proof}
Let $\mathcal A=\closure{\cuco X-\neb(H)}$.  Let $\{V_1,\ldots,V_s\}$ be the hyperplanes of $\cuco X$ that contact $H$, and let $L$ be the full subcomplex of $\concom\cuco X$ generated by $\{V_1,\ldots,V_s\}$.  Lemma~7.11 of~\cite{Hagen:quasi_arboreal} implies that there is an isometrically embedded subcomplex $H'\subset\cuco X$ such that a hyperplane intersects $H'$ if and only if that hyperplane belongs to $\{V_1,\ldots,V_k\}$.  By replacing $H'$ if necessary by its convex hull, we may assume that $H'$ is convex, and hence $\concom H'\cong L$.  The decomposition is obvious from the definitions.
\end{proof}

\section{Automorphisms of the contact graph}\label{section:automorphisms}

We now provide a short example which shows that the analogue of 
Ivanov's Theorem, that the mapping class group is the 
automorphism group of the curve graph \cite{IvanovAut}, fails to hold 
for the contact graph. Indeed, the example below shows 
that this failure holds even if one 
considers an edge-colored version of the contact graph differentiating edges 
associated to crossing hyperplanes from those associated to osculating
ones. We note that in the case of RAAGs, the fact that the 
automorphism group of the contact graph may be much larger than the 
RAAG itself is familiar from the case of the extension graph, where 
Kim--Koberda prove that the automorphism of the extension graph is 
often uncountable \cite[Theorem~66]{KimKoberda:curve_graph}.

\begin{exmp} In the cube complex pictured in Figure~\ref{fig:aut_example}
    one sees that the automorphism group 
    of the cube complex is virtually cyclic, and in particular, pairs 
    of valence zero 
    vertices, such as those labelled $a,b$ can not be 
    permuted willy-nilly. Whereas in the contact graph, the pair of 
    simplices (labelled $a', b'$), one for each hyperplane separating off 
    such a vertex, can be swapped arbitrarily; thus the contact graph 
    contains a $(\mathbb Z/2\integers)^{\infty}$ subgroup.
    
    \begin{figure}[h]\label{fig:aut_example}
    \begin{overpic}[width=0.8\textwidth]{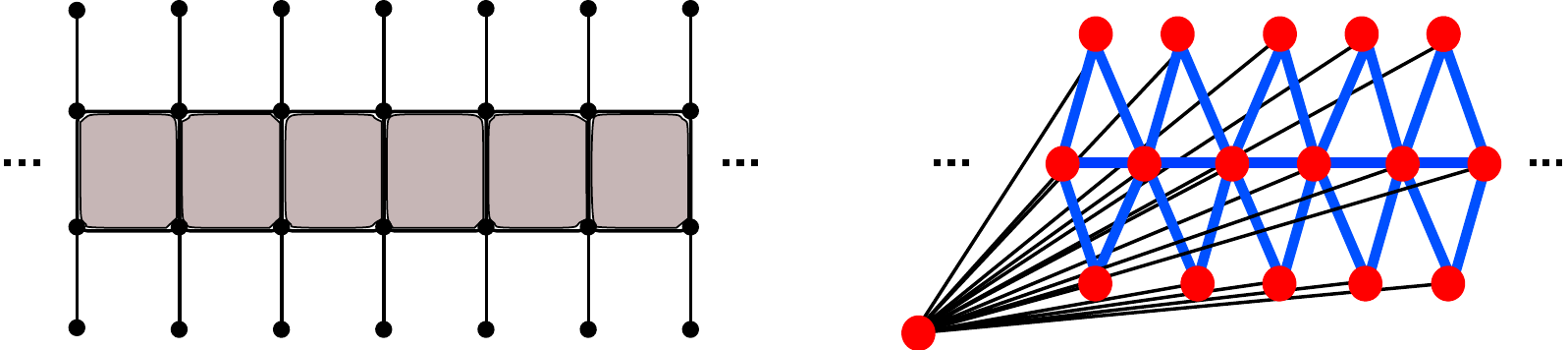}
      \put(33,21){$a$}
      \put(33,-1){$b$}
      \put(88,21){$a'$}
      \put(88,-1){$b'$}
      \put(23,-3){$\cuco X$}
      \put(77,-3){$\fontact \cuco X$}
    \end{overpic}
    \caption{Cube complex drawn at left; its associated contact 
    graph at right.}
    \end{figure}
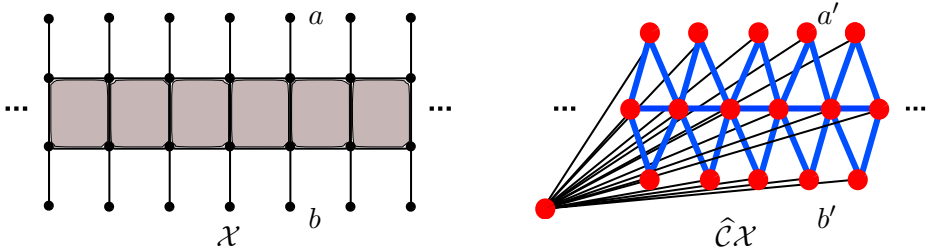
\end{exmp}

\part{Factor systems, projections, and the distance formula}\label{part:distance_formula}
Throughout, $\cuco X$ is a CAT(0) cube complex.  The setting in which we will work is established in Definition~\ref{defn:factor_system} and forces $\cuco X$ to be uniformly locally finite.  

\section{Factored contact graphs}\label{sec:projection_targets}

\subsection{Factor systems}\label{subsec:hypercarrier}
\begin{defn}[Factor system]\label{defn:factor_system}
A set of subcomplexes, $\factorsup$, which satisfies the 
following is called a \emph{factor system} in $\cuco X$:
\begin{enumerate}
 \item\label{item:fs_0}  $\cuco X\in\factorsup$. 
 \item\label{item:fs_convex} Each $F\in\factorsup$ is a nonempty convex
 subcomplex of $\cuco X$.  
 \item\label{item:fs_local_finite} There exists $\Delta\geq 1$ such that for all $x\in\cuco X^{(0)}$, at most $\Delta$ elements of $\factorsup$ contain $x$.
 \item\label{item:fs_all_hyperplanes} Every nontrivial convex subcomplex parallel to a combinatorial hyperplane 
 of $\cuco X$ is in $\factorsup$. 
 \item\label{item:fs_closed_proj} There exists $\xi\geq0$ such that for all $F,F'\in\factorsup$, either $\gate_F(F')\in\factorsup$ or $\diam(\gate_F(F'))<\xi$.
 \end{enumerate}
\end{defn}

Convexity of each $F\in\factorsup$ implies that each $F$ is
a CAT(0) cube complex whose hyperplanes have the form $H\cap F$, where
$H$ is a hyperplane of $\cuco X$, and the map $H\cap F\mapsto H$
induces an injective graph homomorphism $\contact
F\hookrightarrow\contact\cuco X$ whose image is an induced subgraph,
which, by an abuse of notation, we also denote $\contact F$.

If $\factorsup$ is a factor system for $\cuco X$, then for each $F\in\factorsup$, the set $\factorsup_F=\{H\in\factorsup:H\subseteq F\}$ is a factor system in $F$ such that each point in $F$ lies in at most $\Delta-1$ elements of $\factorsup_F$.  (The distinction between $\factorsup_F$ and $\{H\cap F:H\in\factorsup\}$ is small: the latter consists of the former, together with some subcomplexes of diameter at most $\xi$.  It is mainly for convenience that we choose to work, everywhere, with $\factorsup_F$ rather than $\{H\cap F:H\in\factorsup\}$.)

\subsection{Examples of factor systems}\label{subsec:examples_of_fs}
\subsubsection{Special groups}\label{subsubsec:compact_special_factor_system}
Universal covers of special cube complexes, defined
in~\cite{HaglundWiseSpecial}, contain factor systems provided the co-special action has finitely many orbits of hyperplanes, as we shall 
show below.  This provides a
very large class of groups for which the distance formula from
Section~\ref{sec:distance_formula} holds.  Let us start by
studying universal covers of Salvetti complexes.

The following definition is tailored to the proof of Proposition \ref{prop:raag_factor_system}.

\begin{defn}\label{defn:rich_subgfs}
 Let $\Gamma$ be a simplicial graph. A collection $\mathcal R$ of subgraphs of $\Gamma$ is \emph{rich} if
 \begin{enumerate}
  \item $\Gamma\in\mathcal R$,
  \item all links of vertices of $\Gamma$ are in $\mathcal R$, and
  \item if $A,B\in\mathcal R$ then $A\cap B\in\mathcal R$.
 \end{enumerate}
\end{defn}

The collection of all subgraphs of $\Gamma$ is a rich family.
Also, any graph admits a minimal rich family, consisting of $\Gamma$
together with all nonempty intersections of links of vertices.

\begin{prop}\label{prop:raag_factor_system}
 Let $\Gamma$ be a finite simplicial graph, $A(\Gamma)$ a 
 right-angled Artin group, and $\widetilde S_\Gamma$ the 
 universal cover of its 
 Salvetti complex.  Let $\mathcal
 R$ be a rich family of subgraphs of~$\Gamma$.
 
 Let $\factorsup$ be the ($A(\Gamma)$--invariant) collection of convex subcomplexes of $\widetilde S_\Gamma$ containing all lifts of the sub-Salvetti complexes $S_\Lambda$ of $S_\Gamma$, for all $\Lambda\in\mathcal R$. Then $\factorsup$ is a factor system.
 \end{prop}

 In other words, $\factorsup$ contains a subcomplex stabilized by each conjugate of $A(\Lambda)$, for each subgraph $\Lambda$ of $\Gamma$.
 
\begin{proof}[Proof of Proposition~\ref{prop:raag_factor_system}]
In the definition of a factor system, item~\eqref{item:fs_0} holds because $\Gamma\in \mathcal R$. Items~\eqref{item:fs_convex} and~\eqref{item:fs_local_finite} are clear. Item~\eqref{item:fs_all_hyperplanes} holds because combinatorial hyperplanes in $\widetilde S_\Gamma$ are exactly lifts of sub-Salvetti complexes of $S_\Gamma$ corresponding to links of vertices. More precisely, if a hyperplane $H$ is dual to a 1-cube corresponding to the generator $v\in \Gamma^{(0)}$, then the combinatorial hyperplanes on the two sides of $H$ are lifts of $S_{\link(v)}\subseteq S_{\Gamma}$.

The content of the proposition is hence that Item \eqref{item:fs_closed_proj} holds (for $\xi=0$). To simplify the notation, we will identify the 0-skeleton of $\widetilde S_\Gamma$ with $A(\Gamma)$. Under this identification, the 0-skeleton of each $F\in\factorsup$ corresponds to a coset of $A(\Lambda)<A(\Gamma)$ for some subgraph $\Lambda\in\mathcal R$, and each such coset is the 0-skeleton of some $F\in\factorsup$.

Let $F,F'\in\factorsup$, whose 0-skeleta are (possibly after applying an element of $A(\Gamma)$) $A(\Lambda_0)$ and $gA(\Lambda_1)$, for some subgraphs $\Lambda_0,\Lambda_1$ of $\Gamma$ and $g\in A(\Gamma)$. We can assume, using the action of some $g\in A(\Gamma_0)$, that $1\in \gate_F(F')$. Recall from Lemma~\ref{lem:parallel_product} and Lemma~\ref{lem:2.6} that $\gate_F(F')$ is in a natural way one of the factors in a product region $R$ of $\widetilde S_\Gamma$, the other factor being (naturally identified with) a possibly trivial geodesic $\gamma$ from $1$ to $g$, up to changing the choice of $g$ within the same coset of $A(\Lambda_1)$. Also, $\gate_F(F')\times\{1\}$ is contained in $F$ and $\gate_F(F')\times\{g\}$ is contained in $F'$. Let $\Lambda_2$ be the link in $\Lambda$ of the set of vertices of $\Lambda$ that label some 1-cube along $\gamma$. The following claim concludes the proof.

{\bf Claim:} The 0-skeleton of $\gate_F(F')$ is $A(\Lambda)$, where $\Lambda=\bigcap_{i=0,1,2} \Lambda_i\in \mathcal R$.

Let us first show $\gate_F(F')^{(0)}\subseteq A(\Lambda)$. Consider a geodesic $\delta$ joining $1$ to $h\in \gate_F(F')^{(0)}$. Any 1-cube $e$ of $\delta$ is contained in $F$ as well as parallel to a 1-cube of $F'$, which implies that the label $v$ of $e$ belongs to $\Lambda_0\cap \Lambda_1$. Let us now show that $e$ also belongs to $\Lambda_2$. Once we have done that, it is clear that $h$ can be written as a product of generators each belonging to $A(\Lambda)$.

Fix any 1-cube $e'$ of $\gamma$. The product region $R$ contains a square with two 1-cubes parallel to $e$ and two 1-cubes parallel to $e'$. This means that the labels of $e,e'$ commute and are distinct. As this holds for any $e'$, the label of $e$ belongs to $\Lambda_2$, as required.

We are left to show $A(\Lambda) \subseteq \gate_F(F')^{(0)}$. If $h\in A(\Lambda)$, there exists a geodesic $\delta$ from $1$ to $h$ whose 1-cubes are labeled by elements of $\gamma$. It is then easy to see that $\widetilde S_\Gamma$ (and in fact $R$) contains a product region naturally identified with $\delta\times \gamma$ with the property that $\delta\times\{1\}$ is contained in $F$ and $\delta\times\{g\}$ is contained in $F'$. In particular, $\delta$, and thus $h$, is contained in $\gate_F(F')$, as required.
\end{proof}

\begin{defn}[Induced factor system]\label{defn:induced_factor_system}
Let $\cuco X$ be a CAT(0) cube complex with a factor system $\factorsup$ and let $\cuco Y\subseteq\cuco X$ be a convex subcomplex.  The \emph{induced factor system} $\factorsup_{\cuco Y}$ be the set of nonempty subcomplexes of $\cuco Y$ of the form $F\cap\cuco Y, F\in\factorsup$.
\end{defn}

\begin{lem}\label{lem:convex_subcomplex_factor_system}
Let $\cuco X$ be a CAT(0) cube complex with a factor system $\factorsup$ and let $\cuco Y\subseteq\cuco X$ be a convex subcomplex.  Then $\factorsup_{\cuco Y}$ is a factor system in $\cuco Y$.
\end{lem}

\begin{proof}
Item~\eqref{item:fs_0} of Definition~\ref{defn:factor_system} holds since $\cuco X\cap\cuco Y=\cuco Y$.  Item~\eqref{item:fs_convex} follows since intersections of convex subcomplexes are convex.  Item~\eqref{item:fs_local_finite} follows since $\factorsup$ is a uniformly locally finite collection.  To verify Item~\eqref{item:fs_all_hyperplanes}, let $H^+$ be a combinatorial hyperplane of $\cuco Y$ and let $H$ be the (genuine) hyperplane whose carrier contains $H^+$ as one of its bounding copies.  By convexity of $\cuco Y$, the hyperplane $H$ has the form $W\cap\cuco Y$ for some hyperplane $W$ of $\cuco X$, and hence $H^+=W^+\cap\cuco Y$, where $W^+$ is one of the combinatorial hyperplanes bounding the carrier of $W$.  But $W^+\in\factorsup$, by item~\eqref{item:fs_all_hyperplanes}, so $W^+\cap\cuco Y=H^+\in\factorsup_{\cuco Y}$.  

It suffices to verify item~\eqref{item:fs_closed_proj}, namely that $\factorsup_{\cuco Y}$ is closed under (large) projection.  To that end, let $F,F'\in\factorsup$ and suppose that $\diam(\gate_{F\cap\cuco Y}(F'\cap\cuco Y))\geq\xi$, where $\xi$ is the constant associated to $\factorsup$ by Definition~\ref{defn:factor_system}.  Then, by item~\eqref{item:fs_closed_proj}, applied to $\factorsup$, we have $\gate_F(F')\in\factorsup$.  It thus suffices to show that $\gate_F(F')\cap\cuco Y=\gate_{F\cap\cuco Y}(F'\cap\cuco Y)$.  For convenience, let $A=F\cap\cuco Y,A'=F'\cap\cuco Y$.  Since $A\subseteq\cuco Y$, we have $\gate_A(A')\subseteq\cuco Y$.  Since $A\subseteq F$ and $A'\subseteq F'$, we have $\gate_A(A')\subseteq\gate_F(F')$.  Hence $\gate_A(A')\subseteq\gate_F(F')\cap\cuco Y$.  

Conversely, suppose that $x\in\gate_F(F')\cap\cuco Y$.  Then $x\in A$, since $x\in F\cap\cuco Y$.  Since $x\in\gate_F(F')$, there exists $x'\in F'$ such that a hyperplane $H$ separates $x'$ from $F$ if and only if $H$ separates $x'$ from $x$.  Let $z=\gate_{\cuco Y}(x')$.  Note that convexity of $\cuco Y$ is  used here to make $z$ well-defined.  Let $V$ be a hyperplane separating $z$ from $x$.  Then either $V$ separates $x'$ from $x$, and hence from $F$, whence $V$ separates $z$ from $A$, or $V$ separates $x,x'$ from $z$.  Suppose the latter.  Since $V$ separates $x'$ from $z$, it must separate $x'$ from $Y$ since $z=\gate_{\cuco Y}(x')$.  But then $V$ cannot cross $\cuco Y$, and hence cannot separate $z\in\cuco Y$ from $x\in\cuco Y$, a contradiction.  Thus $V$ separates $z$ from $x$ if and only if $V$ separates $z$ from $A$, so $x=\gate_A(z)$.  It remains to show that $z\in A'$, but this follows from Lemma~\ref{lem:convex_intersect}.       
\end{proof}

\begin{lem}\label{lem:convex_intersect}
Let $\cuco X$ be a CAT(0) cube complex and let $\cuco Y,\cuco Z$ be convex subcomplexes, with $A=\cuco Y\cap\cuco Z$.  Then for all $x\in\cuco Z$, we have $\gate_{\cuco Y}(x)\in A$.
\end{lem}

\begin{proof}
Let $y=\gate_{\cuco Y}(x)$, let $a=\gate_A(x)$, and let $m$ be the median of $x,a$, and $z$.  Then $m\in\cuco Z$ since it lies on a geodesic from $x$ to $a$, and $x,a\in\cuco Z$, and $\cuco Z$ is convex.  Similarly, $m$ lies on a geodesic from $y$ to $a$, and thus $m\in\cuco Y$.  It follows that $m\in A$, whence $m=a$ since $\dist_{\cuco X}(x,m)\leq\dist_{\cuco X}(x,a)$.  But then $\dist_{\cuco X}(x,a)\leq\dist_{\cuco X}(x,y)$, so $y=a$. 
\end{proof}

If $C$ is a special cube complex then $C$ admits a local 
isometry into some Salvetti complex by~\cite[Theorem~1.1]{HaglundWiseSpecial}, and this is the Salvetti complex of a finitely-generated right-angled Artin group when $C$ has finitely many immersed hyperplanes (e.g., when $C$ is compact special). Such a local isometry lifts to a convex embedding at the level of universal covers, whence Proposition~\ref{prop:raag_factor_system} and Lemma~\ref{lem:convex_subcomplex_factor_system} immediately imply that the universal cover of $C$ admits a factor system.  (Note that finite generation of the right-angled Artin group is important, since otherwise the universal cover of the Salvetti complex does not have a factor system because each $0$--cube is contained in infinitely many combinatorial hyperplanes in this case.)  We will now describe such factor systems. 

\begin{defn}
 Let $C$ be a special cube complex whose set $\mathcal H$ of immersed hyperplanes is finite. For every $\mathcal A \subseteq \mathcal H$ and 1-cubes $e,e'$ of $C$, write $e\sim_{\mathcal A} e'$ if there is a path $e=e_0\dots e_n=e'$ in the 1-skeleton of $C$ so that each $e_i$ is dual to some immersed hyperplane from $\mathcal A$. Let $C_{\mathcal A}$ be the collection of full subcomplexes of $C$ whose 1-skeleton is an equivalence class of 1-cubes with respect to the equivalence relation $\sim_{\mathcal A}$.
\end{defn}

Notice that each $D\in C_{\mathcal A}$ is locally convex.

\begin{cor}[Factor systems for special groups]\label{cor:compat_special_factor_system}
Let $\cuco X$ be the universal cover of a special cube complex, $C$,  
with finitely many immersed hyperplanes, $\mathcal H$. Then $\cuco X$ admits a
factor system: the collection of all lifts of
subcomplexes in $\bigcup_{\mathcal A \subseteq \mathcal H} C_{\mathcal
A}$ is a factor system for $\cuco X$.
\end{cor}

\begin{proof}
Let $\Gamma$ be the \emph{crossing graph} of $C$, which has a vertex for each immersed hyperplane, with two vertices adjacent if the corresponding immersed hyperplanes have nonempty intersection. Then for each immersed hyperplane $H$ of $C$ there is a corresponding 1-cube $e_H$ in $S_\Gamma$. In~\cite[Theorem~1.1]{HaglundWiseSpecial} it is shown that there is a local isometry $\phi\colon  C\to S_\Gamma$ so that if the 1-cube $e$ is dual to the immersed hyperplane $H$ then it gets mapped isometrically to $e_H$.

The local isometry $\phi$ lifts to a convex embedding 
$\tilde{\phi}\colon \cuco X\rightarrow\widetilde S_\Gamma$ 
(see e.g.~\cite[Lemma~3.12]{WiseNotes}), so in view of Proposition~\ref{prop:raag_factor_system} and Lemma~\ref{lem:convex_subcomplex_factor_system} there is a factor system on $\cuco X$ consisting of all preimages of elements of the factor system $\factorsup$ for $\widetilde S_\Gamma$, where $\factorsup$ is the factor system associated to the collection of all subgraphs of $\Gamma$ described in Proposition~\ref{prop:raag_factor_system}.

From now on we identify the 0-skeleton of $\widetilde S_\Gamma$ with $A(\Gamma)$, and regard the 1-cubes of $\widetilde S_\Gamma$ as labeled by an immersed hyperplane of $C$ (they are naturally labeled by vertices of $\Gamma$, which are immersed hyperplanes of $C$).

Let us consider an element of $F\in\factorsup$ that intersects $\tilde{\phi}(\cuco X)$ in, say, $g\in A(\Gamma)$. We want to show that $F'=F\cap \tilde{\phi}(\cuco X)$ is the image via $\tilde{\phi}$ of the lift $\widetilde D$ of some $D\in \bigcup C_{\mathcal A}$.

The 0-skeleton of $F$ is $gA(\Lambda)$ for some subgraph $\Lambda$ of $\Gamma$ whose set of vertices will be denoted $\mathcal A$. By convexity of $F'$ we know that we can connect $g$ to any element of the 0-skeleton of $F'$ (that is to say $gA(\Lambda)\cap \tilde{\phi}(\cuco X)$) by a path in the 1-skeleton of $F'$ all whose 1-cubes are labeled by elements of $\mathcal A$.
On the other hand, if we can connect $g$ to, say, $h$ by a path in the 1-skeleton of $F'$ all whose 1-cubes are labeled by elements of $\mathcal A$, then $h\in gA(\Lambda)$. These facts easily imply that the 0-skeleton of $F'$ coincides with the 0-skeleton of $\tilde{\phi}(\widetilde D)$ for some $D\in C_{\mathcal A}$, which in turn implies $\tilde{\phi}(\widetilde D)=F\cap \tilde{\phi}(\cuco X)$.

It can be similarly shown that for any $D\in \bigcup_{\mathcal A\subseteq \mathcal H} C_{\mathcal A}$, we have that $\phi(\widetilde D)$ is of the form $F\cap \tilde{\phi}(\cuco X)$ for some $F\in\factorsup$, and this concludes the proof.
\end{proof}

The same proof goes through to show the following more general 
version of the corollary:

\begin{cor}
 Let $\cuco X$ be the universal cover of the special cube complex $C$, whose set of immersed hyperplanes is finite, and let $\Gamma$ be the crossing graph of the immersed hyperplanes of $C$. If $\mathcal R$ is a rich collection of subgraphs of $\Gamma$, then the collection of all lifts of subcomplexes in $\bigcup_{\Lambda\in \mathcal R} C_{\Lambda^{(0)}}$ is a factor system for $\cuco X$.
\end{cor}

\subsubsection{A non-special example}\label{subsubsec:non_special_factor_system}

\begin{exmp}[]\label{exmp:self_strangling}
Let $S,T$ be wedges of finitely many circles and let $D$ be a compact, 2-dimensional, nonpositively-curved cube complex with the following properties:
\begin{enumerate}
 \item $D^{(1)}=S\cup T$.
 \item The universal cover $\widetilde D$ of $D$ is $\widetilde S\times\widetilde T$.
  \item $H=\pi_1D$ has no nontrivial finite quotient.
\end{enumerate}
Such $D$ were constructed by Burger-Mozes~\cite{BurgerMozes}, whose complex actually has simple fundamental group, and by Wise~\cite{Wise:CSC}.  Let $\tilde\alpha\to\widetilde S$ and $\tilde\beta\rightarrow\widetilde T$ be nontrivial combinatorial geodesics mapping to immersed closed combinatorial paths $\alpha,\beta\rightarrow D$.  Since each of $\widetilde S,\widetilde T$ is convex in $\widetilde S\times\widetilde T$, each of the maps $\alpha\rightarrow D,\beta\rightarrow D$ is a local isometry of cube complexes.  Hence $[\alpha],[\beta]\in H-\{1\}$ respectively lie in the stabilizers of the hyperplanes $\widetilde S\times\{o\},\{o'\}\times\widetilde T$ of $\widetilde S\times\widetilde T$, where $(o',o)$ is a basepoint chosen so that $o',o$ are arbitrary midpoints of 1-cubes of $\widetilde S,\widetilde T$.  Suppose moreover that $|\alpha|=|\beta|=r\geq 1$.

Let $R=[0,r]\times[0,2]$.  Regard $R$ as a 2-dimensional non-positively-curved cube complex in the obvious way, and form a complex $Z_0$ from $R\sqcup D$ by identifying $\{0\}\times[0,2]$ with $\{r\}\times[0,2]$ and identifying the images of $[0,r]\times\{0\}$ and $[0,r]\times\{2\}$ with $\alpha$ and $\beta$ respectively.  Since $\alpha$ and $\beta$ are locally convex in both $R$ and $D$, the complex $Z_0$ is nonpositively-curved, and its fundamental group is $H*_{\langle[\alpha]\rangle^t=\langle[\beta]\rangle}$. 

Let $A=[0,1]\times[0,r]$ with the obvious cubical structure, and attach $A$ to $Z_0$ by identifying $\{0\}\times[0,r]$ with the ``meridian'' in $R$, i.e., the image of $[0,r]\times\{1\}$.  Finally, attach a square by gluing a length-2 subpath of its boundary path to the image of $[0,1]\times\{0\}\cup[0,1]\times\{r\}\subset A$ in $A\cup Z_0$.  The resulting complex is $Z$.  Since $C$ is locally convex in $Z_0$ and $\{0\}\times[0,r]$ is locally convex in $A$, the complex $Z$ is a nonpositively-curved cube complex.  See Figure~\ref{fig:non_special_fs}.

\begin{figure}[h]
\begin{overpic}[width=0.75\textwidth]{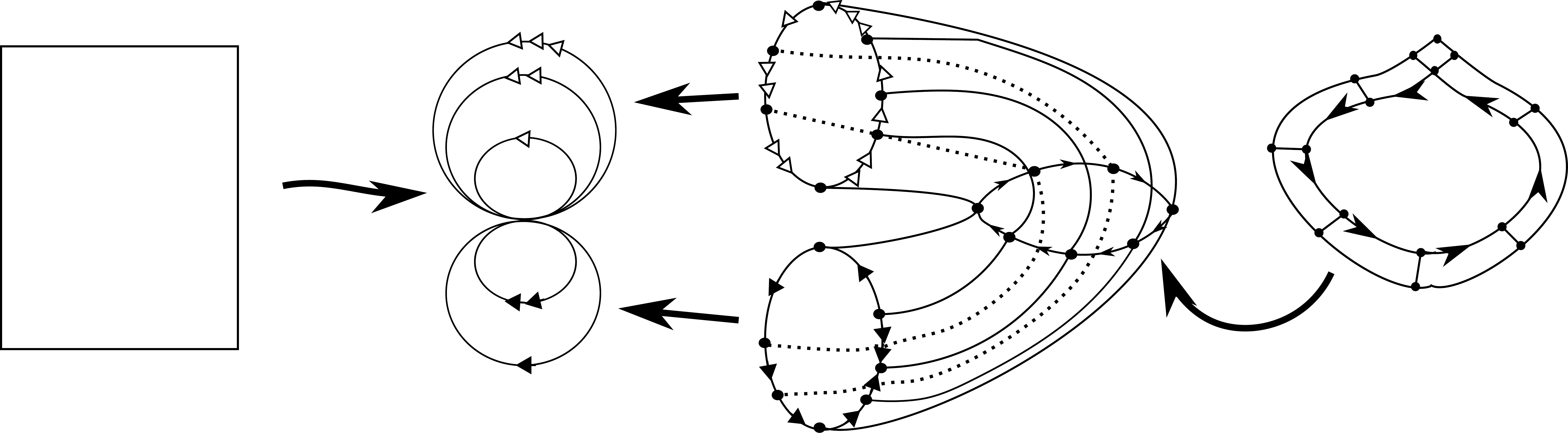}
 \put(90,5){$A$}
 \put(70,3){$R$}
 \put(33,1){$S$}
 \put(33,26){$T$}
 \put(5,15){$D^{(2)}$}
\end{overpic}
\caption{A complex $Z$ with $r=7$.  The square at left represents the $2$-skeleton of $D$.  Note that $A$ is the carrier of a self-intersectign hyperplane, $A$ is attached to $R$ by identifying the ``inner'' boundary component with the waist of $R$, and $R$ is attached to $S\cup T$ by a local isometry of its boundary.}\label{fig:non_special_fs}
\end{figure}
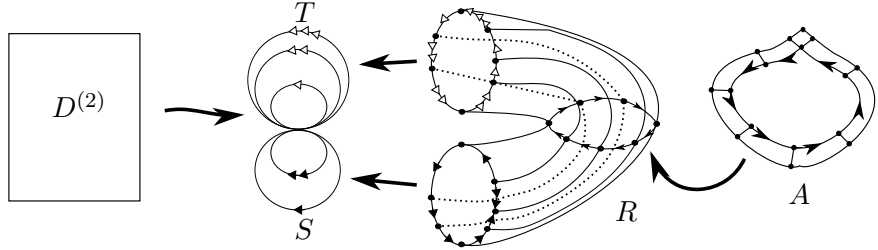

\end{exmp}

The following properties of $Z$ will be needed in Section~\ref{subsec:coloring}.

\begin{prop}\label{prop:Z_factor_system}
The universal cover $\widetilde Z$ of $Z$ has a factor system.
\end{prop}

\begin{proof}
Let $Z'$ be constructed exactly as in the construction of $Z$ in
Example~\ref{exmp:self_strangling}, except beginning with $S\times T$
instead of with the complex $D$.  The 1-skeleta of $D$ and $S\times T$
are identical, so the paths $\alpha,\beta$ exist in $S\times T$,
rendering this construction possible.  It is easily verified that the
universal cover of $Z'$ is isomorphic to $\widetilde Z$ and that $Z'$
has a finite-sheeted special cover.  The claim follows from
Corollary~\ref{cor:compat_special_factor_system}.
\end{proof}

\begin{prop}\label{prop:D_cross_translates}
Let $G\leq_{f.i.}\pi_1Z$.  Then there exists a hyperplane $H$ of $\widetilde D$ and $g\in G$ such that $gH$ crosses $H$.
\end{prop}

\begin{proof}
Let $\overline W\rightarrow Z$ be the immersed hyperplane dual to the
image of $[0,1]\times\{0\}\subset A$ in $Z$.  Then $\overline W$
self-crosses.  Moreover, if $\widehat Z\rightarrow Z$ is a
finite-sheeted cover such that some component $\widehat C$ of the
preimage of the meridian $C$ has the property that $\widehat
C\rightarrow C$ is bijective, then some lift of $W$ to $\widehat Z$ is
again a self-crossing immersed hyperplane.  But $C$ is homotopic into
$D$, whose fundamental group has no finite quotients.  Hence
\emph{any} finite cover of $Z$ has a self-crossing immersed
hyperplane.
\end{proof}

Having shown that there are many interesting groups with factor 
systems, below we show the utility of these systems. First 
though, we propose  
the following question to which we expect a positive answer; we believe that a proof 
of this will require significant work. 

\begin{question}\label{question:geometric_fs}
Does every CAT(0) cube complex which admits a proper,
cocompact group action contain a factor system?
\end{question}

\subsection{Factored contact graphs}\label{subsec:factored_contact_graphs}

\begin{defn}[Factored contact graphs, projection to the factored contact graph]\label{defn:factored_contact_graphs}Let $\factorsup$ be a factor system for $\cuco X$.  For each convex
subcomplex $F\subseteq\cuco X$, the \emph{factored contact graph}
$\fontact F$ is obtained from $\contact F$ as follows.  

Let $F'\in\factorsup_F-\{F\}$ and suppose that $F'$ is either parallel to a combinatorial hyperplane or has diameter $\geq\xi$, and let $[F']$ be its parallelism class.  We add a vertex $v_{F'}$ to $\contact F$, \emph{corresponding to this parallelism class $[F']$}, and join $v_{F'}$ by an edge
to each vertex of $\contact F'\subset\contact F$, i.e.,  
each newly added vertex 
$v_{F'}$ gets connected to each
vertex of $\contact F$ which corresponds to a hyperplane of $F$ that intersect $F'$.  We repeat this for each such $F'\in\factorsup_F-\{F\}$.  We emphasize that $\fontact F$ consists of $\contact F$, together with a cone-vertex for each \emph{parallelism class} of subcomplexes in $\factorsup_F-\{F\}$ that are parallel to combinatorial hyperplanes or have diameter $\geq\xi$ (or both).  Also, observe that $\contact F$ is an induced subgraph of $\fontact F$.

%

The \emph{projection of $\cuco X$ to $\fontact F$} is the map $\pi_F\colon \cuco X\rightarrow 2^{\fontact F^{(0)}}$ obtained by composing the projection to $\contact F$ given in Definition~\ref{defn:proj_to_contact} with the inclusion $\contact F\hookrightarrow\fontact F$.  Recall that this map assigns to each point of $\cuco X$ a clique in $\fontact F$ consisting of the vertices corresponding to the hyperplanes whose carriers contain the given point.
\end{defn}

\begin{rem}
When $F\in\factorsup$ is a single 0-cube, $\contact F=\emptyset$, so $\fontact F=\emptyset$.  Hence $|2^{\fontact F^{(0)}}|=1$ and $\pi_F$ is defined in the obvious way.  
When the constant $\xi$ from Definition~\ref{defn:factor_system} is at least $1$, no $F\in\factorsup$ is a single point.
\end{rem}
%

\begin{lem}\label{lem:subssystem_invariance}
Let $F,F'$ be parallel.  Then there is a bijection $f\colon \factorsup_F\rightarrow\factorsup_{F'}$ so that $f(H)$ is parallel to $H$ for all $H\in\factorsup_F$.
\end{lem}

\begin{proof}
If $F''\in\factorsup$ is contained in $F$, then there is a parallel copy $F'''$ of $F''$ in $F'$. Provided either $F''$ is a combinatorial hyperplane or has diameter at least $\xi$, we have $F'''\in\factorsup$.  The existence of $f$ now follows easily.
\end{proof}

The next lemma is immediate.

\begin{lem}\label{lem:projection_parallel}
Let $F\in\factorsup$ and let $F'$ be a convex subcomplex parallel to $F$.  Then, $\contact F=\contact F'$ and the bijection $f:\factorsup_F\to\factorsup_{F'}$ from Lemma~\ref{lem:subssystem_invariance} descends to a bijection of parallelism classes inducing an isomorphism $\fontact F\to\fontact F'$ that extends the identity $\contact F\to\contact F'$. Moreover, $\pi_F=\pi_{F'}$.
\end{lem}

\begin{rem}[Crossing, osculation, and coning]\label{rem:factored_in_standard_example}
In the case where $\factorsup$ is the factor system obtained by 
closing the set of subcomplexes parallel to hyperplanes under 
diameter-$\geq\xi$ projections, then $\fontact\cuco X$ is quasi-isometric to $\contact\cuco X$.  This is because we are coning off contact graphs of intersections of hyperplane-carriers, and these subgraphs are already contained in vertex-links in $\contact\cuco X$.  More generally, let $H_I=\cap_{i\in I}H_i^+$ be an intersection of combinatorial hyperplanes.  The hyperplanes of $H_I$ have the form $W\cap H_I$, where $W$ is a hyperplane that \emph{crosses} each of the hyperplanes $H_i$.  Coning off $\contact (W^\pm\cap H_I)$ in $\contact H_I$ does not affect the quasi-isometry type, since $\contact(W^\pm\cap H_I)$ is the link of a vertex of $\contact H_I$.  It is when $W$ is disjoint from $H_I$ but exactly one of $W^\pm$ is not --- i.e., precisely when $W$ \emph{osculates with each $H_i$} --- that coning off $\contact(W^+\cap H_I)\subset\contact H_I$ affects the geometry.\end{rem}

The crucial property of projections to factored contact graphs is:

\begin{lem}[Bounded projections]\label{lem:bounded_projections}
Let $\factorsup$ be a factor system for $\cuco X$ and let $\xi$ be the constant from Definition~\ref{defn:factor_system}.  Let $F,F'\in\factorsup$.  Then one of the following holds:
\begin{enumerate}
 \item\label{item:bp_small_proj}  $\diam_{\fontact F}(\pi_F(F'))<\xi+2$.
 \item\label{item:bp_sep_cone}  $F$ is parallel to a subcomplex of $F'$.
\end{enumerate}
\end{lem}

\begin{proof}
If $\diam(\gate_F(F'))\geq\xi$, then either $\pi_F(F')$ is coned off in $\fontact F$ or $\gate_F(F')=F$.
\end{proof}

We also note the following version of Proposition~\ref{prop:bgiI} for factored contact graphs:

\begin{prop}[Bounded Geodesic Image II: The Factoring]\label{prop:BGIIITF}
Let $F\in\factorsup$ and let $\gamma\rightarrow F$ be a combinatorial geodesic.  Suppose $U\in\factorsup_F$ and that $\dist_{\fontact F}(\pi_F(\gamma),\fontact U)\geq1$.  Then $\gate_U(\gamma)$ consists of a single $0$--cube.  Hence $\pi_U(\gamma)$ is a clique.
\end{prop}

\begin{proof}
Suppose that $\gate_U(\gamma)$ contains distinct 0-cubes $x,y$ and let $H$ be a hyperplane separating them.  Then $H$ crosses $U$ and thus corresponds to a vertex of $\contact U\subset\fontact U\subseteq\fontact F$.  On the other hand, $H$ cannot separate any 0-cube of $\gamma$ from $U$ and thus separates the endpoints of $\gamma$.  Hence $H$ is a vertex of $\pi_F(\gamma)$.
\end{proof}

\subsection{Hierarchy paths revisited}\label{subsec:hierarchy_paths_revisited}
Let $\factorsup$ be a factor system for $\cuco X$ and let $F\in\factorsup$.  The vertices of $\fontact F$ are naturally associated to subcomplexes of $\cuco X$: to each vertex $V$ of $\contact F$ ($\subset\fontact F$), we associate one of the two combinatorial hyperplanes $V^\pm$ bounding $\neb(V)$.  The remaining vertices are cone points corresponding to parallelism classes of subcomplexes in $\factorsup_F$.
A path of $\fontact F$ is a sequence $(v_0,\ldots,v_n)$ of vertices with consecutive vertices adjacent.  A combinatorial path $\gamma\rightarrow F$ is \emph{carried} by the path $(v_0,\ldots,v_n)$ of $\fontact F$ if $\gamma=\gamma_0e_0\gamma_1e_1\cdots\gamma_ne_n$, where for $0\leq i\leq n$, we have $\gamma_i\rightarrow T_i$ with $T_i$ associated to $v_i$ and $|e_i|\leq 1$.  In this situation, we say that the sequence $T_0,\ldots,T_n$ \emph{carries} $\gamma$ and \emph{represents} $(v_0,\ldots,v_n)$.

\begin{prop}\label{prop:path_in_fontact}
Let $v_0,v_1,\ldots,v_r$ be a path in $\fontact\cuco X$ and let $x\in T_0,y\in T_r$ be 0-cubes, where $T_0,T_r$ are subcomplexes associated to $v_0,v_r$.  Then for $1\leq i\leq r-1$, we can choose for each vertex $v_i$ an associated subcomplex $T_i$ so that $T_0,\ldots,T_r$ carries a path in $\cuco X$ joining $x$ to $y$.\end{prop}

\begin{proof}
We argue by induction on $r$.  If $r=0$, then the claim follows by
path-connectedness of $T_0$.  Now let $r\geq 1$.  Suppose first that
$v_r$ is a cone-vertex, so that $v_{r-1}$ is a hyperplane-vertex, and
let $V$ be the corresponding hyperplane.
Since $V$ crosses $T_r$,
either combinatorial hyperplane $T_{r-1}$ bounding the carrier of $V$
intersects $T_r$.  There exists a $0$--cube $y'\in T_{r-1}\cap T_r$.  On the other hand, if $v_r$
is a hyperplane-vertex, then $T_r$ is a specified combinatorial
hyperplane parallel to a hyperplane crossing each subcomplex $T_{r-1}$
associated to $v_{r-1}$, if $v_{r-1}$ is a cone-vertex.  If $v_{r-1}$
is also a hyperplane-vertex, then at least one of the possible
choices of $T_{r-1}$ is a combinatorial hyperplane intersecting the
combinatorial hyperplane $T_r$.  Again, we have a 0-cube $y'\in
T_r\cap T_{r-1}$.  In either case, by induction, we can choose
$T_1,\ldots,T_{r-2}$ to carry a path joining $x$ to $y'$, which we
concatenate with a path in $T_r$ joining $y'$ to $y$.  The path
$e_{r-2}$ is nontrivial if $v_{r-1}$ corresponds to a hyperplane
separating $T_r$ from every possible choice of $T_{r-2}$.
\end{proof}

\begin{rem}[Explicit description of paths carried by geodesics]\label{rem:explicit_hier_paths}
In the proof of Proposition~\ref{prop:path_in_fontact}, we used the
fact that a path $v_0,v_1,\ldots,v_r$ in $\fontact F$ has the property
that, if $v_i$ is a cone-vertex, associated to some
$U\in\factorsup_F$, then $v_{i\pm1}$ are hyperplane vertices,
associated to hyperplanes $H_{i\pm1}$ that cross $U$, so that all four
combinatorial hyperplanes $H_{i\pm1}^{\pm}$ intersect $U$.  Hence we
can and shall \emph{always} assume that $|e_{i\pm1}|=0$ when $v_i$ is
a cone-vertex, see Figure~\ref{fig:hier_path_explicit}.  This 
enables us to use the proof of Proposition~\ref{prop:hierpath}
verbatim in the proof of
Proposition~\ref{prop:hier_revisited}. 
\end{rem}

\begin{figure}[h]
\includegraphics[width=0.65\textwidth]{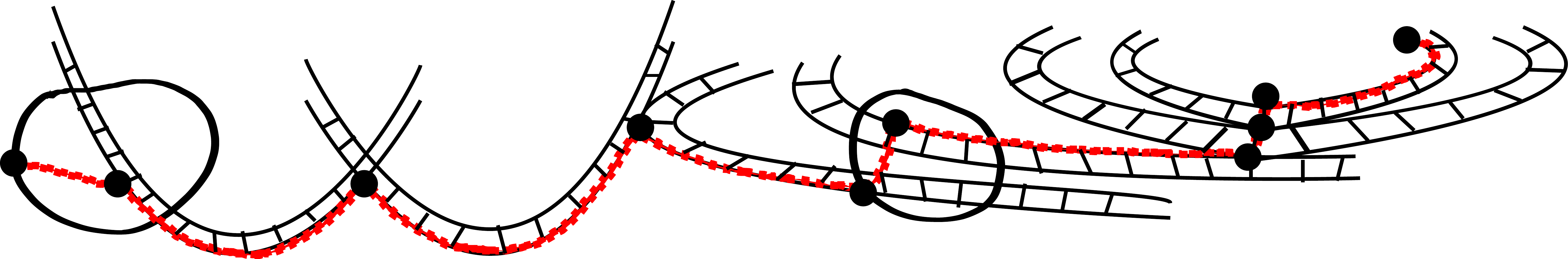}\\
\caption{The two round regions are elements of $\factorsup$; the ladders represent carriers of hyperplanes.  The dotted path is carried by a geodesic in the factored contact graph.  Transitions between various $\gamma_i$ and $e_i$ are indicated by bold vertices.  Notice that there are numerous such paths, but we choose paths with $e_i$ trivial whenever possible.}\label{fig:hier_path_explicit}
\end{figure}

In view of the previous proposition, we will use the same notation for a vertex of $\fontact F$ as for some representative subcomplex.  A \emph{hierarchy path} is a geodesic of $F$ that is carried by a geodesic of $\fontact F$.

\begin{prop}\label{prop:hier_revisited}
Let $F\in\factorsup$.  Then any 0-cubes $x,y\in F$ are joined by a hierarchy path. \end{prop}

\begin{proof}
Let $\tau$ be a combinatorial geodesic of $\cuco X$ joining $x$ to
$y$; since $F$ is convex, $\tau\subseteq F$.  Let $T\in\factorsup_F$
contain $x$ (such $T$ exists because every $0$--cube of $F$ is contained in a combinatorial hyperplane of $F$) and let $T'\in\factorsup_F$ contain $y$.  Let
$T=T_0,T_1,\ldots,T_n=T'$ be a geodesic of $\fontact F$.  Then we can
choose the $T_i$ in their parallelism classes so that
$\neb_1(T_i)\cap\neb_1(T_{i+1})\neq\emptyset$ for $0\leq i\leq n-1$.
Using Proposition~\ref{prop:path_in_fontact}, for each $i$, let
$\gamma_i\rightarrow T_i$ be a geodesic segment, chosen so that
$\gamma_1e_1\ldots\gamma_ne_n$ is a piecewise-geodesic path joining
$x$ to $y$, with each $|e_i|\leq 1$ and $|e_i|=0$ except possibly if
$\gamma_i,\gamma_{i+1}$ lie in disjoint combinatorial hyperplanes
representing vertices of $\contact F$ (i.e., non-cone-vertices).  The
disc diagram argument from the proof of
Proposition~\ref{prop:hierpath} can now be repeated verbatim to
complete the proof, because $\factorsup_F$ satisfies
property~\eqref{item:fs_all_hyperplanes} from
Definition~\ref{defn:factor_system}.\end{proof}

\subsection{Factored contact graphs are quasi-trees}\label{subsec:factored_contact_graphs_are_quasi_trees}
In this section, we will use the following criterion of Manning, that determines if a geodesic metric space is quasi-isometric to a tree:

\begin{prop}[Bottleneck criterion; \cite{ManningPseudocharacters}]\label{prop:bottleneck}
Let $(Z,d)$ be a geodesic metric space.  Suppose that there exists $\delta\geq 0$ such that for all $x,y\in Z$, there exists $m=m(x,y)$ such that $d(x,m)=d(y,m)=\frac{1}{2}d(x,y)$ and every path joining $x$ to $y$ intersects the closed $\delta$--ball about $m$.  Then $(Z,d)$ is $(26\delta,16\delta)$--quasi-isometric to a tree.
\end{prop}

\begin{prop}\label{prop:fontact_quasi_tree}
Let $F\in\factorsup$.  Then $\fontact F$ is $(78,48)$--quasi-isometric to a simplicial tree.
\end{prop}

\begin{proof}
Denote by $\mathcal H$ the set of vertices of $\fontact F$
corresponding to hyperplanes of $F$ and by $\mathcal V$ the set of
cone-vertices.  Let $T,T'$ be complexes associated to vertices of
$\fontact F$ (regarded as representatives, to be chosen, of
parallelism classes of subcomplexes in $\factorsup_F$).  Let
$T=T_0,T_1,\ldots,T_n=T'$ be a geodesic of $\fontact F$ joining $T$ to
$T'$ and let $m$ be the midpoint of this geodesic.  Suppose, moreover,
that $T_0,\ldots,T_n$ carries a hierarchy path
$\gamma=\gamma_0\cdots\gamma_n$ such that each hyperplane intersecting
$\gamma$ separates $T$ from $T'$; indeed we can choose $\gamma$ to connect a point $p\in\gate_{T}(T')$ to its gate $\gate_{T'}(p)$.  We may assume that $n\geq 4$, for
otherwise Proposition~\ref{prop:bottleneck} is satisfied by $m$ with
$\delta=\frac{3}{2}$.  Hence, since every vertex adjacent to an
element of $\mathcal V$ is in $\mathcal H$, there exists $i$ so that
$T_i$ is a combinatorial hyperplane and $\dist_{\fontact F}(T_i,m)\leq
1$.  Let $T_i'$ be the hyperplane such that $T_i$ is in the image of
$T_i'\times\{\pm1\}$ under $T_i'\times[-1,1]\cong
\neb(T'_i)\hookrightarrow F$.  Since every hyperplane crossing $\gamma_i$ separates $T$ from $T'$, we have that either $T'_i$ separates $T$ from $T'$, or
$T'_i$ crosses a hyperplane that separates $T$ from $T'$.  Hence, there is a
hyperplane $H$ such that $H$ separates $T,T'$ and $\dist_{\fontact
F}(H,m)\leq 2$.  Let $T=S_0,\ldots,S_k=T'$ be another path in
$\fontact F$ joining $T$ to $T'$.  Then, there exists $j$ such that
either $S_j$ is a combinatorial copy of $H$, or $H$ crosses $S_j$.
Thus $\dist_{\fontact F}(m,S_j)\leq 3$, and the claim follows from
Proposition~\ref{prop:bottleneck}.
\end{proof}

\section{The distance formula}\label{sec:distance_formula}

The main theorem of this section is:

\begin{thm}[Distance formula]\label{thm:distance_formula}
Let $\cuco X$ be a CAT(0) cube complex and let $\factorsup$ be a factor system.  Let $\factorseq$ contain exactly one representative of each parallelism class in $\factorsup$.  Then there exists $s_0\geq 0$ such that for all $s\geq s_0$, there are constants $K\geq 1,C\geq 0$ such that for all $x,y\in\cuco X^{(0)}$, $$\dist_{\cuco X}(x,y)\asymp_{_{K,C}}\sum_{F\in\factorseq}\ignore{\dist_{\fontact F}(\pi_F(x),\pi_F(y))}{s}.$$
\end{thm}

The rest of this section is devoted to proving Theorem~\ref{thm:distance_formula}.  Throughout, $\cuco X$ and $\factorsup$ are as in the statement of the theorem.  The constants $\xi\geq1,\Delta\geq1$ are those from Definition~\ref{defn:factor_system}.  For convenience, if $F\in\factorsup$ and $x,y\in\cuco X$, we write $\dist_{\fontact F}(x,y):=\dist_{\fontact F}(\pi_F(x),\pi_F(y))$.  

\subsection{Projection lemmas}\label{subsec:projection_lemmas}
\begin{lem}\label{lem:simple_projection}
Let $\cuco X$ be a CAT(0) cube complex, let $A\subseteq B\subseteq\cuco X$ be convex subcomplexes, and let $x,y\in\cuco X$ be 0-cubes.  Then $\dist_{\fontact A}(\pi_A(x),\pi_A(y))\leq\dist_{\cuco X}(\gate_A(x),\gate_A(y))\leq\dist_{\cuco X}(\gate_B(x),\gate_B(y))$.
\end{lem}

\begin{proof} 
The first inequality follows since projection to the factored contact
graph is distance non-increasing.
The second inequality is Lemma~\ref{lem:simple_gate}.   
\end{proof}

\begin{lem}\label{lem:length_comparison}
Let $x,y\in\cuco X^{(0)}$ and let $T_0,T_1,\ldots,T_r$ represent a geodesic in $\fontact\cuco X$ with $x\in T_0,y\in T_r$.  For $0\leq i\leq r$, let $x_i=\gate_{T_i}(x),y_i=\gate_{T_i}(y)$.  Then $$\dist_{\cuco X}(x,y)-\dist_{\fontact\cuco X}(x,y)\leq\sum_{i=0}^r\dist_{T_i}(x_i,y_i)\leq 3\dist_{\cuco X}(x,y).$$
\end{lem}

\begin{proof}
For $0\leq i\leq r$, let $\mathcal S_i$ be the set of hyperplanes
separating $x_i,y_i$, so that $|\mathcal S_i|=\dist_{\cuco
X}(x_i,y_i)=\dist_{T_i}(x_i,y_i)$.  For each $i$, observe that if the
hyperplane $U$ separates $x$ from $y$ and crosses $T_i$, then $U$
separates $x_i$ from $y_i$.  Otherwise, $U$ would separate $x_i$ from
$x$ (say), but the only such hyperplanes are those separating $x$ from
$T_i$, so this situation would contradict the fact that $U$ crosses $T_i$.
Hence the set of hyperplanes separating $x$ from $y$ and crossing some
$T_i$ is $\cup_{i=0}^r\mathcal S_i$.  Each hyperplane $H$ separating
$x,y$ either crosses some $T_i$, or there exists $i$ so that $\neb(H)$
contains $T_i$ in one of the bounding copies of $H$.  Hence
$$\dist_{\cuco X}(x,y)\leq\left|\cup_{i=0}^r\mathcal
S_i\right|+r\leq\sum_{i=0}^r|\mathcal S_i|+r,$$ which establishes the
first inequality.

For some $i\neq j$, let $U\in\mathcal S_i\cap\mathcal S_j$.  Then $|i-j|\leq 2$, since otherwise $T_i,U,T_j$ would provide a shortcut contradicting the fact that $\dist_{\fontact(\cuco X,\mathcal P)}(T_0,T_r)=r$.  Hence at most three elements of $\{\mathcal S_i\}_{i=0}^r$ contain $U$.  Thus $\sum_{i=0}^r|\mathcal S_i|\leq 3|\cup_{i=0}^r\mathcal S_i|$, and the second inequality follows.
\end{proof}

\begin{prop}[Large Link Lemma]\label{prop:bgiII} 
Let $T_0,T_1,\ldots,T_r$ be a geodesic in $\fontact\cuco X$ between 
0-cubes $x\in T_0$ and $y\in T_r$.  Let $F\in\factorsup-\{\cuco X\}$ have the property that $\dist_{\fontact F}(x,y)\geq 4\xi+10$.  Then there exists $i\in\{0,1,\ldots,r\}$ such that $F$ is parallel to some $F'\in\factorsup_{T_i}$.  Moreover, any geodesic of $\cuco X$ contained in $\cup_iT_i$ that joins $x,y$ passes through a subcomplex of $T_i$ parallel to $F'$.
\end{prop}

\begin{proof}
Recall that we can choose $T_i$ within its parallelism class for $1\leq i\leq r-1$ so that $\neb_1(T_i)\cap T_{i+1}\neq\emptyset$.
We first exhibit $J\leq 3$ and $i\leq r$ such that for some $j$ 
satisfying $0\leq j\leq
J$ we have $\diam_{\fontact F}(\pi_F(T_{i+j}))\geq\xi$ and for all 
$i'$ satisfying  $i'<i$ or $i'>i+J$ we have $\diam_{\fontact
F}(\pi_F(T_{i'}))=0$. 

Let $V$ be a hyperplane separating $\gate_{F}(x)$ from $\gate_{F}(y)$.
Then $V$ intersects $F$ and hence $V$ cannot separate $x$ or $y$ from
their gates in $F$.  Hence $V$ must separate $x$ from $y$; and thus $V$ intersects some $\neb_1(T_i)$.  Either $V$ separates $\gate_{T_i}(x)$ from $\gate_{T_i}(y)$, or $V$ is the unique hyperplane separating $T_i$ from $T_{i\pm1}$.  Let $\mathcal G$ be the set of hyperplanes $V$ of the former type.

Let $V\in\mathcal G$.  Then $V$ intersects $T_i$ for at least $1$ and at most $3$ values of $i$.  If some other hyperplane $V'\in\mathcal G$ separates $\gate_{F}(x)$ from
$\gate_{F}(y)$, then $V'$ separates $\gate_{T_j}(x)$ from
$\gate_{T_j}(y)$ for at least $1$ and at most $3$ values of $j$, and for any such $i,j$, we have $|i-j|\leq4$;
otherwise $T_i,V,F,V',T_j$ would be a shortcut from $T_i$ to $T_j$,
since $F\neq \cuco X$ represents a vertex in $\fontact\cuco X$.  Let
$i,i+1,\ldots,i+J$ be the indices for which $T_{i+j}$ is crossed by a
hyperplane separating $\gate_F(x)$ from $\gate_F(y)$;  equivalently, the indices so that $\gate_{T_k}(F)$ is trivial if $k<i$ or $k>i+J$.

The graph $\pi_F(\cup_{j=0}^J\gate_F(T_{i+j}))=\pi_F(\cup_{j=0}^r\gate_F(T_{i}))$ has diameter at least $4\xi+10$, since it contains $\pi_F(x),\pi_F(y)$,  
whence $\diam_{\fontact F}(\gate_F(T_{i+j}))\geq \xi$ for some $j$.  Indeed, there is at most one hyperplane separating $T_{i+j}$ from $T_{i+j+1}$, whence at most one hyperplane separates their gates in $F$.  We then have that $F$ is parallel to a subcomplex of $T_{i+j}$ by Lemma~\ref{lem:bounded_projections}.

Let $\gamma\subset\cup_kT_k$ be a geodesic joining $x,y$ and let $\gamma_{i+j}=\gate_{T_{i+j}}(\gamma)$. Then $\gamma_{i+j}$ has a subpath $\tau$ joining $\gate_{T_{i+j}}(\gate_F(x))$ to $\gate_{T_{i+j}}(\gate_F(y))$. The path $\tau$ is parallel to a path in $F$ and hence belongs to some parallel copy of $F$ in $T_{i+1}$.
\end{proof}

\begin{lem}\label{lem:old_3.15}

There exist $s_0\geq 0$ such that for all $x,y\in\cuco X^{(0)}$ and any hierarchy path $\gamma_0e_0\cdots\gamma_re_r$ joining $x$ to $y$ and carried by a geodesic $T_0,\ldots, T_r$ of $\fontact\cuco X$, one of the following holds for all $F\in\factorsup$:
\begin{enumerate}
 \item there exists $i\leq r$ and $F'\in\factorsup_{T_i}$ such that $F$ is parallel to $F'$;
 \item $\dist_{\fontact F}(x,y)<s_0$;
 \item $F=\cuco X$.
\end{enumerate}

\end{lem}

\begin{proof}
Let $s_0=4\xi+10$ and let $F\neq \cuco X$ be an element of $\factorsup$ with $\dist_{\fontact F}(x,y)\geq s_0$.  By Proposition~\ref{prop:bgiII}, there exists $i$ so that $F$ is parallel to some $F'\in\factorsup_{T_i}$, as required.
\end{proof}

\subsection{Proof of the distance formula}\label{subsec:df_proof}
We have now assembled all ingredients needed for:

\begin{proof}[Proof of Theorem~\ref{thm:distance_formula}]
For each $F\in\factorseq$, let $\factorseq_F$ consist of exactly one
element from each parallelism class in $\factorsup_F$.  We
will argue by induction on $\Delta$, which we recall is the maximal 
number of elements in the factor system which can contain any given vertex.\\

\textbf{Base case:} When $\Delta=1$, the fact that $\cuco X$ and each
combinatorial hyperplane belongs to $\factorsup$ ensures that $\cuco
X$ consists of a single $0$--cube, so we are summing over the empty
set of projections, and we are done.\\

\textbf{Induction hypothesis:} Assume $\Delta\geq 2$.  For each
$F\in\factorseq-\{\cuco X\}$, the set $\factorsup_F$ is a factor
system in $F$.  Every 0--cube of $F$ is contained in at most $\Delta-1$
elements of $\factorsup_F$, since each 0--cube of $F$ lies in at most
$\Delta$ elements of $\factorsup$, one of which is $\cuco X$ (which 
is not contained in~$F$).  We can
therefore assume, by induction on $\Delta$, that for all $s\geq s_0$, where $s_0\geq2$ is
the constant from Lemma~\ref{lem:old_3.15}, there exist $K'\geq1,C'\geq0$ so that for all
$F\in\factorseq-\{\cuco X\}$ and all $x,y\in\cuco X$ we have
\begin{equation*}\dist_{F}(\gate_F(x),\gate_F(y))\asymp_{_{K',C'}}\sum_{T\in\factorseq_F}\ignore{\dist_{\fontact T}(\gate_F(x),\gate_F(y))}{s}.
\end{equation*}
\\

\textbf{Choosing a hierarchy path:}  By Proposition~\ref{prop:hier_revisited}, there exists a hierarchy path $\gamma=\gamma_0e_0\cdots\gamma_re_r$ that joins $x$ to $y$ and is carried on a geodesic $T_0,\ldots,T_r$ in $\fontact\cuco X$ from $T_0$ to $T_r$.  For each $i$, let $x_i=\gate_{T_i}(x)$ and $y_i=\gate_{T_i}(y)$.  Recall that each $\gamma_i$ lies in $\neb_1(T_i)$.  At the cost of adding $1$ to our eventual multiplicative constant $K$, we can assume that each $e_i$ is trivial.\\

\textbf{Enumeration of the nonvanishing terms:}  Let $s\geq s_0$ and let $F\in\factorsup$ be such that $\dist_{\fontact F}(x,y)\geq s$.  By Lemma~\ref{lem:old_3.15}, either $F$ is parallel to some $F'\in\factorsup_{T_i}$ for some $i$, or $F=\cuco X$.

Let $\factorsup_{\{T_i\}}$ be the set of all $F\in\factorsup$ that are parallel to a proper subcomplex of some $T_i$, and let $\factorseq_{\{T_i\}}$ contain exactly one representative for each parallelism class of elements of $\factorsup_{\{T_i\}}$. Recalling that $\dist_{\fontact F}(x,y)=\dist_{\fontact F'}(x,y)$ when $F$ is parallel to $F'$, we get
\begin{equation*}\sum_{F\in\factorseq}\ignore{\dist_{\fontact F}(x,y)}{s}=\I+\II+\ignore{r}{s},
\end{equation*}
where $$\I=\sum_{i=0}^r\ignore{\dist_{\fontact T_i}(x,y)}{s}$$ and
$$\II=\sum_{F\in\factorseq_{\{T_i\}}}\ignore{\dist_{\fontact
F}(x,y)}{s}\asymp_{3,0}\sum_{i}\sum_{F\in\factorseq_{T_i}-\{T_i\}}\ignore{\dist_{\fontact
F}(x,y)}{s}.$$ The estimate of $\II$ follows since no factor is
parallel to both $T_i$ and $T_{i'}$ when $|i-i'|\geq3$.\\

\textbf{The upper bound:}  By the inductive hypothesis, $$\I+\II\asymp_{K',C'}\sum_{i=0}^r\dist_{T_i}(\gate_{T_i}(x),\gate_{T_i}(y)).$$

Hence, by Lemma~\ref{lem:length_comparison}, $$\frac{\dist_{\cuco X}(x,y)-r}{3K'}-C'\leq\I+\II,$$ and we obtain the upper bound
\begin{eqnarray*}
\dist_{\cuco X}(x,y)&\leq& 3K'(\I+\II+r)+3C'K'\\
&=&3K'\sum_{F\in\factorseq-\{\cuco X\}}\ignore{\dist_{\fontact F}(x,y)}{s}+3K'\dist_{\fontact\cuco X}(x,y)+3C'K'\\
&\asymp_{1,s}&\sum_{F\in\factorseq-\{\cuco X\}}\ignore{\dist_{\fontact F}(x,y)}{s}+3K'\ignore{\dist_{\fontact\cuco X}(x,y)}{s}+3C'K'.
\end{eqnarray*}

\textbf{The lower bound:}  We have that $\dist_{\cuco X}(x,y)\geq r$ and we have the estimate $\I+\II\leq3K'\dist_{\cuco X}(x,y)+C'r$, by the induction hypothesis and Lemma~\ref{lem:length_comparison}.  Hence $$\I+\II+\ignore{r}{s}\leq(3K'+C'+1)\dist_{\cuco X}(x,y),$$ and we are done.
\end{proof}

\section{Projection of parallelism classes}\label{sec:par_proj}
Let $\cuco X$ be CAT(0) cube complex with a factor system $\factorsup$, and let $\factorseq$ be the set of parallelism classes of elements of $\factorsup$.  We have previously defined projections from $\cuco X$ to factored contact graphs of elements of $\factorsup$, and noted that $\pi_F\colon \cuco X\rightarrow\fontact F$ is independent of the choice of representative of the parallelism class $[F]$.  However, care must be taken in order to define projections of parallelism classes:

\begin{prop}[Projections of parallelism classes are bounded or cover]\label{prop:projections}
Let $F',F''\in\factorsup$ with $F'$ parallel to $F''$. 
Then one of the following holds for each $F\in\factorsup$:
\begin{enumerate}
 \item\label{item:par_1} $\dist_{\fontact
 F}(\pi_F(F'),\pi_F(F''))\leq\xi+2$, where $\xi$ is a constant of the
 factor system;
 \item\label{item:par_2} there is a cubical isometric embedding
 $F'\times F\rightarrow\cuco X$ such that for some $f',f''\in F$, the
 subcomplexes $F',F''$ are the images of $F'\times \{f'\}, F'\times \{f''\}$,
 and for some $f\in F'$, the subcomplex $F$ is the image of $\{f\}\times
 F$.  In particular, $\cup_{F_1\in[F']}\pi_F(F_1)=\fontact F$.
\end{enumerate}
Moreover, whether assertion~\eqref{item:par_1} or~\eqref{item:par_2} holds for $F',F''$ with respect to $F$ is independent of the choice of parallelism class representative of $F$.
\end{prop}

\begin{proof}
By Lemma~\ref{lem:parallel_product}, there is a convex subcomplex $B$
and a cubical isometric embedding $F'\times B\rightarrow\cuco X$ with
the following properties: $B$ is the convex hull of a shortest
geodesic $[0,b]$ joining $F'$ to $F''$ and $F',F''$ are respectively
the images of $F'\times\{0\}$ and $F'\times\{b\}$, and in fact $[0,b]$ joins $\gate_{F'}(F'')$ to $\gate_{F''}(F')$ (the fact that $[0,b]$ is shortest and joins $\gate_{F'}(F'')$ to $\gate_{F''}(F')$ follows from the fact that it only crosses the hyperplanes that separate $F'$ from $F''$, which is part of Lemma~\ref{lem:parallel_product}).  Since
$B$ is, at minimum, contained in a combinatorial hyperplane (indeed,
cubical isometric embeddings take hyperplanes to hyperplanes), it is
contained in some minimal $Y\in\factorsup$.  If $b\leq\xi$, then the
first assertion holds.  Hence suppose $b\geq\xi$.  Then either
$\pi_F(Y)$ has diameter at most $\xi+2$, in which case the first
assertion again holds, or by Lemma~\ref{lem:bounded_projections}, $F$
is parallel to a subcomplex $F_1$ of $Y$ containing $B$.  By
minimality of $Y$, we have $F_1=Y$. Similarly, for all $f\in F'$, we have a
parallel copy $ \{f\}\times F_1$ of $F$ containing $\{f\}\times B$.
Item~\eqref{item:par_2} follows.
\end{proof}

\begin{defn}[Orthogonal]\label{defn:orthogonal}
 Let $[F_1],[F_2]\in\factorseq$.  If there are $F_1',F_1''\in[F_1]$ 
such that Proposition~\ref{prop:projections}.\eqref{item:par_2} holds 
for $F_1',F_1''$ projected to $F_2$, then $[F_1]$ and $[F_2]$ (and 
$F_1$ and $F_2$) are said to be \emph{orthogonal}. 
\end{defn}

\begin{rem}\label{rem:diam_transverse}
Lemma~\ref{lem:bounded_projections} implies that either $F$ is 
parallel to a subcomplex of $F'$ or $\pi_F(F')$ has uniformly bounded 
diameter.  In the latter case, Proposition~\ref{prop:projections} 
implies that either $\cup_{F_1\in[F']}\pi_F(F_1)$ is uniformly 
bounded, or $[F]$ and $[F_1]$ are orthogonal.  In fact, Lemma \ref{lem:bounded_projections} says that if $F$ is not parallel to a subcomplex of $F'$, then $\pi_{F}(F'')$ has diameter bounded by $\xi+2$ for each $F''$ parallel to $F'$. On the other hand, whenever $\pi_{F}(F'')$ has diameter bounded by $\xi+2$ for each $F''$ parallel to $F'$, Proposition~\ref{prop:projections} implies that either $\cup_{F_1\in[F']}\pi_F(F_1)$ has diameter bounded by $3\xi+6$ or $[F]$ and $[F']$ are orthogonal.\end{rem}

Motivated by the remark, we give the following two definitions.
Definition~\ref{defn:parproj} provides projections between parallelism
classes; as explained in Remark~\ref{rem:diam_transverse},
Proposition~\ref{prop:projections} gives conditions ensuring that
these projections are coarsely well-defined. 

\begin{defn}[Transverse]\label{defn:transverse}
Parallelism classes $[F],[F']\in\factorseq$ are \emph{transverse} if
they are not orthogonal and if $F$ is not parallel to a 
subcomplex of $F'$, or vice versa, for some (hence all)
$F\in[F],F'\in[F']$.
\end{defn}

\begin{defn}[Projection of parallelism classes, projection distances]\label{defn:parproj}
Let $[F]\in\factorseq$.  The \emph{projection to $[F]$ (really to $\fontact F$, which depends only on the parallelism class)} is the map $\factorseq-\{[F]\}\rightarrow 2^{\fontact F}$ given by $\pi_{[F]}([F'])=\cup_{F_1\in[F']}\pi_F(F_1)$. 

For each $[Y]\in\factorseq$, define a function $\dist_Y^\pi\colon \left(\factorseq-\{[Y]\}\right)^2\rightarrow[0,\infty]$ by $$\dist_Y^\pi([F],[F'])=\diam_{\fontact Y}(\pi_{[Y]}([F])\cup\pi_{[Y]}([F'])).$$
\end{defn}

We fix, until the end of the section, any subset 
$\factorseq_{_{tr}}\subset \factorseq$ with the property that for each distinct pair 
$[F],[F']\in\factorseq_{_{tr}}$ we have that 
$[F]$ is transverse to $[F']$.

Observe that for each $F\in\factorseq_{_{tr}}$, the restriction of 
$\dist_F^\pi$ to $\factorseq_{_{tr}}$ takes uniformly bounded values.  From the 
definition, it is obvious that 
$\dist_F^\pi([F'],[F''])=\dist_F^\pi([F''],[F'])$ and 
$\dist_F^\pi([F'],[F''])+\dist_F^\pi([F''],[F'''])\geq\dist_F^\pi([F'],[F'''])$ for all $[F],[F'],[F''],[F''']\in\factorseq_{_{tr}}$.  Moreover, the fact that the right-hand side of the distance formula (Theorem~\ref{thm:distance_formula}) is finite shows that, for a suitable $\eta$, we have 
$\left|\{F:\dist^\pi_F([F'],[F''])\geq\eta\}\right|<\infty$ for all $[F'],[F'']\in\factorseq_{_{tr}}$.

\begin{prop}\label{prop:bounded_projection}
For all $[F]\neq[F']$ in $\factorseq_{_{tr}}$, we have $\diam_{\fontact F}(\pi_F(F'))\leq\xi+2$.
\end{prop}

\begin{proof}
This is just Lemma~\ref{lem:bounded_projections} combined with the fact that $F$ cannot be parallel to a subcomplex of $F'$ by our assumptions on $\factorseq_{_{tr}}$.
\end{proof}

The next result is the analogue for cubical groups of the inequality 
that Behrstock established in \cite{Behrstock:asymptotic} for the 
mapping class group. Versions of this inequality in other contexts 
have appeared in \cite{BBF:quasi_tree, BestvinaFeighn:projections, SabalkaSavchuk:projections, Taylor:RAAGoutqie}

\begin{prop}[Behrstock inequality]\label{prop:behrstock_inequality}
For all $[F],[F'],[F'']\in\factorseq_{_{tr}}$, $$\min\{\dist^\pi_F([F'],[F'']),\dist^\pi_{F'}([F],[F''])\}\leq3\xi+6.$$
\end{prop}

\begin{proof}
Let $F,F',F''$ satisfy $\dist_F^\pi([F'],[F''])>3\xi+6$.  Let $x\in F'$. We claim that there exists $Y$ parallel to $F$ and $y\in Y$ so that $\gate_{F''}(x)=\gate_{F''}(y)$, which will conclude the proof.

Let $\gamma$ be a geodesic from $x$ to $\gate_{F''}(x)$. By Lemma~\ref{lem:geodesic_near_gate},
there exists $Y$ parallel to $F$ such that $\gamma$ intersects $Y$, say at $y$. Since $y$ is on $\gamma$, we have $\gate_{F''}(x)=\gate_{F''}(y)$, as required.
\end{proof}

\begin{lem}\label{lem:geodesic_near_gate}
Let $F\in\factorsup$ and let $x,y\in\cuco X^{(0)}$ satisfy $\dist_{\fontact F}(\pi_F(x),\pi_F(y))>2\xi+4$.  Then any geodesic from $x$ to $y$ enters a parallel copy of $F$.
\end{lem}

\begin{proof}
Let $\gamma$ be a geodesic from $x$ to $y$.  By Lemma~\ref{lem:projection_parallel} --- which implies  $\dist_{\fontact F}(\pi_F(a),\pi_F(b))=\dist_{\fontact F'}(\pi_{F'}(a),\pi_{F'}(b))$ for all $a,b\in\cuco X$ whenever $F,F'$ are parallel --- we can assume that $F$ was chosen within its parallelism class so that $\dist_{\cuco X}(F,\gamma)$ is minimal.  

Observe that no hyperplane separates $\gamma$ from $F$.  Indeed, suppose that $H$ is such a hyperplane.  Then $\gate_F(x),\gate_F(y)\subset\gate_F(H)$.  By our assumption that $\dist_{\fontact F}(\pi_F(x),\pi_F(y))>2\xi+4$, we have $\diam_{\fontact F}(\gate_F(H))>2\xi+4>\xi+2$, whence, by Lemma~\ref{lem:bounded_projections}, $F$ is parallel to a subcomplex $F'$ of the combinatorial hyperplane $H^-$ on $\neb(H)$ which is separated from $F$ by $H$.  For any $z\in\gamma$, any hyperplane separating $z$ from $F'$ separates $z$ from $F$, whence $\dist_{\cuco X}(\gamma,F')<\dist_{\cuco X}(\gamma,F)$, contradicting our choice of $F$.

Let $\mathfrak L$ be the set of hyperplanes separating $x$ from $\{y\}\cup F$ and let $\mathfrak R$ be the set of hyperplanes separating $y$ from $\{x\}\cup F$.  If $\mathfrak L=\emptyset$, then either $x\in F$, since no hyperplane separates $x$ and $y$ from $F$ (for otherwise, by convexity of halfspaces, that hyperplane would separate $\gamma$ from $F$).  Similarly, $\mathfrak R\neq\emptyset$.  

Let $L\in\mathfrak L$ be closest to $y$ and let $R\in\mathfrak R$ be closest to $x$.  Suppose that $L\cap R=\emptyset$.  Then there exists a $0$--cube $z\in\gamma$ so that no hyperplane in $\mathfrak L\cup\mathfrak R$ separates $z$ from $F$.  Hence any hyperplane separating $z$ from $F$ separates $\{x,y\}$, and thus $\gamma$, from $F$, which is impossible.  

Thus $L$ and $R$ cross.  Since $\gate_F(x),\gate_F(y)\in\gate_F(L)\cup\gate_F(R)$, arguing as above yields, say, $\diam_{\fontact F}(\gate_F(L))>\xi+2$, so $F$ is parallel into $L$, violating our choice of $F$ as above.
\end{proof}

Observe that Propositions~\ref{prop:bounded_projection} and~\ref{prop:behrstock_inequality}, together with the discussion preceding them, imply the following, which we record for later convenience:

\begin{cor}[BBF Axioms]\label{cor:bbf}
Let $\factorsup$ be a factor system and let $\factorseq$ be the set of parallelism classes in $\factorsup$.  Let $\factorseq_{_{tr}}$ be a subset of $\factorseq$ such that $[F],[F']$ are transverse for all distinct $[F],[F']\in\factorseq_{_{tr}}$.  Then $\{\dist_Y^\pi:[Y]\in\factorseq_{_{tr}}\}$ satisfies Axioms~$(0)-(4)$ of~\cite[Section~2.1]{BBF:quasi_tree}.
\end{cor}

\section{Efficient embeddings into products of trees}\label{subsec:coloring}

We now let $\cuco X$ be a CAT(0) cube complex with a proper cocompact
action by a group $G$, and we let $\factorsup$ be a $G$--invariant
factor system (recall that the existence of \emph{any} factor system
ensures the existence of a $G$--invariant one).  In this
section we produce a very particular $G$--equivariant
quasi-isometric embedding of $\cuco X$ into the product of finitely many
quasi-trees.  We begin with the following fact which is 
well-known in the study of cubical groups:

\begin{prop}[Embeddings in products of trees]\label{prop:sep_embed}
Suppose that there exists $G'\leq_{f.i.}G$ such that no hyperplane in $\cuco X$ crosses its $G'$--translates.  Then there exists a finite collection $\{T_1,\ldots,T_k\}$ of simplicial trees, a $G$--action on $\prod_{i=1}^kT_i$, and a $G$--equivariant isometric embedding $\cuco X\rightarrow\prod_{i=1}^kT_i$.  

Conversely, suppose that $G$ acts properly and cocompactly on a CAT(0) cube complex $\cuco X$, and that there is a $G$--equivariant cubical isometric embedding $\cuco X\rightarrow\prod_{i=1}^kT_i$ with each $T_i$ a simplicial tree.  Then there exists a finite-index subgroup $G'\leq G$ such that for all hyperplanes $H$ of $\cuco X$ and all $g\in G'$, the hyperplanes $gH$ and $H$ do not cross.
\end{prop}

\begin{proof}
By passing to the normal core, we can assume $G'$ is normal.  Let $H_1,\ldots,H_k$ be a complete list of representatives of $G'$--orbits of hyperplanes.  For each $i$, the cube complex dual to the wallspace $(\cuco X^{(0)},G'\cdot H_i)$ is a simiplicial tree $T_i$ on which $G'$ acts by isometries, and $G$ acts on $\prod_{i=1}^kT_k$, permuting the factors, with $G'$ stabilizing each factor (if $h\in G'$ and $g\in G$ and $H_j=gH_i$, then $hH_j=hgH_i=g(g^{-1}hg)H_i=gH_i$ since $G'$ is normal).  The existence of the embedding follows from~\cite[Corollary~1]{ChepoiHagen:embedding}; equivariance is easily checked.

Conversely, let $\cuco X\rightarrow\prod_{i=1}^kT_i$ be a $G$--equivariant embedding, and let $G'$ be the kernel of the action of $G$ on the factors of $\prod_{i=1}^k\cuco X$.  Since $\cuco X\rightarrow\prod_{i=1}^kT_i$ sends hyperplanes to hyperplanes, and sends crossing hyperplanes to crossing hyperplanes, the coloring of the hyperplanes of $\prod_{i=1}^nT_i$ by the tree factor that they cross pulls back to a $G'$--invariant coloring of the hyperplanes of $\cuco X$ by $k$ colors, with the property that hyperplanes of like color do not cross.  By definition, two hyperplanes in the same $G'$ orbit have the same color.\end{proof}

An element of $\factorsup-\{\cuco X\}$ is \emph{maximal} if it is not properly contained in an element of $\factorsup-\{\cuco X\}$.

\begin{defn}[Hereditarily flip-free]\label{defn:flip_free}
Let $G$ act properly and cocompactly on the cube complex $\cuco X$.  The action is \emph{flip-free (with respect to $\factorsup$)} if for each cubical isometric embedding $A\times B\rightarrow\cuco X$ such that $A,B\in\factorsup-\{\cuco X\}$ and each of $A,B$ is parallel to a maximal element of $\factorsup-\{\cuco X\}$, no $g\in G$ has $gA$ parallel to $B$.  The action is \emph{hereditarily flip-free} with respect to $\factorsup$ if for each $H\in\factorsup$, the action of $\stabilizer_G(H)$ on $H$ is flip-free with respect to $\factorsup_{H}$.\end{defn}

\begin{lem}[Coloring the factor system]\label{lem:coloring_factor_systems}
Let $G$ act properly and cocompactly on $\cuco X$, let $\factorsup$ be a $G$--invariant factor system, and suppose that there exists $G'\leq_{f.i.}G$ such that the action of $G'$ on $\cuco X$ is hereditarily flip-free with respect to $\factorsup$.  Then there exists $k\in\naturals$ and a coloring $\chi\colon \factorsup\rightarrow\{1,\ldots,k\}$ such that:
\begin{enumerate}
 \item \label{item:color_par}if $F,F'\in\factorsup$ are parallel, then $\chi(F)=\chi(F')$;
 \item \label{item:color_trans}$\chi(gF)=\chi(F)$ for all $F$ and all $g\in G'$;
 \item \label{item:color_nonmax}if $F$ is parallel to a proper subcomplex of $F'$, then $\chi(F)\neq\chi(F')$;
 \item \label{item:color_orth}if $[F]$ and $[F']$ are orthogonal, then $\chi(F)\neq\chi(F')$.
\end{enumerate}
Hence there is a $G'$--invariant coloring $\chi\colon \factorseq\rightarrow\{1,\ldots,k\}$ such that for $1\leq i\leq k$, the set $\factorseq_i=\chi^{-1}(i)$ is a $G'$--invariant collection satisfying the hypotheses of Corollary~\ref{cor:bbf}.
\end{lem}

\begin{proof}
For $n\geq 1$ and $\{q_i\}_{i=1}^n$ determined below, let $\{c^0_1,\ldots,c^0_{q_n}\},\{c^1_1,\ldots,c^1_{q_1}\}\ldots,\{c^n_1,\ldots,c^n_{q_n}\}$ be disjoint sets of colors, and let $\varepsilon$ be an extra color.

Define the relation $\sim$ on $\factorsup$ to be the transitive closure of the union of the following two relations: $F\sim F'$ if $F,F'$ are parallel; $F\sim F'$ if there exists $g\in G'$ such that $F=gF$.  Since the $G$--action preserves parallelism, if $F\sim F'$, then there exists $g\in G'$ with $F,gF'$ parallel. 

We say that $F\in\factorsup-\{\cuco X\}$ is \emph{maximal} if it is not parallel to a proper subcomplex of any other $F'\in\factorsup-\{\cuco X\}$. Notice that maximal elements exist. In fact, if we had a chain $F_1,\dots,F_N$, with $N$ larger than the local finiteness constant $\Delta$ of the factor system, so that $F_i$ is parallel to a proper subcomplex of $F_{i+1}$, then there would be some $0$--cube in $\bigcap_i\gate_{F_N}(F_i)$, violating local finiteness.  Let $F_1,\ldots,F_{n}\in\factorsup$ be maximal elements of $\factorsup-\{\cuco X\}$, chosen so that any other maximal $F\in\factorsup-\{\cuco X\}$ satisfies $F\sim F_i$ for some unique $i$; there are finitely many $\sim$-classes since our hypotheses ensure that there are finitely many $G$-orbits in $\factorsup$.  Let $\factorsup_0\subseteq\factorsup$ consist of those $F$ such that $F\sim F_i$ for some $i$, so that each maximal factor subcomplex is in $\factorsup_0$.  Define $\chi_0\colon \factorsup_0\rightarrow\{c_1^0,\ldots,c_{n}^0\}$ by $\chi_0(F)=c^0_i$, where $i$ is the unique index with $F\sim F_i$.  The coloring $\chi_0$ of $\factorsup_0$ satisfies items~\eqref{item:color_par},~\eqref{item:color_trans},~\eqref{item:color_nonmax} by definition (in particular, maximality ensures that item~\eqref{item:color_nonmax} holds).  Item~\eqref{item:color_orth} follows from the assumption that the action of $G'$ is flip-free.  Indeed, if $[F]$ and $[F']$ are orthogonal then, by definition, we can choose $F,F'$ within their parallelism classes so that $F\times F'\hookrightarrow\cuco X$ isometrically.  But if $\chi_0(F)=\chi_0(F')$, then $F$ is parallel to $gF'$ for some $g\in G'$, contradicting flip-freeness since each of $F,F'$ is parallel to a maximal element, namely a translate of some $F_i$.   

Note that since $\factorsup_0$ is a locally finite collection and $G'$ acts on $\cuco X$ cocompactly, $G_F':=\stabilizer_{G'}(F)$ acts cocompactly on $F$ for each $F\in\factorsup$.  For each $F\in\factorsup$, let $\factorsup_{F}$ be the induced factor system, so that $\cup_F\factorsup_{F}=\factorsup$.  By hereditary flip-freeness of the overall action of $G'$ on $\cuco X$, the action of $G_F'$ on each $F$ is flip-free with respect to $\factorsup_{F}$.  Hence, by induction, for $1\leq i\leq n$, we have a coloring $\chi'_{F_i}\colon \factorsup_{F_i}\rightarrow\{c^i_1,\ldots,c_{q_i}^i\}$ satisfying all four conclusions of the proposition with respect to the $G'_{F_i}$--action.  

We define $\chi_i\colon \cup_{F\sim F_i}\factorsup_{F}\rightarrow\{c^i_1,\ldots,c_{q_i}^i\}$ as follows.  If $F$ is parallel to $F_i$, then each $F'\in\factorsup_{F}$ is parallel to some $F'_i\in\factorsup_{F_i}$, and if $F'$ is also parallel to $F''_i\in\factorsup_{F_i}$, then $\chi_{F_i}'(F'_i)=\chi_{F_i}'(F''_i)$.  Hence let $\chi_i(F')=\chi_{F_i}(F_i')$.  We now extend $G'$--equivariantly to obtain $\chi_i$.  More precisely, for each $F\sim F_i$, choose $g\in G'$ so that $gF$ is parallel to $F_i$, and define $\chi_i$ on $\factorsup_{F}$ so that $\chi_i(F')=\chi_i(gF')$ for all $F'\in\factorsup_{F}$.  This is independent of the choice of $gF$ because of how $\chi_i$ was defined on the subcomplexes parallel to $F_i$.  Moreover, since $\chi_{F_i}'$ is $G'_{F_i}$--invariant, and parallel complexes have identical stabilizers, this is independent of the choice of $g$.

Let $F\sim F_i$ and let $F',F''\in\factorsup_{F}$.  Suppose that $\chi_i(F')=\chi_i(F'')$.  Choose $g\in G'$ so that $gF$ is parallel to $F_i$.  Then $\chi_i(gF')=\chi_i(gF'')=\chi_i(F_i')=\chi_i(F_i'')$, where $F_i',F_i''\in\factorsup_{F_i}$ are parallel to $gF',gF''$ respectively.  Hence neither of $F_i''$ nor $F_i'$ is parallel to a subcomplex of the other, and they are not orthogonal, by our induction hypothesis.  Hence the same is true of $F',F''$, so that $\chi_i$ correctly colors $\factorsup_{F}$.     

For $i\geq 0$, we extend $\chi_i$ to the rest of $\factorsup$ by letting $\chi_i(H)=\varepsilon$ when $H\not\in\cup_{F\sim F_i}\factorsup_{F}$ (or $H\not\in\factorsup_0$ when $i=0$).  We finally define the coloring $\chi$ on $\factorsup$ by $$\chi(F)=\left(\chi_0(F),\chi_{1}(F),\ldots,\chi_{n}(F)\right).$$  (The total number of colors is $(n+1)\prod_{i=1}^n(q_i+1)$.)  

Let $F,F'\in\factorsup$.  If the set of $i$ for which $F$ is contained in $\factorsup_{H}$ for some $H\sim F_i$ differs from the corresponding set for $F'$, then $\chi(F)\neq\chi(F')$ since there is at least one coordinate in $\chi(F)$ equal to $\varepsilon$ for which the corresponding coordinate in $\chi(F')$ is not $\varepsilon$, or vice versa.  It follows from the definition of the $\chi_i$ that $F$ cannot be parallel to a proper subcomplex of $F'$, and they cannot be orthogonal, if $\chi(F)=\chi(F')$.  Hence $\chi$ is the desired coloring.  
\end{proof}

We can now prove Theorem~\ref{thmi:quasi_trees}:

\begin{prop}[Embedding in a product of quasi-trees]\label{prop:qie_prod_trees}
Let $G,\cuco X,$ and $\factorseq$ be as in Lemma~\ref{lem:coloring_factor_systems}.  Then there is a partition $\factorseq\cup\{[\cuco X]\}=\sqcup_{i=1}^k\factorseq_i$ such that each $\bbf(\factorseq_i)$ is a quasi-tree, $G$ acts by isometries on $\prod_{i=1}^k\bbf(\factorseq_i)$, and there is a $G$--equivariant quasi-isometric embedding $\cuco X\rightarrow\prod_{i=1}^k\bbf(\factorseq_i)$.
\end{prop}

\begin{proof}
Lemma~\ref{lem:coloring_factor_systems} allows us to finitely and equivariantly color the elements of $\factorsup$ so that elements of the same color are transverse.  By Corollary~\ref{cor:bbf} and~\cite[Theorem~A]{BBF:quasi_tree}, each color gives rise to a quasi-tree and a map from $\cuco X$ to each quasi-tree.  Comparing Theorem~\ref{thm:distance_formula} to the distance estimate~\cite[Theorem~4.13]{BBF:quasi_tree} shows, as in~\cite[Section~5.3]{BBF:quasi_tree}, that these maps give the desired quasi-isometric embedding in the product of the quasi-trees associated to the various colors.
\end{proof}

\begin{rem}\label{rem:no_crossing}
Proposition~\ref{prop:example} shows that Proposition~\ref{prop:sep_embed} does not imply Proposition \ref{prop:qie_prod_trees}.  Any stipulation that no hyperplane crosses its $G'$--translates is satisfied when, for example, hyperplane-stabilizers are separable
(this uses Scott's criterion for separability~\cite{Scott:surfaces}). 
Such separability occurs when the action of $G$ on $\cuco X$ is virtually cospecial
(e.g., when $G$ is word-hyperbolic~\cite{Agol:virtual_haken}); this is
stronger than assuming the action of $G$ on $\cuco X$ is
virtually flip-free.  
\end{rem}

\begin{prop}\label{prop:example}
There exists a group $G$ acting properly and cocompactly on a CAT(0) cube complex $\cuco X$ in such a way that $\cuco X$ admits a $G$--equivariant quasi-isometric embedding into the product of finitely many quasi-trees, but $\cuco X$ does not admit a $G'$--equivariant isometric embedding in a finite product of trees for any finite-index $G'\leq G$.
\end{prop}
	
\begin{proof}
We will actually find $G$ and $\cuco X$ such that the following three properties hold:
\begin{enumerate}
 \item \label{item:ff_fs}$\cuco X$ admits a factor system $\factorsup$.
 \item \label{item:ff_sc}For every finite-index subgroup $G'\leq G$, there exists a hyperplane $H$ of $\cuco X$ and some $g\in G'$ such that the hyperplanes $H,gH$ cross.
 \item \label{item:ff_ff}The action of $G''$ on $\cuco X$ is hereditarily flip-free for some $G''\leq_{f.i.}G$.
\end{enumerate}

Let $Z$ be the complex from Example~\ref{exmp:self_strangling}, whose
notation we will use throughout this proof, let $G=\pi_1Z$, and let
$\cuco X=\widetilde Z$.  Property~\eqref{item:ff_fs} is
Proposition~\ref{prop:Z_factor_system} and
Property~\eqref{item:ff_sc} is
Proposition~\ref{prop:D_cross_translates}.  Thus, by
Proposition~\ref{prop:sep_embed}, $\cuco X$ has no $G$--equivariant
isometric embedding in a finite product of trees.  The existence of 
the desired 
quasi-isometric embedding will follow from Proposition~\ref{prop:qie_prod_trees} once
we have established property~\eqref{item:ff_ff}.

To that end, it will be convenient to compute an explicit factor system for $\cuco X$.  We take $\xi\geq r+1$ and let $\factorsup$ be the smallest subset containing $\cuco X$, all combinatorial hyperplanes of $\cuco X$, and $\gate_F(F')$ whenever $F,F'\in\factorsup$ and $\diam(\gate_F(F'))\geq\xi$.    

We first describe the combinatorial hyperplanes in $\cuco X$.  We can
and shall assume that the paths $\alpha,\beta\colon [0,r]\rightarrow
S,T$ from Example~\ref{exmp:self_strangling} are surjective.  This is
for convenience only; it reduces the number of isomorphism types of
hyperplane that we need to consider.  Let $\cuco X_0$ be the universal
cover of $Z_0$, which decomposes as a tree of spaces whose
vertex-spaces are copies of $\widetilde S\times\widetilde T$ and whose
edge-spaces are lines.  Then $\cuco X$ is formed from $\cuco X_0$ by
attaching all lifts of the carrier of the self-crossing hyperplane $W$
(which is a segment of length $r+2$) along a segment of length $r$.

$\cuco X$ thus has compact hyperplanes, all of which are lifts of $W$.
Each corresponds to a cardinality-$2$ parallelism class of
combinatorial hyperplanes.  The projection of such a combinatorial
hyperplane $W^+$ onto some other hyperplane is either an isomorphism
or has image of diameter at most $r$, whence this projection is
excluded from $\factorsup$ by our choice of $\xi$.  The projection of
any other hyperplane onto $W^+$ also has diameter at most $r$ and is
similarly excluded.  Every product subcomplex of the form $W^+\times
E$ with $W^+$ a compact combinatorial hyperplane has the property that
$E$ is contained in a $1$--cube and thus is not a combinatorial
hyperplane.  Hence the compact hyperplanes do not affect whether the 
action of $G$
is flip-free.

The remaining combinatorial hyperplanes are of two types: there are \emph{two-ended} combinatorial hyperplanes, which are parallel to lifts of $\widetilde C$, and \emph{bushy} hyperplanes, that are trees of the following type.  Each is formed by beginning with a copy of $\widetilde S$, attaching a path of length $2$ to each vertex, attaching a copy of $\widetilde T$ to the end of each of these paths, etc.  Then, at the midpoint of each of the length-$2$ paths (i.e., the points at which the hyperplane intersects a lift of $\widetilde C$), attach a single 1-cube (dual to a compact hyperplane).  Observe that any subcomplex parallel to a bushy hyperplane is a bushy hyperplane, but that two-ended hyperplanes are parallel to lines properly contained in bushy hyperplanes.

Let $T,T'$ be subcomplexes parallel to two-ended hyperplanes and let $B,B'$ be bushy hyperplanes.  If $B,B'$ intersect some lift of $\widetilde S\times\widetilde T$ but do not intersect, then $B$ and $B'$ are parallel, so that $\gate_B(B')=B$.  If $B,B'$ are distinct and intersect, then $\gate_B(B')$ is a single point.  If $B,B'$ do not intersect a common vertex-space, then $\gate_B(B')=\gate_B(T)$ for some two-ended hyperplane separating $B$ from $B'$.  But $\gate_B(T)$ is either a single point or is a parallel copy of $T$ for any bushy $B$ and two-ended $T$.  Hence $\factorsup$ consists of $\cuco X$, combinatorial hyperplanes and their parallel copies.  Maximal elements of $\factorsup-\{\cuco X\}$ are bushy hyperplanes or lifts of $\widetilde C$.

Let $G''$ be the unique index-$2$ subgroup of $G$ (this is the kernel 
of the map $G\rightarrow\integers_2$ sending the element represented 
by $C$ to $1$).  Then the action of $G''$  on $\cuco X$ is flip-free.  Indeed, $\cuco X$ does not contain the product of two bushy hyperplanes, so we need only consider the case of a product $\widetilde\alpha\times\widetilde\beta$ in some $\widetilde S\times\widetilde T$.  But any $G''$--translate of $\widetilde\beta$ is either in the same product piece --- and thus parallel to $\beta$ --- or at even distance from $\widetilde\beta$ in the Bass--Serre tree, and hence is again a component of the preimage of $\beta$, and hence not parallel to $\widetilde\alpha$.  To verify hereditary flip-freeness, it suffices to note that each element of $\factorsup-\{\cuco X\}$ is a tree, and thus contains no nontrivial products.
\end{proof}

\section{Consistency and realization}\label{sec:consistency_realization}
In this section, we prove an analogue in the cubical context of the  
Consistency Theorem in the mapping class group;   
see \cite[Theorem~4.3]{BKMM:consistency}. (For an analogue of 
this theorem in Teichm\"{u}ller space with the Teichm\"{u}ller metric, see 
\cite[Section~3]{EskinMasurRafi:large_scale_rank}.) 
For the purposes of this section, $\cuco X$ is an arbitrary CAT(0) 
cube complex with a factor system $\factorsup$.  Let $\xi$ be the 
constant from Definition~\ref{defn:factor_system}.  For each 
$F\in\factorsup$, let $\mathfrak BF= \{S\subseteq\contact F^{(0)}: 
\diam_{\fontact F}(S)\leq 1\}$.

If $U,V\in\factorsup$, and $V$ is parallel to a proper subcomplex of $U$, then there is a map $\subp{V}{U}\colon \mathfrak BU\rightarrow 2^{\fontact V^{(0)}}$ defined as follows: let $b\in\mathfrak BU$ be a clique in $\contact U$ and let $\subp{V}{U}(b)=\cup_{W}\pi_V(W)$, where $W$ varies over all combinatorial hyperplanes parallel to hyperplanes of $U$ in the clique $b$.

Let $\tup b\in\prod_{[F]\in\factorseq}\mathfrak BF$ be a tuple, whose
$[F]$--coordinate we denote by $b_F$.  The tuple $\tup b$ is
\emph{realized} if there exists $x\in\cuco X$ such that $\pi_F(x)=b_F$
for all $[F]\in\factorseq$; note that by Lemma~\ref{lem:projection_parallel}  
$b_{F}$ is independent of the choice of representative $F\in [F]$.  In this section we first give a set of 
\emph{consistency conditions} on coordinate projections which hold for any element of  
$x\in\cuco X$ (for the analogue in the mapping class group see 
\cite{Behrstock:asymptotic}). We then show that this set of necessary 
consistency conditions on a vector of coordinates is, essentially, 
also sufficent for that vector to be realized by an element of 
$\cuco X$.

\begin{rem}
In this section, we use the sets $\mathfrak BF$, simply because the
projection of a point to the contact graph is always a clique.  In the
more general context introduced in Section~\ref{sec:quasi_box} ---
specifically in Definition~\ref{defn:space_with_distance_formula} ---
we abstract the conditions on tuples in
$\prod_{[F]\in\factorseq}\mathfrak BF$ provided by the next several 
results, in terms of tuples in
$\prod_{S\in\mathfrak S}2^{S}$, where $\mathfrak S$ is some set of
uniformly hyperbolic spaces, and the tuple restricted to 
each coordinate is a 
uniformly bounded (though not necessarily diameter--$1$) set.
In the cubical setting, instead of doing this using $S=\fontact F$,
we instead use $\mathfrak BF$ because it yields more 
refined statements.
\end{rem}

Recall that the parallelism classes $[F],[F']\in\factorseq$ are \emph{transverse} if
they are not orthogonal and if $F$ is not parallel to a
subcomplex of $F'$, or vice versa, for some (hence all)
$F\in[F],F'\in[F']$.
We begin with the consistency conditions, which take the form of the following inequalities:

\begin{prop}[Realized tuples are consistent]\label{prop:realization_properties}
Let $\tup b\in\prod_{[F]\in\factorseq}\mathfrak BF$ be realized by $x\in\cuco X$.  There exists $\kappa_0=\kappa_0(\xi)$ such that for all $U,V\in\factorsup$, the following hold:
\begin{enumerate}
 \item \label{item:cr_trans} If $U$ and $V$ are transverse, then
 $$\min\left\{\dist_{\fontact U}(b_U,\pi_U(V)),\dist_{\fontact
 V}(b_V,\pi_V(U))\right\}\leq\kappa_0;$$
 \item \label{item:cr_contain} If $V$ is parallel to a proper subcomplex of $U$, then $$\min\left\{\dist_{\fontact U}(b_U,\pi_U(V)),\diam_{\fontact V}(b_V\cup\subp{V}{U}(b_U))\right\}\leq\kappa_0.$$
\end{enumerate}
\end{prop}

A tuple $\tup b$ (realized or otherwise) satisfying the conclusions 
of the above proposition is said to be $\kappa_0$--\emph{consistent}. 
More generally, a tuple $\tup b$ satisfying inequalities~\eqref{item:cr_trans} and~\eqref{item:cr_contain} from Proposition~\ref{prop:realization_properties}, with $\kappa_0$ replaced by some $\kappa\geq\kappa_0$, will be called \emph{$\kappa$-consistent}. 

\begin{proof}[Proof of Proposition~\ref{prop:realization_properties}]
First suppose that neither $U$ nor $V$ is parallel to a subcomplex of
the other.  Suppose that $\dist_{\fontact U}(b_U,\pi_U(V))>2\xi+4$.
Then $\dist_U(\gate_U(x),\gate_U(V))>\xi$, so that the set $\mathcal
H_U$ of hyperplanes crossing $U$ and separating $x$ from $V$ has
cardinality at least $\xi$.  Suppose also that $\dist_{\fontact
V}(b_V,\pi_V(U))>2\xi+4$, so that the set $\mathcal H_V$ of
hyperplanes crossing $V$ and separating $x$ from $U$ has cardinality
at least $\xi$.  Let $W\in\mathcal H_V$ and let $W^+$ be an associated
combinatorial hyperplane not separated from $x$ by any hyperplane
crossing $V$.  Notice that $\pi_U(W^+)$ contains $\pi_U(x)=b_U$
(because $W$ separates $x$ from $U$) and intersects $\pi_U(V)$.  In
particular $\diam_{\fontact U}(\pi_U(W^+))>\xi+2$ and hence, by
Lemma~\ref{lem:bounded_projections}, $U$ is parallel to a subcomplex
$U'$ of $W^+$ containing $\gate_{W^+}(x)$.  We have $\diam_{\fontact
V}(\pi_V(U'))\leq\xi+2$ for otherwise $V$ would be parallel into $U'$
which is impossible since $U'$ is parallel to $U$.  Also, $\pi_V(U')$
contains $\pi_V(x)=b_V$ and hence $\dist_{\fontact
V}(\pi_V(U),\pi_V(U'))>\xi+2$, i.e.,
Proposition~\ref{prop:projections}.(1) does not hold, so $U$ and $V$
are orthogonal by Proposition~\ref{prop:projections}.

Let $V$ be parallel to a proper subcomplex of $U$ and let $\{W_i\}$ be the set of combinatorial hyperplanes of $U$ containing $\gate_U(x)$.  Suppose that $\dist_{\fontact U}(b_U,\pi_U(V))>\xi+2$.  Then for each $i$, we have that $\gate_V(W_i)$ is a single point (since otherwise there would be a hyperplane intersecting $\gate_V(W_i)$ and hence intersecting both $V$ and $W_i$), hence $\gate_V(W_i)=\gate_V(x)$, whence $b_V=\subp{V}{U}(b_U)$.  In either case, $\kappa_0=\xi+2$ suffices.
\end{proof}

\begin{rem} We note that Proposition~\ref{prop:behrstock_inequality} 
    can alternatively be proven as an immediate consequence of 
    Proposition~\ref{prop:realization_properties} and 
    Lemma~\ref{lem:bounded_projections}. The formulation here is very 
    close to that in the mapping class group, see 
    \cite{Behrstock:asymptotic} and \cite{BKMM:consistency}.
\end{rem}

\begin{thm}[Consistency and realization]\label{thm:consistency_realization}
Let $\kappa_0\geq 1$ be the constant from
Proposition~\ref{prop:realization_properties}.  For each
$\kappa\geq\kappa_0$, there exists $\theta\geq 0$ such that, if $\tup
b\in\prod_{[F]\in\factorseq}\mathfrak BF$ is $\kappa$--consistent,
then there exists $y\in\cuco X$ such that for all $[F]\in\factorseq$,
we have $\dist_{\fontact F}(b_F,\pi_F(y))\leq\theta$.
\end{thm}

\begin{proof}
For each $[U]\in\factorseq$, let $E_U$ be the smallest convex subcomplex of $\cuco X$ with the property that $y\in E_U$ if $\pi_U(y)$ intersects the $2\kappa$--neighborhood of $b_U$ in $\fontact U$.  By Lemma~\ref{lem:projection_parallel}, we have $E_U=E_{U'}$ if $[U]=[U']$. 

Moreover, we claim that $\pi_U(E_U)$ has diameter at most $5\kappa$.
Indeed, let $\mathcal E$ be the set of all $0$--cubes $x$ so that
$\pi_U(x)\cap\neb_{2\kappa}(b_U)\neq\emptyset$.  Given $x,y\in\mathcal
E$, let $H_1,\ldots,H_k$ and $V_1,\ldots,V_\ell$ be sequences of
subcomplexes, representing geodesics in $\fontact U$, so that $x\in
H_1,y\in V_\ell$, and $H_k,V_1$ correspond to points of $b_U$.  Hence
there is a combinatorial path from $x$ to $y$ and carried on
$\bigcup_iH_i\cup\bigcup_jV_j$; by construction, each $0$--cube on
this path lies in $\mathcal E$.  Hence $\mathcal E$ is $1$--coarsely
connected.  Now let $z\in E_U-\mathcal E$ be a $0$--cube and let $H$
be a hyperplane whose carrier contains $z$ and which crosses a
geodesic from $z$ to a point of $\mathcal E$.  Since $E_U$ is the
convex hull of $\mathcal E$, the hyperplane $H$ separates two
$0$--cubes of $\mathcal E$ lying in $\neb(H)$.  Hence $\dist_{\fontact
U}(\pi_U(z),b_U)\leq 2\kappa+1$ and thus $\pi_U(E_U)\leq 5\kappa$.

Suppose $[U],[V]\in\factorseq$ are transverse.  By
$\kappa$--consistency and Lemma~\ref{lem:bounded_projections}, we have
that either $V\subseteq E_U$ or $U\subseteq E_V$.  If $V$ is parallel
to a proper subcomplex of $U$, then the same is true, by
$\kappa$--consistency.  Finally, if $U\times V$ isometrically embeds
in $\cuco X$, then some representative of $[V]$ lies in $E_U$ and some
representative of $[U]$ lies in $E_V$.  Hence, in all cases, at least
one representative subcomplex corresponding to one of the cliques
$b_U$ or $b_V$ belongs to $E_U\cap E_V$, i.e., $E_U\cap
E_V\neq\emptyset$ for all $U,V$.  Thus, by the Helly property, for
any finite collection $[F_1],\dots,[F_n]$ of elements of $\factorseq$,
we can find an $x\in \cuco X$ so that
$\dist_{\fontact{F_i}}(\pi_{F_i}(x), b_{F_i})\leq 5\kappa$ for each~$i$.

Fix any 0-cube $x_0\in \cuco X$ and let $\theta_1$ be a (large) constant to be chosen later. Also, denote by $\factorseq_{\max}$ the set of all $[F]\in\factorseq-\{\cuco X\}$ so that
\begin{itemize}
 \item $\dist_{\fontact F}(\pi_F(x_0), b_F)> \theta_1$, and 
 \item $F$ is not parallel to a proper subcomplex of a representative of $[F']\in\factorseq-\{\cuco X\}$ for which $\dist_{\fontact F'}(\pi_{F'}(x_0), b_{F'})> \theta_1$.
\end{itemize}
Let $l=\dist_{\fontact\cuco X}(\pi_{\cuco X}(x_0),b_{\cuco X})$.  We
claim that $\factorseq_{\max}$ has finitely many elements.  Suppose
not, and let $[F_1],\dots,[F_{p(l+5\kappa)+1}]\in \factorseq_{\max}$,
for some $p$ to be determined below so that $\dist_{\fontact
F}(\pi_{F_i}(x_0), b_{F_i})> \theta_1$ for $1\leq i\leq
p(l+5\kappa)+1$.  Consider $x\in \cuco X$ satisfying  
$\dist_{\fontact{F_i}}(\pi_{F_i}(x), b_{F_i})\leq 5\kappa$ for each
$i$ and such that $\dist_{\fontact\cuco X}(\pi_{\cuco X}(x),b_{\cuco
X})\leq 5\kappa$.  The existence of such a point $x$ was established
above.

Thus, there exists $m\leq l+5\kappa$ such that
$[T_0],\ldots,[T_m]\in\factorsup-\{\cuco X\}$ is a sequence, with
$T_0$ a combinatorial hyperplane parallel to a hyperplane representing
a vertex of $\pi_{\cuco X}(x_0)$ and $T_m$ enjoying the same property
with $x$ replacing $x_0$, representing a geodesic in $\fontact\cuco X$
joining $\pi_{\cuco X}(x_0)$ to $\pi_{\cuco X}(x)$.  Suppose that 
we have chosen constants $\theta_{1}$ and $\kappa$ so that 
$\theta_1\geq 6\kappa-6\geq 5\kappa+4\xi+10$.  Hence, for each $i$, we
have $\dist_{\fontact
F_i}(\pi_{F_i}(x_0),\pi_{F_i}(x))\geq\theta_1-5\kappa\geq 4\xi+10$.
Proposition \ref{prop:bgiII} now implies that each $F_i$ is parallel
to a subcomplex of $T_j$ for some $j\in\{0,\ldots,m\}$.  Hence
$F_i\in\factorseq_{\max}\cap\factorseq_{T_j}$, so that it suffices to
show that the latter is finite.

Note that for each $j$, the set $\factorseq_{\max}\cap\factorseq_{T_j}$ is the set of maximal elements $H\in\factorseq_{T_j}$ with $\dist_{\fontact H}(\pi_{H}(\pi_{T_j}(x_0)), b_H)> \theta_1$.  This follows from the fact that $\gate_H(\gate_{T_j}(x_0))=\gate_H(x_0)$.  Either $\factorseq_{\max}\cap\factorseq_{T_j}=\{T_j\}$, or $\dist_{\fontact T_j}(\pi_{T_j}(x_0), b_{T_j})\leq \theta_1$ and, by induction on $\Delta$, there exists $p\geq 1$ with $|\factorseq_{\max}\cap\factorseq_{T_j}|\leq p$.  Hence $|\factorseq_{\max}|\leq p(l+5\kappa)$.

The finiteness of $\factorseq_{\max}$ holds for any choice of $x_0$, but it will be convenient for what follows to assume that $x_0$ has been chosen so that $\pi_{\cuco X}(x_0)$ contains the clique $b_{\cuco X}$. 

As before, choose $x\in \cuco X$ so that $\dist_{\fontact{F}}(\pi_{F}(x), b_{F})\leq 5\kappa$ for each $[F]\in\factorseq_{\max}$. Let $\factorseq'_{\max}$ be the set of all $[F]\in \factorseq_{\max}$ with the property that for each $[F']\in \factorseq_{\max}$ transverse to $[F]$, the closest parallel copy of $F$ to $x$ is closer in $\cuco X$ than the closest parallel copy of $F'$ to $x$.

\begin{claim}\label{claim:terminal}
Let $[F_1],[F_{-1}]\in\factorseq_{\max}$ be transverse.  Then there exists a unique $i\in\{\pm1\}$ such that $\dist_{\fontact F_{i}}(\pi_{F_{i}}(x_0),\pi_{F_{i}}(F_{-i}))\leq\kappa$.  Moreover, $\dist_{\cuco X}(F_i,x)<\dist_{\cuco X}(F_{-i},x)$, assuming $F_{\pm i}$ is the closest representative of $[F_{\pm i}]$ to $x$.
\end{claim}

\begin{proof}[Proof of Claim~\ref{claim:terminal}]
Suppose first $\dist_{\fontact F_{1}}(\pi_{F_{1}}(x_0),\pi_{F_{1}}(F_{-1}))>\kappa$. Then, by Proposition \ref{prop:realization_properties}, we have $\dist_{\fontact F_{-1}}(\pi_{F_{-1}}(x_0),\pi_{F_{-1}}(F_{1}))\leq\kappa$ and we are done. If $\dist_{\fontact F_{1}}(\pi_{F_{1}}(x_0),\pi_{F_{1}}(F_{-1}))\leq \kappa$, then $\dist_{\fontact F_{1}}(\pi_{F_{1}}(x),\pi_{F_{1}}(F_{-1}))>\kappa$. Proposition \ref{prop:realization_properties} implies that $\dist_{\fontact F_{-1}}(\pi_{F_{-1}}(x),\pi_{F_{-1}}(F_{1}))\leq\kappa$. But then $\dist_{\fontact F_{-1}}(\pi_{F_{-1}}(x_0),\pi_{F_{-1}}(F_{1}))\geq\theta_1-\kappa-\xi-2>\kappa$.

To prove the ``moreover'' clause, notice that $\dist_{\fontact F_{i}}(\pi_{F_{i}}(x),\pi_{F_{i}}(F_{-i}))>4\xi+10$. Hence, by Proposition \ref{prop:bgiII}, for any point $x'$ on any parallel copy of $F_{-i}$ there is a geodesic from $x'$ to $x$ passing through a parallel copy of $F_i$.
\end{proof}

If $\factorseq_{\max}=\emptyset$ then we can choose $y=x_0$.  Otherwise, $\factorseq_{\max}'\neq\emptyset$, and Claim~\ref{claim:terminal} implies that the elements of $\factorseq_{\max}'$ are pairwise-orthogonal; hence $\cuco X$ contains $P=\prod_{[F]\in\factorseq_{\max}'}F$.  By Lemma~\ref{lem:parallel_product}, there is a convex subcomplex $Q\subset\cuco X$ such that the convex hull of the union of all subcomplexes parallel to $P$ is isomorphic to $P\times Q$.  Since $|P|>1$ by our choice of $\theta_1$, there is at least one hyperplane crossing $P$, whence $Q$ is not unique in its parallelism class and hence lies in a combinatorial hyperplane $U$.  By induction on $\Delta$, there exists $y_U\in U$ such that $\dist_{\fontact W}(\pi_W(y_U),b_W)\leq\theta_2$ for some fixed $\theta_2$ and all $[W]\in\factorsup_U$.  Let $y_Q=\gate_Q(y_U)$.

By induction on $\Delta$, for each $[F]\in\factorseq_{\max}'$, we can choose $F\in[F]$ and a point $y_F\in F$ such that $\dist_{\fontact H}(\pi_H(y_F),b_H)\leq\theta_2$ for all $[H]\in\factorsup_F$ and some $\theta_2$.  Importantly, $\theta_2$ can be chosen to depend only on $\kappa,\xi$ and $\Delta-1$, i.e., $\theta_2$ is independent of $\factorseq_{\max}$, and we can thus assume that $\theta_1>\theta_2+2\kappa$.  By pairwise-orthogonality, there exists $y\in P\times Q$ such that $\gate_F(y)=y_F$ for all $[F]\in\factorseq_{\max}'$ and $\gate_Q(y)=y_Q$.  We now choose $\theta$ so that $\dist_{\fontact W}(\pi_W(y),b_W)\leq\theta$ for all $[W]\in\factorsup$.

First, if $W$ is parallel to a proper subcomplex of some $[F]\in\factorseq_{\max}'$,  then any $\theta\geq\theta_2$ suffices.

Second, suppose that some $[F]\in\factorseq_{\max}'$ is parallel to a proper subcomplex of $W$.  Suppose $\dist_{\fontact W}(b_W,\pi_W(F))\leq \kappa+2\theta_1$.  Notice that $\pi_W(y)=\pi_W(y_F)$ because $F$ is parallel to a subcomplex of $W$, and hence $\dist_{\fontact W}(b_W,\pi_W(y_F))=\dist_{\fontact W}(b_W,\pi_W(y))\leq 2\kappa+2\theta_1$. Hence suppose $\dist_{\fontact W}(b_W,\pi_W(F))>\kappa+2\theta_1$, so $\diam_{\fontact F}(b_F\cup\rho^W_F(b_W))\leq \kappa$. Since, up to parallelism, $W$ properly contains $F$ and $[F]\in\factorseq_{\max}$, we can conclude that if $W\neq \cuco X$ then we have $\dist_{\fontact W}(b_W,\pi_W(x_0))\leq\theta_1$, while the same inequality also holds in the case $W=\cuco X$ because of our choice of $x_0$.  Hence $\dist_{\fontact W}(\pi_W(x_0),\pi_W(F))>\kappa+\theta_1$.  By Proposition~\ref{prop:bgiII}, and the fact that a hierarchy path joining $\gate_W(x_0)$ to a point in $W$ projecting to $b_W$ cannot pass through an element of $[F]$ (because $\dist_{\fontact W}(\pi_W(x_0),\pi_W(F))> \kappa+\theta_1\geq \dist_{\fontact W}(\pi_W(x_0),b_W)$), we have $\dist_{\fontact F}(\pi_F(x_0),\rho_F^W(b_W))\leq 4\xi+10$. Hence $\dist_{\fontact F}(\pi_F(x_0),b_F)\leq 4\xi+10+\kappa$, which contradicts $F\in\factorseq_{\max}$, so that this case does not arise. 

Third, suppose that $[W]$ is transverse to some $[F]\in\factorseq_{\max}'$. If $\dist_{\fontact F}(\pi_F(y),\pi_F(W))>\kappa+\theta_2$, then by consistency of $\tup b$, Proposition \ref{prop:realization_properties}, and the fact $\pi_W(y)\in\pi_W(F)$ (since $y$ lies in a parallel copy of $F$), we have $\dist_{\fontact W} (\pi_W(y), b_W)\leq \kappa$. Hence, suppose $\dist_{\fontact F}(\pi_F(y),\pi_F(W))\leq \kappa+\theta_2$. In this case $\dist_{\fontact F}(\pi_F(x_0),\pi_F(W))>\theta_1 -\theta_2-\kappa $, whence $\dist_{\fontact W}(\pi_W(x_0),\pi_W(F))\leq \kappa$.

Now, since $y$ belongs to (a parallel copy of) $F$, we have $\dist_{\fontact W}(\pi_W(y),\pi_W(F))=0$. We now wish to argue that the inequality $\dist_{\fontact W}(\pi_W(y),b_W)\geq \theta_1+\xi+2+\kappa$ is impossible, thus concluding the proof in this case. In fact, if it holds, then $$\dist_{\fontact W}(\pi_W(x_0),b_W)\geq \dist_{\fontact W}(\pi_W(y),b_W) - \diam(\pi_W(F))-\dist_{\fontact W}(\pi_W(F),\pi_W(x_0))\geq \theta_1.$$
Hence $W$ is contained in some $W'$ with $[W']\in\factorseq_{\max}$. Notice that $[W']$ must be transverse to $[F]$, since $F$ cannot be parallel into $W$ nor vice versa, by maximality of $F$ and $W'$, and since, if $[W']$ was orthogonal to $[F]$, then $[W]$ would be as well.

Note $\dist_{\fontact F}(\pi_F(y),\pi_F(W'))\leq \kappa$, since this holds for $W$. Thus $\dist_{\fontact W'}(\pi_{W'}(x_0),\pi_{W'}(F))\leq \kappa$, by an argument above. But then, by Claim \ref{claim:terminal}, the closest parallel copy of $W'$ to $x$ is closer than the closest parallel copy of $F$, in contradiction with $[F]\in\factorseq'_{\max}$.

The only case left is when $W$ is orthogonal to all $[F]\in\factorseq_{\max}'$, which implies that $W$ is, up to parallelism, contained in $Q$. In this case, $\pi_W(y)=\pi_W(\gate_Q(y_Q))=\pi_W(y_Q)$, whence we are done by the choice of $y_Q$.
\end{proof}

\begin{rem}
Let $\kappa\geq\kappa_0$ and let $\tup b$ be a $\kappa$--consistent
tuple.  Theorem~\ref{thm:consistency_realization} yields a 
constant~$\theta$ and a $0$--cube $x\in\cuco X^{(0)}$ such that $\dist_{\fontact
F}(x,b_F)\leq\theta$ for all $F\in\factorsup$.  By
Theorem~\ref{thm:distance_formula}, any other $0$--cube $y$ with this
property is at a uniformly bounded distance --- depending on $\kappa$
and $\xi$ --- from $x$.  The definition of a factor system implies
that $\cuco X$ is uniformly locally finite, so the number of such
$0$--cubes $y$ is bounded by a constant depending only on $\kappa,\xi,$ and
the growth function of $\cuco X^{(0)}$. 
\end{rem}

\part{Hierarchically hyperbolic spaces}\label{part:applications}

\section{Quasi-boxes in hierarchically hyperbolic spaces}\label{sec:quasi_box}
In this section, we work in a level of generality which includes 
both mapping class groups and each CAT(0) cube complex $\cuco X$ with a factor system $\factorsup$.

\begin{defn}[Hierarchically hyperbolic space]\label{defn:space_with_distance_formula}
The metric space $(\cuco X,\dist_{\cuco X})$ is a \emph{hierarchically
hyperbolic space} if there exists $\delta\geq0$, an index set
$\mathfrak S$, and a set $\{\fontact W \mid W\in\mathfrak S\}$ of
$\delta$--hyperbolic spaces, such that the following conditions are
satisfied: \begin{enumerate}
\item\textbf{(Projections.)}\label{item:dfs_curve_complexes} There is
a set $\{\pi_W:\cuco X\rightarrow2^{\fontact W}\mid W\in\mathfrak S\}$
of \emph{projections} sending points in $\cuco X$ to sets of diameter
bounded by some $\xi\geq0$ in the various $\fontact W\in\mathfrak S$.
 \item \textbf{(Nesting.)} \label{item:dfs_nesting} $\mathfrak S$ is
 equipped with a partial order $\nest$, and either $\mathfrak
 S=\emptyset$ or $\mathfrak S$ contains a unique $\nest$--maximal
 element; when $V\nest W$, we say $V$ is \emph{nested} in $W$.  We
 require that $W\nest W$ for all $W\in\mathfrak S$.  For each
 $W\in\mathfrak S$, we denote by $\mathfrak S_W$ the set of
 $V\in\mathfrak S$ such that $V\nest W$.  Moreover, for all $V,W\in\mathfrak S$
 with $V$ properly nested into $W$ there is a specified subset
 $\rho^V_W\subset\fontact W$ with $\diam_{\fontact W}(\rho^V_W)\leq\xi$.
 There is also a \emph{projection} $\rho^W_V\colon \fontact
 W\rightarrow 2^{\fontact V}$.  (The similarity in notation is
 justified by viewing $\rho^V_W$ as a coarsely constant map $\fontact
 V\rightarrow 2^{\fontact W}$.)
 \item \textbf{(Orthogonality.)} \label{item:dfs_orthogonal}
 $\mathfrak S$ has a symmetric and anti-reflexive relation called
 \emph{orthogonality}: we write $V\orth W$ when $V,W$ are orthogonal.
 Also, whenever $V\nest W$ and $W\orth U$, we require that $V\orth U$.
 Finally, we require that for each $T\in\mathfrak S$ and each
 $U\in\mathfrak S_T$ for which $\{V\in\mathfrak S_T\mid V\orth
 U\}\neq\emptyset$, there exists $W\in\mathfrak S_T-\{T\}$, so that
 whenever $V\orth U$ and $V\nest T$, we have $V\nest W$.  Finally, if
 $V\orth W$, then $V,W$ are not $\nest$--comparable.
 \item \textbf{(Transversality and consistency.)}
 \label{item:dfs_transversal} If $V,W\in\mathfrak S$ are not
 orthogonal and neither is nested in the other, then we say $V,W$ are
 \emph{transverse}, denoted $V\transverse W$.  There exists
 $\kappa_0\geq 0$ such that if $V\transverse W$, then there are
  sets $\rho^V_W\subseteq\fontact W$ and
 $\rho^W_V\subseteq\fontact V$ each of diameter at most $\xi$ and 
 satisfying: $$\min\left\{\dist_{\fontact
 W}(\pi_W(x),\rho^V_W),\dist_{\fontact
 V}(\pi_V(x),\rho^W_V)\right\}\leq\kappa_0$$ for all $x\in\cuco X$; 
 alternatively, in the case $V\nest W$, then for all
 $x\in\cuco X$ we have: $$\min\left\{\dist_{\fontact
 W}(\pi_W(x),\rho^V_W),\diam_{\fontact
 V}(\pi_V(x)\cup\rho^W_V(\pi_W(x)))\right\}\leq\kappa_0.$$
 
 \noindent Suppose that: either $U\propnest V$ or $U\transverse V$, and either $U\propnest W$ or $U\transverse W$.  Then we have: if $V\transverse W$, then 
 $$\min\left\{\dist_{\fontact
 W}(\rho^U_W,\rho^V_W),\dist_{\fontact
 V}(\rho^U_V,\rho^W_V)\right\}\leq\kappa_0$$
 and if $V\propnest W$, then 
 $$\min\left\{\dist_{\fontact
 W}(\rho^U_W,\rho^V_W),\diam_{\fontact
 V}(\rho^U_V\cup\rho^W_V(\rho^U_W))\right\}\leq\kappa_0.$$
 
 Finally, if $V\nest U$ or $U\orth V$, then $\dist_{\fontact W}(\rho^U_W,\rho^V_W)\leq\kappa_0$ whenever $W\in\mathfrak S-\{U,V\}$ satisfies either $V\nest W$ or $V\transverse W$ and either $U\nest W$ or $U\transverse W$.
 \item \textbf{(Finite complexity.)} \label{item:dfs_complexity} There exists $n\geq0$, the \emph{complexity} of $\cuco X$ (with respect to $\mathfrak S$), so that any sequence $(U_i)$ with $U_i$ properly nested into $U_{i+1}$ has length at most $n$.
 \item \textbf{(Distance formula.)} \label{item:dfs_distance_formula}
 There exists $s_0\geq\xi$ such that for all $s\geq s_0$ there exist
 constants $K,C$ such that for all $x,x'\in\cuco X$, $$\dist_{\cuco
 X}(x,x')\asymp_{(K,C)}\sum_{W\in\mathfrak S}\ignore{\dist_{\fontact
 W}(\pi_W(x),\pi_W(x'))}{s}.$$
 We often write $\sigma_{\cuco X,s}(x,x')$ to denote the right-hand side of
 Item~\eqref{item:dfs_distance_formula}; more generally, given
 $W\in\mathfrak S$, we denote by $\sigma_{W,s}(x,x')$ the corresponding
 sum taken over $\mathfrak S_W$.
 
 \item \textbf{(Large links.)} \label{item:dfs_large_link_lemma} There
exists $\lambda\geq1$ such that the following holds.
Let $W\in\mathfrak S$ and let $x,x'\in\cuco X$.  Let
$N=\lambda\dist_{\fontact W}(\pi_W(x),\pi_W(x'))+\lambda$.  Then there exists $\{T_i\}_{i=1,\dots,\lfloor
N\rfloor}\subseteq\mathfrak S_W-\{W\}$ such that for all $T\in\mathfrak
S_W-\{W\}$, either $T\in\mathfrak S_{T_i}$ for some $i$, or $\dist_{\fontact T}(\pi_T(x),\pi_T(x'))<s_0$.  Also, $\dist_{\fontact W}(\pi_W(x),\rho^{T_i}_W)\leq N$ for each $i$.

 \item \textbf{(Bounded geodesic image.)} \label{item:dfs:bounded_geodesic_image} For all $W\in\mathfrak S$, all $V\in\mathfrak S_W-\{W\}$, and all geodesics $\gamma$ of $\fontact W$, either $\diam_{\fontact V}(\rho^W_V(\gamma))\leq B$ or $\gamma\cap\neb_E(\rho^V_W)\neq\emptyset$ for some uniform $B,E$.
 \item \textbf{(Realization.)} \label{item:consistency_and_realization} For each $\kappa$ there exists $\theta_e,\theta_u$ such that the following holds. Let $\tup
 b\in\prod_{W\in\mathfrak S}2^{\fontact W}$ have each coordinate 
 correspond to a subset of $\fontact W$ of diameter at most $\kappa$; for each $W$, let $b_W$ denote the $\fontact
 W$--coordinate of $\tup b$. Suppose that whenever $V\transverse W$ we have
 $$\min\left\{\dist_{\fontact W}(b_W,\rho^V_W),\dist_{\fontact V}(b_V,\rho^W_V)\right\}\leq\kappa$$
 and whenever $V\nest W$ we have
 $$\min\left\{\dist_{\fontact W}(b_W,\rho^V_W),\diam_{\fontact V}(b_V\cup\rho^W_V(b_W))\right\}\leq\kappa.$$
 Then there the set of all $x\in \cuco X$ so that $\dist_{\fontact W}(b_W,\pi_W(x))\leq\theta_e$ for all $\fontact W\in\mathfrak S$ is non-empty and has diameter at most $\theta_u$.
 (A tuple $\tup b$ satisfying the inequalities above is called $\kappa$--\emph{consistent}.)
 
 \item \textbf{(Hierarchy paths.)} \label{item:hierarchy_paths} There exists $D\geq 0$ so that any pair of points in $\cuco X$ can be joined by a $(D,D)$--quasi-geodesic $\gamma$ with the property that, for each $W\in\mathfrak S$, the projection $\pi_W(\gamma)$ is at Hausdorff distance at most $D$ from any geodesic connecting $\pi_W(x)$ to $\pi_W(y)$.  We call such quasi-geodesics \emph{hierarchy paths}.
\end{enumerate}
We say that the metric spaces $\{\cuco X_i\}$ are \emph{uniformly hierarchically hyperbolic} if each $\cuco X_i$ satisfies the axioms above and all constants can be chosen uniformly.
\end{defn}

Observe that a space $\cuco X$ is
hierarchically hyperbolic with respect to $\mathfrak S=\emptyset$,
i.e., hierarchically hyperbolic of complexity 0, if and only if $\cuco
X$ is bounded.  Similarly, $\cuco X$ is hierarchically hyperbolic of
complexity $1$, with respect to $\mathfrak S=\{\cuco X\}$, if and only
if $\cuco X$ is hyperbolic.

\begin{rem}[Cube complexes with factor systems are hierarchically hyperbolic spaces]\label{rem:cube_complex_case}
In the case where $\cuco X$ is a CAT(0) cube complex with a factor system, we take $\mathfrak S$ to be a subset of the factor system containing exactly one element from each parallelism class. We take $\{\fontact W \mid W\in\mathfrak S\}$ to
be the set of factored contact graphs of elements of $\mathfrak S$.
All of the properties required by
Definition~\ref{defn:space_with_distance_formula} are satisfied: given
$U,V\in\mathfrak S$, we have $U\nest V$ if and only if $U$ is parallel
into $V$; orthogonality is defined as in Section~\ref{sec:par_proj}.
The last condition on orthogonality is satisfied since each element of
a factor system is either unique in its parallelism class or is
contained in a hyperplane.  The first two inequalities in
Item~\eqref{item:dfs_transversal} follow from
Proposition~\ref{prop:behrstock_inequality}; the second two follow
from the first two since each $\rho^U_V$ is the image of the convex
subcomplex $U\in\mathfrak S$ under the projection $\pi_V$.  Finiteness
of complexity follows from the definition of a factor system
(specifically, uniform local finiteness of the family of factors).
Item~\eqref{item:dfs_distance_formula} is provided by
Theorem~\ref{thm:distance_formula},
Item~\eqref{item:dfs_large_link_lemma} by
Proposition~\ref{prop:bgiII},
Item~\eqref{item:dfs:bounded_geodesic_image} is
Proposition~\ref{prop:BGIIITF},
Item~\eqref{item:consistency_and_realization} is
Theorem~\ref{thm:consistency_realization}, and
Item~\eqref{item:hierarchy_paths} is
Proposition~\ref{prop:hier_revisited}.  In order to ensure that nesting and orthogonality are mutually exclusive, we require that elements of the factor system $\mathfrak S$ are not single points.
\end{rem}

\begin{rem}
The Bounded Geodesic Image constant $E$ from Definition~\ref{defn:space_with_distance_formula} can be taken to be $1$ for cube complexes with factor systems, as well as for the mapping class group.  However, we allow an arbitrary (fixed) constant in the general definition for greater flexibility.
\end{rem}

\begin{rem}
    In the case of a factor system, the constant in 
    Item~\eqref{item:dfs_curve_complexes} of 
    Definition~\ref{defn:space_with_distance_formula} is 1, since
    points of $\cuco X$ project to cliques in each factored contact graph.
    The constant in Item~\eqref{item:dfs_nesting} is $\xi+2$, 
    where $\xi$ is the constant from
    Definition~\ref{defn:factor_system}. For simplicity, 
    in Definition~\ref{defn:space_with_distance_formula}, we use a 
    single constant to fulfill both of these
    roles.
\end{rem}

\subsection{Product regions and standard boxes}\label{subsec:box}

Let $\cuco X$ be a hierarchically hyperbolic space and let
$U\in\mathfrak S$.  With $\kappa_0$ as in
Definition~\ref{defn:space_with_distance_formula}, we say that the
tuple $\tup a_{\nest}\in\prod_{W\nest
U}2^{\fontact W}$ or $\tup a_{\orth}\in\prod_{W\orth U}2^{\fontact W}$, is
$\kappa_0$--\emph{consistent} if the inequalities from
Item~\eqref{item:dfs_transversal} are satisfied for each pair of 
coordinates 
$a_{W}$ and $a_{V}$ in the tuple with  
the relevant $\pi_W(x)$ term 
replaced by $a_{W}$ (and similarly for $V$).

Let
$\tup a_{\nest}$ and $\tup a_{\orth}$ be 
$\kappa_0$--consistent.  Let 
$\tup a\in \prod_{\fontact W\in\mathfrak S}2^{\fontact W}$ be the tuple with $a_W=\rho^U_W$ for each $W\transverse U$ or $U\nest W$, and so that 
$a_W$ otherwise agrees with the relevant coordinate of $\tup a_{\nest}$ or $\tup 
a_{\orth}$, depending on whether $W\nest U$ or $W\orth U$. 
We claim that $\tup a$ is
$\kappa_0$--consistent.  Indeed, all inequalities involving some
$W\transverse U$ or $U\nest W$ are satisfied because of the last three inequalities in Definition~\ref{defn:space_with_distance_formula}.\eqref{item:dfs_transversal}.  Otherwise, if $W\nest
U$ or $W\orth U$, then, if $V\nest W$, we have that $V, W$ are
either both nested into or both orthogonal to $U$, and in both cases
the consistency inequality holds by assumption.

Item \ref{item:consistency_and_realization} of
Definition~\ref{defn:space_with_distance_formula} (realization) combines with the
above discussion to yield a natural coarsely defined map $\phi_U\colon
E_U\times F_U\to \cuco{X}$, where $E_U$ (resp.  $F_U$) is the set of
consistent elements of $\prod_{W\orth U}2^{\fontact W}$ (resp.
$\prod_{W\nest
U}2^{\fontact W}$).  

\begin{rem}\label{rem:product_pieces_cube_cx}
When $\cuco X$ is a cube complex with a factor system $\mathfrak S$, 
the space $F_U$ can be taken to be the subcomplex $U\in\mathfrak S$, we can take 
$E_U$ to be the complex provided by Lemma~\ref{lem:parallel_product}, and 
$\phi_U$ to be the inclusion, which is a cubical isometric embedding.
\end{rem}

We will denote $P_U=\phi_U(E_U\times F_U)$. In this context we define 
a \emph{gate}, $\gate_{P_U}\colon \cuco X\to P_U$, in the following
way.  Given $x\in\cuco X$, we let $\gate_{P_U}(x)$ be a point that coarsely
realizes the coordinates $a_W=\pi_W(x)$ for $W$ nested into or
orthogonal to $U$ and $\pi_W(x)=\rho^U_W$ for $W\transverse U$ or $U\nest W$.

By
Definition~\ref{defn:space_with_distance_formula}.\eqref{item:dfs_distance_formula} (distance formula), for each fixed $U$ 
the subspaces $\phi_U(E_U\times\{b_{\nest}\})$ of $\cuco X$ are
pairwise uniformly quasi-isometric, and similarly for $F_U$.  Abusing
notation slightly, we sometimes regard $E_U$ as a metric space by
identifying it with some $\phi_U(E_U\times\{b_{\nest}\})$, and
similarly for $F_U$.  Moreover, $E_U$ is a hierarchically hyperbolic
space, with respect to $\{V\in\mathfrak S:V\orth U\}\cup\{W\}$, where
$W$ is some non-maximal element of $\mathfrak S$ into which each
element orthogonal to $U$ is nested; and, similarly, $F_U$ is a hierarchically
hyperbolic space with respect to $\mathfrak S_U$.
It is straightforward to check that $\{E_U\}_{U\in\mathfrak S}$ are uniformly hierarchically hyperbolic with respect to the index set described above, with constants depending only on the constants for $\cuco X$.

\begin{rem}[How to induct on complexity]\label{rem:unif_hier}
Observe that if $U$ is not $\nest$--maximal, then any $\nest$--chain in 
$\mathfrak S_U$ is strictly contained in a $\nest$--chain of 
$\mathfrak S$, so that the complexity of $F_U$ is strictly less than 
that of $\cuco X$.  Moreover, by definition, there exists 
$W\in\mathfrak S$, not $\nest$--maximal, such that each $V$ with 
$V\orth U$ satisfies $V\nest W$.  Hence $E_U$ has complexity strictly 
less than that of $\cuco X$.  We will make use of these observations when 
inducting on complexity in the proof of Theorem~\ref{thm:density_point_ascone}.
\end{rem}

A \emph{standard $1$--box} is a hierarchy path $I\to F_{U}$ where  
$U\in\mathfrak S$. 
We then inductively define a \emph{standard $n$--box} to be a map of
the kind $\iota\colon (B\times I\subseteq \reals^{n-1}\times
\reals)\to \cuco X$, with $\iota(b,t)=\phi_U(\iota_1(b),\gamma(t))$
where $\iota_1\colon B\to E_U$ is a standard box and $\gamma\colon
I\to F_U$ is a hierarchy path.  It is straightforward to show
inductively that a standard box is a quasi-isometric embedding.

\subsection{Results}\label{subsec:qb_theorems}
The goal of this section is to establish the following three theorems:

\begin{thm}\label{thm:quasi_box}
 Let $\cuco X$ be a hierarchically hyperbolic space. Then for every $n\in\mathbb N$ and every $K,C,R_0,\epsilon_0$ the following holds.
There exists $R_1$ so that for any ball $B\subseteq \mathbb R^n$ of radius at least $R_1$ and
$f\colon B\to \cuco X$ a $(K,C)$--quasi-Lipschitz map, there is a ball
$B'\subseteq B$ of radius $R'\geq R_0$ such that $f(B')$ lies inside the
$\epsilon_0 R'$--neighborhood of a standard box.
\end{thm}

\begin{remnon}In Theorem~\ref{thm:quasi_box}, we do not require $B,B'$ to be centered at the same point in $\reals^n$.\end{remnon}

\begin{thm}\label{thm:nilpemb}
 Let $\cuco X$ be a hierarchically hyperbolic space. Then for every
simply connected nilpotent Lie group $\mathcal N$, with a
left-invariant Riemannian metric, and every $K,C$ there exists $R$
with the following property.  For every $(K,C)$--quasi-Lipschitz map
$f\colon B\to\cuco X$ from a ball in $\mathcal N$ into $\cuco X$ and
for every $h\in\mathcal N$ we have $\diam(f(B\cap h[\mathcal
N,\mathcal N]))\leq R$.  In particular, any finitely generated
nilpotent group which quasi-isometrically embeds into $\cuco X$ 
is virtually abelian.
\end{thm}

\begin{thm}\label{thm:rank}
 Let $\cuco X$ be a hierarchically hyperbolic space with respect to a
 set $\mathfrak S$.  If there exists a quasi-isometric embedding
 $f\co \mathbb R^n\to\cuco X$ then $n$ is bounded by the maximal
 cardinality of a set of pairwise-orthogonal elements of $\mathfrak
 S$, and in particular by the complexity of $\cuco X$.
\end{thm}

\begin{notation} Let $\omega$ be a non-principal ultrafilter, which 
    will be arbitrary, but fixed for the rest of the paper. For a 
    sequences of spaces $(X_{m})$, we denote their ultralimit $\ulim 
    X_{m}$ by 
$\seq X$, similarly for sequences of maps $\ulim f_{m}=\seq f$, etc. 
For scaling constants, we will have sequences of positive real 
numbers $(r_{m})$; when comparing two such sequences we write 
$(r_{m}) < (r'_{m})$ if $r_{m}<r'_{m}$ for $\omega$--a.e.~$m$. 
Also, we write $(r_{m}) \ll (r'_{m})$ if $\ulim r'_m/r_m=\infty$.
\end{notation}

Each of the above theorems follows from:

\begin{thm}\label{thm:density_point_ascone}
 Let $\{\cuco X_m\}$ be uniformly hierarchically hyperbolic spaces and let $\mathcal N$ be a simply connected nilpotent Lie group endowed with a left-invariant Riemannian metric.
Let $(f_m\colon B_m\to \cuco X_m)$ be a sequence of quasi-Lipschitz maps with uniformly bounded constants, where $B_m\subseteq\mathcal N$ is a ball of radius $r^1_m$. Then:
\begin{enumerate}
 \item \label{item:nilp} For every ultralimit $\seqcuco$ of 
 $(\cuco X_m)$ and corresponding ultralimit map $\seq f\colon \seq B\to \seqcuco$, the map $\seq f$ is constant along ultralimits of cosets of $[\mathcal N,\mathcal N]$.
 \item \label{item:ab} Suppose that $\mathcal N$ is abelian and let
 $(r^0_m)$ satisfy $(r_{m}^0) \ll  (r_{m}^1)$.  Then there
 exist sequences $(p_m)$ and $(r_{m})$ with $(r^0_{m})\ll (r_{m}) \leq (r^1_{m})$ and $B(p_m,r_m)\subseteq B_m$ so that the following holds. Let $\seqcuco$ be the ultralimit of $(\cuco X_m)$ with scaling factor
 $(r_{m})$ and observation point $(f(p_m))$ and let $\seq B$ be the ultralimit of the sequence $B(p_m,r_m)$ with observation point $(p_m)$ and scaling factors $(r_m)$.  Then there exists an ultralimit $\seq F$ of standard boxes so that the ultralimit map $\seq f\colon \seq B\to \seqcuco$ satisfies $\seq f(\seq B)\subseteq \seq F$.
\end{enumerate}
\end{thm}

We first deduce the other theorems from Theorem~\ref{thm:density_point_ascone}.

\begin{proof}[Proof of Theorem \ref{thm:quasi_box}.] Fix $n,K,C,\epsilon_0,R_0>0$. If the statement was false then there would exist a sequence of maps $f_m\colon B_m\to\cuco X$ where
\begin{itemize}
 \item each $B_m$ is a ball of radius $r^1_m$ about $0\in\reals^n$, each $f_m$ is $(K,C)$--quasi-Lipschitz and $r^1_m\to\infty$,
 \item no ball $B'\subseteq B_m$ of radius $R'\geq R_0$ is so that $f_m(B')$ is contained in the $\epsilon_0R'$--neighborhood of a standard box.
\end{itemize}

In the notation of Theorem~\ref{thm:density_point_ascone}.\eqref{item:ab}, if
$(p_m)$ is the observation point of $\seq B$, then it is readily
checked that, $\omega$--a.e., we have $f_m(B(p_m,r_m))\subseteq
\neb_{\epsilon_0 r_m}(F_m)$, where $\seq r=(r_m)$, $\seq F=(F_m)$.  By the proof of Theorem~\ref{thm:density_point_ascone}.\eqref{item:ab}, we have that $\dist_{\reals^n}(p_m,0)\leq r^1_m/2$.  Letting $s_m=\min\{r_m,r^1_m/2\}$, and recalling that $\ulim r^1_m=\ulim r^m=+\infty$, we have $s_m\geq R_0$ $\omega$--a.e.; thus for sufficiently large $m$, taking 
$B'=B(p_m,s_m)$ contradicts the second item above.
\end{proof}

\begin{proof}[Proof of Theorem \ref{thm:nilpemb}.]
This follows directly from 
Theorem~\ref{thm:density_point_ascone}.\eqref{item:nilp}. The 
statement about nilpotent groups follows from the well-known fact that 
every finitely generated nilpotent group is virtually a lattice in a 
simply-connected nilpotent Lie group. \cite{Malcev-completion}
\end{proof}

\begin{proof}[Proof of Theorem \ref{thm:rank}.]  Using
Theorem~\ref{thm:density_point_ascone}.\eqref{item:ab}, we see that
some rescaled ultralimit $\seq f$ of $f$ maps an asymptotic cone of $\mathbb
R^n$, which is itself a copy of $\mathbb R^n$, into a rescaled ultralimit of
standard boxes, whence the claim follows.
\end{proof}

\subsection{Proof of Theorem~\ref{thm:density_point_ascone}}\label{subsec:proof_of_last_prop}
We now complete the proofs of Theorems~\ref{thm:quasi_box}, 
\ref{thm:nilpemb}, and~\ref{thm:rank} by proving
Theorem~\ref{thm:density_point_ascone}.  Throughout, $\cuco X$ is a
hierarchically hyperbolic space; below we use the notation from
Definition~\ref{defn:space_with_distance_formula}.

Fix $W\in\mathfrak S$ and let $\gamma\colon I\rightarrow F_W$ be a 
hierarchy path:   
we  now define the ``projection'' from $\cuco X$ to $\gamma$. Roughly speaking, we map $x\in \cuco X$ to any point in $\gamma$ that minimizes the $\fontact
W$-distance from $x$. 
Let
$\beta=\pi_W(\gamma(I))$ and let $c_\beta\colon \fontact
W\rightarrow\beta$ be closest-point projection.  Define a map
$q_\gamma\colon \cuco X\rightarrow \gamma$ as follows: for each
$z\in\beta$, let $t(z)$ be an arbitrary point of $\gamma$ so that
$\pi_W(t(z))=z$, and then, for each $x\in\cuco X$, let
$q_\gamma(x)=t(c_\beta(\pi_W(x)))$.

Let $P_W$ denote the image in $\cuco X$ of the restriction of $\phi_W$ to 
$E_W \times (\image\gamma)$.

Lemmas~\ref{lem:loc_const_q} and~\ref{lem:qgamma_cone} accomplish in 
the hierarchically hyperbolic setting an 
analog of a result in the mapping class group proven in  
\cite[Theorems~3.4 and~3.5]{BehrstockMinsky:dimension_rank}; note that in 
\cite{BehrstockMinsky:dimension_rank} the results are formulated in 
the asymptotic cone, rather than in the original space.

\begin{lem}\label{lem:loc_const_q}
Let $(\cuco X,\dist)$ be a hierarchically hyperbolic space.  There
exists a constant $\ell\geq 1$ depending only on the constants from
Definition~\ref{defn:space_with_distance_formula} so that the
following holds.  Suppose that $\gamma\colon I\rightarrow F_W$ is a 
hierarchy path connecting $x$ to $y$  
and let $M=\sup_{U\in\mathfrak S_{W}-\{W\}}\sigma_{U,s_0}(x,y)$.  Let $z\in\cuco X$ and let $R=\dist_{\cuco X}(z,P_W)$.  Then
 $$\diam_{\cuco X}(q_\gamma(B(z,R/\ell)))\leq \ell M+\ell.$$  
\end{lem}

\begin{proof}
For a suitable constant $\ell$, we will prove the following: given $v,w\in\cuco
X$, if $\dist (q_\gamma(v),q_\gamma(w))>\ell M+\ell$, then any
hierarchy path $\alpha$ from $v$ to $w$ intersects the $\ell
M$--neighborhood of $P_W$; then the the lemma will follow by increasing
$\ell$.
 
 First, fix a hierarchy path $\alpha$ from $v$ to $w$.  For $\ell$
 large enough, the $\delta$--hyperbolicity of $\fontact W$ and the
 fact that $\alpha$ is a hierarchy path guarantee the existence of
 points $t_1,t,t_2$, appearing in this order along $\alpha$, so that: 
 \begin{enumerate}
  \item $\dist_{\fontact W}(t,\gamma)\leq 2\delta+2D$,
  \item $\dist_{\fontact W}(t, c_{\beta}(\pi_{W}(v))), d_{\fontact 
  W}(t, c_{\beta}(\pi_{W}(w)))\geq 100\delta+100D$,
  \item $\dist_{\fontact W}(t,t_i)>10\kappa_0$ for $i\in\{1,2\}$.
\end{enumerate}

Using 
Definition~\ref{defn:space_with_distance_formula}.\eqref{item:consistency_and_realization} (Realization) 
we pick any $t'\in P_W$ whose $\fontact U$--coordinate is
$\kappa$--close to that of $c_{\beta}(t)$ whenever $U\nest W$ and
whose $\fontact U$--coordinate is $\kappa$--close to that of $t$
whenever $U\orth W$.

\begin{figure}[h]
 \includegraphics{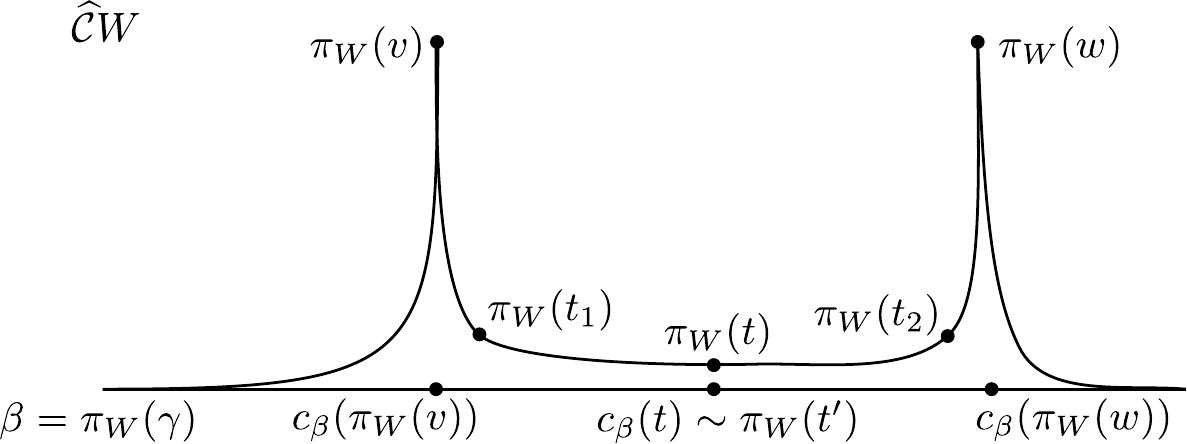}
\end{figure}

By choice of $t'$ and taking the threshold in the distance formula 
larger than $\kappa$ we have that  
any term in the distance formula contributing to $\dist_{\cuco X}(t,t')$ comes from some $U\in\mathfrak S$ which is either nested in
$W$ or transverse to it.  We will first argue that the latter terms
are uniformly bounded.  This is because if
$U\transverse W$ contributes to the distance formula with any
threshold $\geq\kappa_0+\kappa$, we have $d_{\fontact
U}(\rho^W_U,\pi_{U}(t))> \kappa_0$ and hence $d_{\fontact
W}(\rho^U_W,\pi_{ W}(t))\leq \kappa_0$ by the first consistency
inequality (see Definition~\ref{defn:space_with_distance_formula}.\eqref{item:dfs_transversal}).  But then we also have $d_{\fontact W}(\rho^U_W,\pi_{
W}(t_i))> \kappa_0$ and hence again by the first consistency 
condition we have $d_{\fontact U}(\rho^W_U,\pi_{U}(t_i))\leq
\kappa_0$.  However, since $\pi_W\circ\alpha$ is an unparameterized
$(D,D)$--quasi-geodesic and $t$ lies between $t_1$ and $t_2$, there is
a bound on $d_{\fontact U}(\rho^W_U, \pi_{U}(t))$ in terms of $\kappa_0$, $D$,
$\delta$ only; this is the uniform bound we wanted.  From now on, we assume that the thresholds used in the
distance formula exceed this bound.  We now know that
$\sigma_{W,s}(t,t')$ coarsely bounds $\dist_{\cuco X}(t,t')$ for any
sufficiently large threshold $s$.

By 
Definition~\ref{defn:space_with_distance_formula}.\eqref{item:dfs_large_link_lemma} 
(Large Link Lemma), any term that contributes to $\sigma_{W,s_0}(t,t')$ is either $W$ or an element of $\mathfrak S_{T_i}$, for at most $\lambda(2\delta+2D+1)$ elements $T_i\in\mathfrak S_W-\{W\}$, each within distance $\lambda(2\delta+2D+1)$ from $\pi_{W}(t)$, where $\lambda\geq1$ is a uniform constant.
 
 By
 Definition~\ref{defn:space_with_distance_formula}.\eqref{item:dfs:bounded_geodesic_image}
 (Bounded Geodesic Image), applied to geodesics joining $\pi_W(v)$ to
 $\pi_W(x)$ and from $\pi_W(w)$ to $\pi_W(y)$ (or vice versa), we have
 $d_{\fontact F}(v,x),d_{\fontact F}(w,y)\leq B$ for any $F\in
 \mathfrak S_{T_i}$.  Since $\pi_F(t)$ is $D$--close to
 a geodesic from $\pi_F(v)$ to $\pi_F(w)$ for $F\in\mathfrak S_{T_i}$ and
 similarly for $\pi_F(t')$ with respect to $\pi_F(v)$ and $\pi_F(w)$, we 
 thus have
 $$\sum_{F\in\mathfrak S_{T_i}} \ignore{d_{\fontact F}(t,t')}{s_0+2D+2B}\leq R \sum_{F\in\mathfrak S_{T_i}} \ignore{d_{\fontact F}(x,y)}{s_0}$$
 where $R=(s_0+2D+2B)/s_0$. Hence
 $$\sigma_{W,s_0+2D+2B}(t,t')\leq \lambda(2\delta+2D+1)+\lambda(2\delta+2D+1) 
 R\max_i\left\{\sum_{F\in\mathfrak S_{T_i}}\ignore{d_{\fontact F}(x,y)}{s_0}\right\}\leq \ell'M+\ell'.$$
 We hence get $\dist_{\cuco X}(t,t')\leq \ell M$, as required.
\end{proof}

\begin{lem}\label{lem:qgamma_cone}
 Let $\{\cuco X_m\}$ be uniformly hierarchically hyperbolic spaces,
 with respect to $\{\mathfrak S_m\}$, and let $\seqcuco$ be an
 ultralimit of $(\cuco X_m)$.  Also, let $\seq x$, $\seq y$ be
 distinct points of $\seqcuco$.  Then there exists an ultralimit
 $F_{\subseq U}$ of the seqeunce  $(F_{U_m})$, an ultralimit $\seq
 \gamma$ of hierarchy paths $\gamma_m$, contained in $F_{\subseq U}$
 and connecting distinct points $\seq x', \seq y'$, and a Lipschitz
 map $\seq q_{\subseq \gamma}\colon \seqcuco\to \seq\gamma$ such that
\begin{enumerate}
 \item \label{item:qgc_1}$\seq q_{\subseq\gamma}(\seq x)=\seq x'$, $\seq q_{\subseq\gamma}(\seq y)=\seq y'$;
 \item \label{item:qgc_2}$\seq q_{\subseq\gamma}$ restricted to $\seq 
 P_{\subseq U}$ is the projection on the first factor, where $\seq 
 P_{\subseq U}\cong \seq E_{\subseq U}\times \seq\gamma$ is an ultralimit of $(P_{U_m})$;
 \item \label{item:qgc_3}$\seq q_{\subseq\gamma}$ is locally constant outside $\seq P_{\subseq U}$.
\end{enumerate}
\end{lem}

\begin{proof}
 We first claim that there exist $U_m\in \mathfrak S_m$ so that
\begin{enumerate}
 \item $\ulim\frac{1}{s_m} \sigma_{U_m,s_0}(x_m,y_m)>0$,
 \item for any $U'_m\in\mathfrak S_{U_m}-\{U_m\}$, we have
 $\ulim\frac{1}{s_m} \sigma_{U'_m,s_0}(x_m,y_m)=0$.
\end{enumerate}

In fact, since $\seq x\neq \seq y$, there exists some $(\fontact
U_m)$ satisfying the first property, by the distance formula and the
existence of a $\nest$--maximal element.  Also, if $(\fontact U_m)$
satisfies the first property but not the second one, then by
definition we have a sequence $(\fontact V_m)$ satisfying the first
property and so that each $V_m$ is $\omega$--a.e.\ properly nested
into $U_m$.  By
Definition~\ref{defn:space_with_distance_formula}.\eqref{item:dfs_complexity},
in finitely many steps we find $(U_m)$ with the desired property.

Now let $\gamma_m$ be a hierarchy path in $F_{U_m}$ connecting
$x'_m=\gate_{P_{U_m}}(x_m)$ to $y'_m=\gate_{P_{U_m}}(y_m)$.  We can
define $\seq q_{\subseq\gamma}$ to be the ultralimit of the maps
$q_{\gamma_m}$ as in Lemma~\ref{lem:loc_const_q}, and all properties
are easily verified.
\end{proof}

\begin{proof}[Proof of Theorem~\ref{thm:density_point_ascone}.]
The main task is to prove Theorem~\ref{thm:density_point_ascone}.\eqref{item:ab}. Along the way we will point out how to adapt the first part of the argument to obtain Theorem~\ref{thm:density_point_ascone}.\eqref{item:nilp}.

We will prove the proposition by induction on complexity, the base
case being that of complexity $0$, where the distance formula implies
that every ultralimit of $(\cuco X_m)$ is a point.

Consider the ultralimit $\seqcuco^1$ of $(\cuco
X_m)$ with scaling factor $(r^1_{m})$ and an ultralimit map $\seq
f^1\colon \seq B^1\to \seqcuco^{1}$.  (For 
\ref{thm:density_point_ascone}.\eqref{item:nilp}, we
don't restrict the choice of scaling factors used in forming 
$\seqcuco^{1}$.)

If $\seq f^1(\seq B^1)$ is a single point, then the
conclusion is immediate.  Hence, consider $\seq x\neq \seq y$ in
$\seq f^1(\seq B^1)$.  Let $\seq \gamma$, $\seq U$, $\seq x', \seq y'$
and $\seq q_{\subseq \gamma}\colon \seqcuco^1 \to \seq\gamma$ be as in
Lemma \ref{lem:qgamma_cone}.  In the situation of
\ref{thm:density_point_ascone}.\eqref{item:nilp}, towards a 
contradiction, we pick $\seq x\neq \seq y$ in the same ultralimit of cosets
of $[\mathcal N,\mathcal N]$, which we can do if $\mathcal N$ is not 
abelian.  It follows by Pansu's Differentiability
Theorem~\cite[Theorem~2]{Pansu:C_C}, 
which applies because the asymptotic cones of
$\mathcal N$ are Carnot groups \cite{PansuConesNilp}, that every Lipschitz map from
$\seq B^1$ to $\mathbb R$ is constant along cosets of $[\mathcal
N,\mathcal N]$; this contradicts the properties of $\seq q_{\subseq
\gamma}$ established in Lemma~\ref{lem:qgamma_cone} when 
$\mathcal N$ is not abelian.  
Theorem~\ref{thm:density_point_ascone}.\eqref{item:nilp} is hence
proven, and from now on we focus exclusively on
Theorem~\ref{thm:density_point_ascone}.\eqref{item:ab}.  By
Rademacher's Theorem there exists $(p_{m})\in (B^1_{m})$ so that 
$\seq g_{\subseq
\gamma}=\seq q_{\subseq \gamma}\circ \seq f^1$ is differentiable at $\seq
p$ and the differential is nonzero and $\dist_{\reals^n}(p_m,0)\leq r^1_m/2$.  In particular $B(p_m,r^1_m/2)\subseteq B_m$.  Abusing notation slightly, we
are regarding $\seq \gamma$ as an interval in $\mathbb
R$.\\

\renewcommand{\qedsymbol}{$\blacksquare$}

\noindent{\bf Claim~1.} For every $\epsilon>0$ there exists
$l_\epsilon>0$ so that for any $l\leq l_\epsilon$ we have that $\seq
f^1(B(\seq p,l))$ is contained in the $\epsilon l$--neighborhood of
$\seq P_{\subseq U}$.

\begin{proof}[Proof of Claim~1.]  We identify a neighborhood of $\seq p$
with a neighborhood of $0$ in $\mathbb R^n$.
    
We know that there exists a linear function $A$ so that 
for any $v_{1}$ and 
$v_{2}$ in this neighborhood we have $|\seq g_{\subseq
\gamma}(v_1)-\seq g_{\subseq
\gamma}(v_2)|=|A(v_2-v_1)|+o(\max\{\|v_1\|,\|v_2\|\})$. Moreover, 
consider $v$ of norm~$1$ and set $Av=\theta> 0$.  Given   
any $v_1$ with 
$d(\seq f^1(v_1), \seq E_{\subseq U}\times\seq
\gamma)\geq \epsilon \|v_1\|$, 
then, since $\seq f^1$ is $K$--Lipschitz and
$\seq q_{\subseq \gamma}$ is constant outside $\seq E_{\subseq U}\times
\seq\gamma$, we have $\seq g_{\subseq \gamma}(v_1)=\seq g_{\subseq
\gamma}(v_2)$, where $v_2=v_1+\frac{\epsilon \|v_1\|}{K} v$.  Hence
$$0=|\seq g_{\subseq \gamma}(v_1)-\seq g_{\subseq \gamma}(v_2)|=
\frac{\epsilon\theta \|v_1\|}{K} + o((1+\epsilon/K)\|v_1\|)$$
which cannot happen for $\|v_1\|$ small enough depending on $\epsilon$.
\end{proof}

\noindent {\bf Claim~2.} There exists $(r^2_m)$ with $(r^0_m)\leq
(r^2_m)\ll (r^1_m)$ so that for all $(r_m)$ with
$(r^2_m)\ll (r_m)\ll (r^1_m)$ we have $\seq f(\seq
B)\subseteq \seq E_{\subseq U'}\times \seq\gamma'$, where $\seq B$ is the
ultralimit of $(B_m)$ with observation point $(p_m)$ and scaling
factor $(r_m)$.\\

Claim~2 follows from Claim~1 by an application of the 
principle from nonstandard analysis called 
\emph{underspill}, nonetheless, we provide a proof in the interest of 
self-containment. 

\begin{proof}[Proof of Claim~2.]  In view of Claim 1, there exists a
function $\epsilon\colon \mathbb N\to \mathbb R^+$ so that
$\epsilon(k)\to 0$ as $k\to\infty$ and so that for every $k$ the set
$A_k\subseteq \mathbb N$ defined below satisfies $\omega(A_k)=1$, 
$$A_{k}=\{m| f_m(B(p_m,r^1_m/2^k))\subseteq \neb_{\epsilon(k)
r^1_m/2^k}(E_{U_m}\times\gamma_m)\}.$$
Let $k(m)=\max\{k\leq m:m\in \bigcap_{i\leq k} A_i\}$, which by the 
above is
well-defined $\omega$--a.e.\ and satisfies $\ulim 
k(m)=\infty$. By definition of $k(m)$, for $\omega$--a.e.\ $m$ we have
$f_m(B(p_m,r^1_m/2^k))\subseteq \neb_{\epsilon(k)
r^1_m/2^k}(E_{U_m}\times \gamma_m)$ for every $k\leq k(m)$.

Let $(r^2_m)=\max\{(r^0_{m}),(r^1_m/2^{k(m)})\}$. 
We claim that if $(r^2_m) \ll (r_m)\ll (r^1_m)$, then the
result holds for the scaling factor $(r_m)$. 
To do this, we show for each $j\in\mathbb N$ that $\omega(A'_j)=1$, 
where: 
$$A'_j=\{m|f_m(B(p_m, jr_m))\subseteq \neb_{r_m/j}(E_{U_m}\times \gamma_m)\}.$$
Fix $j$. For a given $m$, let $k'(m)$ satisfy  
$r^1_m/2^{k'(m)+1}<jr_m\leq r^1_m/2^{k'(m)}$.  For $\omega$--a.e.\ $m$
we have $k'(m)\leq k(m)$ since $(r_m)\gg (r^1_m/2^{k(m)})$.  In
particular for $\omega$--a.e.\ $m$ we have:
$$f_m(B(p_m, jr_m))\subseteq f_m(B(p_m, r^1_m/2^{k'(m)}))\subseteq
\neb_{\epsilon(k'(m)) r^1_m/2^{k'(m)}}(E_{U_m}\times \gamma_m)
\subseteq \neb_{2\epsilon(k'(m))jr_m}(E_{U_m}\times\gamma_m).$$
Since $(r_m)\ll (r^1_m)$, we have $\ulim k'(m)=\infty$, and
hence $2\epsilon(k'(m))\leq 1/j^2$ for $\omega$--a.e.\ $m$; the claim
follows.
\end{proof}

\renewcommand{\qedsymbol}{$\square$}

By induction on complexity, enabled by Remark~\ref{rem:unif_hier}, we
know that for some $(r_m)\gg (r^2_m)$ the ultralimit $\seq g$ of
$g_m=\gate_{E_{U_m}}\circ f_m$ maps $\seq B$ into some ultralimit
$\seq F'$ of standard boxes, whence it is readily deduced that $\seq
f(\seq B)$ is contained in the ultralimit $\seq F'\times \seq \gamma'$
of standard boxes.  Note that provided $r_m\leq r^1_m/2$, the ultralimit $\mathbf B$ contains the required ultralimit of $(B(p_m,r_m))$.
\end{proof}

\section{Acylindricity}\label{sec:acyl}
\begin{defn}[Automorphism of a hierarchically hyperbolic space]\label{defn:hieraut}
Let $\cuco X$ be a hierarchically hyperbolic space, with respect to
$\mathfrak S$.  An \emph{automorphism} $g$ of $(\cuco X,\mathfrak S)$
is a bijection $g\colon \mathfrak S\rightarrow\mathfrak S$ and a
collection $\{g_U\colon \fontact U\rightarrow \fontact {(gU)}:U\in\mathfrak S\}$ of
isometries such that for all $U,V\in\mathfrak S$ we have
$g_V(\rho^U_V)=\rho^{gU}_{gV}$ and such that $g$ preserves
$\orth,\nest,$ and $\transverse$.
\end{defn}

By
Definition~\ref{defn:space_with_distance_formula}.\eqref{item:dfs_distance_formula},\eqref{item:consistency_and_realization},
every automorphism $g$ of $\mathfrak S$ induces a quasi-isometry
$\phi_g\colon\cuco X\rightarrow\cuco X$ with uniformly bounded
constants.  Such a quasi-isometry $\phi_g$ can be described as
follows.  For $x\in \cuco X$ and $U\in\mathfrak S$, $\pi_{gU}(\phi_g(x))$
coarsely coincides with $ g_U(\pi_U(x))$.  When it will not introduce
confusion, we will use the notation ``$g$'' for the quasi-isometry
$\phi_g$ as well as for the element $g\in\Aut(\mathfrak S)$.  In the
remainder of this section, we are interested in the situation where
the action of $G$ on $\cuco X$ by uniform quasi-isometries is proper
and cocompact.

Let $S=S(\cuco X,\mathfrak S)$ be the unique $\nest$--maximal element
and consider $G\leq\Aut(\mathfrak S)$.  Since each $g\in G$ preserves
the nesting relation, $G$ fixes $S$, and hence $g_S$ is an isometry of
$\fontact S$ for each $g\in G$.  Note that this action coarsely preserves the
union of the sets in $\pi_S(\cuco X)$.

\begin{rem}\label{rem:s_cubical}
When $\cuco X$ is a cube complex with a factor system, $S=\cuco X$ and $\fontact S$ is the factored contact graph $\fontact\cuco X$.  The union of sets in the image of $\pi_\cuco X\colon \cuco X\rightarrow 2^{\fontact\cuco X^{(0)}}$ is coarsely equal to $\fontact\cuco X$, by the definition of $\pi_\cuco X$.
\end{rem}

\begin{thm}\label{thm:acyl_no_box}
Let the group $G\leq\Aut(\mathfrak S)$ act properly and cocompactly on the hierarchically hyperbolic space $\cuco X$ and let $S$ be the unique $\nest$--maximal element of $\mathfrak S$. Then for all $\epsilon>0$, there exist constants $R,N$ such that there are at most $N$ elements $g\in G$ such that $\dist_{\fontact S}(\pi_S(x),\pi_S(gx))<\epsilon,\dist_{\fontact S}(\pi_S(y),\pi_S(gy))<\epsilon$ whenever $x,y\in\cuco X$ satisfy $\dist_{\fontact S}(\pi_S(x),\pi_S(y))\geq R$.
\end{thm}

Before proving Theorem~\ref{thm:acyl_no_box}, we note two corollaries.  Recall that the action of the group $G$ on the metric space $(M,\dist)$ is \emph{acylindrical} if for all $\epsilon>0$, there exist constants $R,N$ such that there are at most $N$ elements $g\in G$ such that $\dist(x,gx)<\epsilon,\dist(y,gy)<\epsilon$ whenever $x,y\in M$ satisfy $\dist(x,y)\geq R$.  (Note that if $M$ is bounded, then any action on $M$ is automatically acylindrical.  In particular, if $\cuco X$ is a cube complex decomposing as the product of unbounded subcomplexes and $G$ acts geometrically on $\cuco X$, then the action of $G$ on $\contact\cuco X$ is trivially acylindrical.)

\begin{cor}[Acylindrical hyperbolicity]\label{cor:acyl_hyp}
Let $G\leq\Aut(\mathfrak S)$ act properly and cocompactly on the hierarchically hyperbolic space $\cuco X$.  Then $G$ acts acylindrically on a hyperbolic space quasi-isometric to $\mathcal U=\bigcup_{x\in\cuco X}\pi_S(x)$. In particular, if $\mathcal U$ is unbounded and $G$ is not virtually cyclic, then $G$ is acylindrically hyperbolic.
\end{cor}

\begin{proof}
Let $S$ be as in Theorem~\ref{thm:acyl_no_box}.  Let $\mathcal T_0\subseteq\fontact S$ be the union of all sets of the form $\pi_S(x)$ with $x\in\cuco X$, together with all of their $G$--orbits, so that $\mathcal T_0$ is $G$--invariant and coarsely equal to the union of the elements of $\pi_S(\cuco X)$.  Definition~\ref{defn:space_with_distance_formula}.\eqref{item:hierarchy_paths} thus ensures that $\mathcal T_0$ is quasiconvex in the hyperbolic space $\fontact S$, so that we can add geodesics of $\fontact S$ to $\mathcal T_0$ to form a $G$--hyperbolic space $\mathcal T$ that is $G$--equivariantly quasi-isometric to the union of the elements of $\pi_S(\cuco X)$.  The action of $G$ on $\mathcal T$ is thus acylindrical by Theorem~\ref{thm:acyl_no_box}.   
\end{proof}

In the cubical case, acylindricity can be witnessed by the contact graph instead of the factored contact graph.

\begin{cor}\label{cor:cubical_acyl}
Let $G$ act properly and cocompactly on the CAT(0) cube complex $\cuco X$, and suppose that $\cuco X$ contains a $G$--invariant factor system.  Then $G$ acts acylindrically on $\fontact\cuco X$, and hence on $\contact\cuco X$.
\end{cor}

\begin{proof}
For any $G$--invariant factor system $\factorsup$, the action of $G$ on $\fontact\cuco X$ is acylindrical by Theorem~\ref{thm:acyl_no_box} and Remarks~\ref{rem:cube_complex_case} and~\ref{rem:s_cubical}.  Let $\factorsup$ be a factor-system in $\cuco X$ and let $\xi$ be the constant for $\factorsup$ from Definition~\ref{defn:factor_system}.  Let $\factorsup_0$ be the smallest set of convex subcomplexes of $\cuco X$ that contains $\cuco X$ and each subcomplex parallel to a combinatorial hyperplane, and has the property that $\gate_F(F')\in\factorsup_0$ whenever $F,F'\in\factorsup_0$ and $\diam(\gate_F(F'))\geq\xi$.  By definition, $\factorsup_0\subseteq\factorsup$, so $\factorsup_0$ has bounded multiplicity and thus $\factorsup_0$ is a factor system.  The  associated factored contact graph of $\cuco X$, on which $G$ acts acylindrically, is $G$--equivariantly quasi-isometric to $\contact\cuco X$, and the result follows.
\end{proof}

\begin{proof}[Proof of Theorem~\ref{thm:acyl_no_box}]
Fix $\epsilon>0$, and for convenience assume that it is $100$ times larger
than all the constants in Definition~\ref{defn:space_with_distance_formula}.  Let $R_0\geq
1000\epsilon$ and consider $x,y\in\cuco X$ such that $R=\dist_{\fontact S}(\pi_S(x),\pi_S(y))\geq R_0$, and let $\mathfrak H$ be the set of all $g\in G$
such that $\dist_{\fontact S}(\pi_S(gx),\pi_S(x))<\epsilon$ and
$\dist_{\fontact S}(\pi_S(gy),\pi_S(y))<\epsilon$.

We will consider, roughly speaking, the set of all $U\in\mathfrak S$ so that $x,y$ project far away in $U$, the corresponding $\rho^U_S$ is near the middle of a geodesic from $\pi_S(x)$ to $\pi_S(y)$, and $U$ is $\nest$--maximal with these properties. We do so because these $U$ correspond to product regions ``in the middle'' between $x$ and $y$. Formally, let $\mathfrak L_1$ be the set of all $U\in\mathfrak S-\{S\}$ with the following properties:
\begin{enumerate}
 \item $\dist_{\fontact U}(\pi_U(x),\pi_U(y))>\epsilon$;
 \item $|\dist_{\fontact S}(\pi_S(x),\rho^{U}_S)-\frac{R}{2}|\leq 10\epsilon$;
 \item $U$ is not properly nested into any $U'\in\mathfrak S-\{S\}$ with $\dist_{\fontact U'}(\pi_{U'}(x),\pi_{U'}(y))>\epsilon$.
\end{enumerate}

When applying an element $g$ that moves $x,y$ a bounded amount, any $U\in\mathfrak L_1$ gets moved to some $gU$ with similar properties but slightly worse constants. To capture this, we let $\mathfrak L_2$ be the set of all $U\in\mathfrak S-\{S\}$ such that:
\begin{enumerate}
 \item $\dist_{\fontact U}(\pi_U(x),\pi_U(y))>\epsilon/2$;
 \item $|\dist_{\fontact S}(\pi_S(x),\rho^U_S)-\frac{R}{2}|\leq 11\epsilon$;
 \item $U$ is not properly nested into any $U'\in\mathfrak S-\{S\}$ with $\dist_{\fontact U'}(\pi_{U'}(x),\pi_{U'}(y))>2\epsilon$.\\
\end{enumerate}

\noindent\textbf{Bounding $|\mathfrak L_2|$:}  Consider a hierarchy path $\gamma$ from $x$ to $y$. Then there are $x',y'$ on $\gamma$ so that
\begin{itemize}
 \item $\dist_{\fontact S}(\pi_S(x'),\pi_S(y'))\leq 23\epsilon$, and
 \item whenever $U\in\mathfrak S-\{S\}$ is so that $\dist_{\fontact U}(\pi_U(x),\pi_U(y))>\epsilon/2$ and $|\dist_{\fontact S}(\pi_S(x),\rho^U_S)-\frac{R}{2}|\leq 11\epsilon$, then $\dist_{\fontact U}(\pi_U(x'),\pi_U(y'))>s_0$ (recall that $s_0$ is the minimal threshold of the Distance Formula).
\end{itemize}

The existence of $x',y'$ follows since we can choose $x'$ and $y'$ projecting close to points on a $\fontact S$--geodesic from $\pi_S(x)$ to $\pi_S(y)$ that lie on opposite sides of the midpoint, at distance slightly larger than $11\epsilon$ from the midpoint. Bounded Geodesic Image (Definition~\ref{defn:space_with_distance_formula}.\eqref{item:dfs:bounded_geodesic_image}) guarantees that the second condition holds because it ensures that $\pi_U(x),\pi_U(y)$ coarsely coincide with $\pi_U(x'),\pi_U(y')$ (recall that $\epsilon$ is much larger than $s_0$ and all other constants in Definition~\ref{defn:space_with_distance_formula}).

By Definition~\ref{defn:space_with_distance_formula}.\eqref{item:dfs_large_link_lemma} 
(Large Link Lemma), with $W=S$ and $x,x'$ replaced by $x',y'$, each $U\in\mathfrak L_2$ is nested into one of 
at most $23\epsilon\lambda+\lambda$ elements $T$ of $\mathfrak S$. For $p$ as in the Claim below, the number of $U\in\mathfrak L_2$ nested into the same such $T$ is bounded by $p$, for otherwise some $U$ (the $U_i$ in the conclusion of the Claim) would fail to satisfy the third property in the definition of $\mathfrak L_2$. Hence $|\mathfrak L_2|\leq p\lambda(23\epsilon+1)$.\\

\noindent{\bf Claim.} There exists $p$ with the following property. Let $T\in\mathfrak S$, let $x,y\in\cuco X$, and let $\{U_i\}_{i=1}^{p}\subseteq \mathfrak S_T$ be distinct and satisfy $\dist_{\fontact U_i}(\pi_{U_i}(x),\pi_{U_i}(y))\geq \epsilon$. Then there exists $U'\in\mathfrak S_T$ and $i$ so that $U_i\propnest U'$ and $\dist_{\fontact U'}(\pi_{U'}(x),\pi_{U'}(y))>2\epsilon$.

\begin{proof}[Proof of Claim]
Define the level of $Y\in\mathfrak S$ to be the maximal $k$ so that there exists a $\nest$--chain of length $k$ in $\mathfrak S_Y$.
The proof is by induction on the level $k$ of a $\nest$--minimal $T'\in\mathfrak S_T$ into which each $S_i$ is nested. For the base case $k=1$ it suffices to take $p=2$ since in this case there is no pair of distinct $U_1,U_2\in\mathfrak S_{T'}$.
 
Suppose that the statement holds for a given $p(k)$ when the level
of $T'$ is at most $k$.  Suppose further that $|\{U_i\}|\geq p(k+1)$
(where $p(k+1)$ is a constant much larger than $p(k)$ that will be
determined shortly) and there exists a $\nest$--minimal $T'\in\mathfrak
S_T$ of level $k+1$ into which each $U_i$ is nested.  There are two
cases.

If $\dist_{\fontact T'}(\pi_{T'}(x),\pi_{T'}(y))> 2\epsilon$, then we are done (for $p\geq 2$).  If not, then
 Definition~\ref{defn:space_with_distance_formula}.\eqref{item:dfs_large_link_lemma} (Large Link Lemma)
yields $K$ and $T_1,\dots,T_K$, each properly
nested into $T'$ (and hence of level $\leq k$), so that any
given $U_i$ is nested into some $T_j$.  In particular, if $p(k+1)\geq
Kp(k)$, there exists $j$ so that at least $p(k)$ elements of
$\{U_i\}$ are nested into $T_j$.  By the induction hypothesis and Definition~\ref{defn:space_with_distance_formula}.\eqref{item:dfs_complexity} (Finite Complexity), 
we are done. 
\end{proof}

\noindent\textbf{$\mathfrak L_1\neq \emptyset$ case:}  Suppose that there 
exists $U\in\mathfrak L_1$. In this case the idea is that there actually are product regions in the middle between $x,y$, and the action of an element moving $x,y$ not too much permutes the gates into such product regions, so that there cannot be too many such elements.

First of all, our choice of $R_0$ and Definition~\ref{defn:space_with_distance_formula}.\eqref{item:dfs:bounded_geodesic_image} (Bounded Geodesic Image) ensure that for all 
$g\in\mathfrak H$, we have $gU\in\mathfrak L_2$. Indeed, if $U\in\mathfrak L_1$ then $\dist_{\fontact gU}(\pi_{gU}(x),\pi_{gU}(y))\geq \dist_{\fontact gU}(\pi_{gU}(gx),\pi_{gU}(gy))-2B\geq \epsilon-2B> \epsilon/2$, ensuring the first property in the definition of $\mathfrak L_2$. The second property follows from the fact that $\pi_S(x)$ gets moved distance $\leq\epsilon$ by $g$, and the third property holds for $gU$ because otherwise, using Bounded Geodesic Image, we would find a contradiction with the third property of $U$ from the definition of $\mathfrak L_1$.

Fix any 
$g\in\mathfrak H$ and let $P_1=E_U\times F_U, P_2=E_{gU}\times F_{gU}$ be the spaces provided by Section~\ref{subsec:box}. We claim that $\dist_{\cuco X}(g\gate_{P_1}(x), \gate_{P_2}(x))$ is uniformly bounded. From this, properness of the action and the bound on $|\mathfrak L_2|$ yield a bound on $|\mathfrak H|$.

Clearly, $g\gate_{P_1}(x)$ coarsely coincides with
$\gate_{P_2}(gx)$, since coordinates of gates are defined
equivariantly.  Hence we must show that $w=\gate_{P_2}(gx)$
coarsely coincides with $z=\gate_{P_2}(x)$. 

By Definition~\ref{defn:space_with_distance_formula}.\eqref{item:dfs_distance_formula} (Distance Formula), it suffices to show that the projections of $w$ and $z$ coarsely coincide in every $\fontact Y$ for $Y\in\mathfrak S$. By definition of $\gate_{P_2}$, it suffices to consider $Y\in\mathfrak S$ which is either nested into or orthogonal to $gU$. For such $Y$, $\rho^Y_S$ coarsely coincides with $\rho^{gU}_S$ by the final part of Definition~\ref{defn:space_with_distance_formula}.\eqref{item:dfs_transversal} (Transversality and consistency).  Moreover, any geodesic from $x$ to $gx$ stays far from $\rho^Y_S$, so that Definition~\ref{defn:space_with_distance_formula}.\eqref{item:dfs:bounded_geodesic_image} (Bounded Geodesic Image) gives a uniform bound on $\dist_{\fontact Y}(\pi_Y(x),\pi_Y(gx))$, as required.\\

\noindent\textbf{$\mathfrak L_1=\emptyset$ case:}  Suppose now that 
$\mathfrak L_1=\emptyset$. In this case, the idea is that a hierarchy path from $x$ to $y$ does not spend much time in any product region near the middle, and hence it behaves like a geodesic in a hyperbolic space. Fix a hierarchy path $\gamma$ in $\cuco X$ joining $x,y$ and let $p\in\gamma$ satisfy $|\dist_{\fontact S}(\pi_S(p),\pi_S(x))-\frac{R}{2}|\leq\epsilon$.  We will produce a constant $M_3$, depending on the constants from Definition~\ref{defn:space_with_distance_formula} and on $\epsilon$ such that $\dist_{\cuco X}(p,gp)\leq M_3$.  It will then follow that $|\mathfrak H|$ is bounded in view of properness of the action of $G$ on $\cuco X$.

We now bound $\dist_{\cuco X}(p,gp)$ using the distance formula. First, note that $\dist_{\fontact S}(\pi_S(p),\pi_S(gp))\leq 10\delta+\epsilon$ (where $\delta$ is the hyperbolicity constant of $\fontact S$).

If $U\in\mathfrak S$ contributes to the sum $\sigma_{\cuco X,s_0}(p,gp)$ with threshold $s_0$, then, given the bound on $\dist_{\fontact S}(p,gp)$ and Bounded Geodesic Image, $\dist_{\fontact S}(\pi_S(p),\rho^U_S)\leq 2\epsilon$.
Fix now any $U$ satisfying $\dist_{\fontact S}(\pi_S(p),\rho^U_S)\leq 2\epsilon$. Our goal is now to bound $\dist_{\fontact U}(\pi_U(p),\pi_U(gp))$ uniformly. It follows from the assumption that $\mathfrak L_1=\emptyset$ that $d_{\fontact U}(\pi_U(x),\pi_U(y))\leq 3\epsilon$, and also $d_{\fontact U}(\pi_U(gx),\pi_U(gy))=d_{\fontact g^{-1}U}(\pi_{g^{-1}U}(x),\pi_{g^{-1}U}(y))\leq 3\epsilon$.  Indeed, if, say, $U$ satisfied $d_{\fontact U}(\pi_U(x),\pi_U(y))>3\epsilon$, then either $U$ or some other $U'$ with $d_{\fontact U'}(\pi_{U'}(x),\pi_{U'}(y))>\epsilon$ and $U\nest U'$ would belong to $\mathfrak L_1$, since for $U\nest U'$ and $U'\neq S$, the sets $\rho^U_S$ and $\rho^{U'}_S$ coarsely coincide.

Since $\gamma$ is a hierarchy path, $\dist_{\fontact U}(\pi_U(x),\pi_U(p))$ and $d_{\fontact U}(\pi_U(gx),\pi_U(gp))$ are bounded by $\dist_{\fontact U}(\pi_U(x),\pi_U(y))+2D$ and $d_{\fontact U}(\pi_U(gx),\pi_U(gy))+2D$, respectively, which are both bounded by, say, $4\epsilon$. Hence, $\dist_{\fontact U}(\pi_U(p),\pi_U(gp))\leq 9\epsilon+\dist_{\fontact U}(\pi_U(x),\pi_U(gx))$ (where the diameters of the projection sets are taken care of by the extra $\epsilon$), and the last term is uniformly bounded by Bounded Geodesic Image. We then get the desired bound, concluding the proof.
\end{proof}

\begin{cor}\label{cor:free_subgroups}
Let $G\leq\Aut(\mathfrak S)$ act properly and cocompactly on the hierarchically hyperbolic space $\cuco X$.  Let $g,h\in G$ be hyperbolic on the maximal $S\in\mathfrak S$ and satisfy $g^nh\neq hg^n,h^ng\neq gh^n$ for all $n\neq 0$.  Then there exists $N>0$, depending on $g$ and $h$, such that $\langle g^N,h^N\rangle$ is freely generated by $g^N,h^N$.
\end{cor}

\begin{proof}
By Corollary~\ref{cor:acyl_hyp}, $G$ acts acylindrically on a hyperbolic space, whence the claim follows from~\cite[Proposition~2.4]{Fujiwara:acyl}.
\end{proof}

In particular, when $G$ acts properly and cocompactly on a CAT(0) cube
complex with a factor system, e.g., when $G$ is compact
special, the conclusion of Corollary~\ref{cor:free_subgroups} is
satisfied.  Thus acylindricity can be used to find free subgroups of
groups acting on cube complexes by different means than are used in
the discussion of the Tits
alternative~\cite{CapraceSageev:rank_rigidity,SageevWiseTits}.  The
above corollary recovers Theorem~47 of~\cite{KimKoberda:curve_graph}
about subgroups of right-angled Artin groups generated by powers of
elements acting loxodromically on the extension graph once we observe,
as in~\cite{KimKoberda:curve_graph}, that there is a
quasi-isometry from the extension graph to the contact graph 
(whenever the right-angled Artin group does not have a free $\integers$ factor).

In the case $\cuco X$ is a uniformly locally finite cube complex, not 
necessarily equiped with a factor system, one can
obtain the same conclusion, provided $g,h\in\Aut(\cuco X)$ act loxodromically on
$\contact\cuco X$, by a ping-pong argument.

\bibliographystyle{alpha}
\bibliography{contractible_contact}
\end{document}